\documentclass[12pt]{amsart}
\usepackage{amsmath,amssymb,amsfonts,a4wide, color, mathtools,mathrsfs}
\usepackage{url,tikz,tikz-cd,tikzsymbols}
\usepackage{relsize}
\usepackage{xcolor}

\usepackage{hyperref}
\hypersetup{
    colorlinks=true, 
    linktoc=all,     
   linkcolor=blue,  
}

\def\sideremark#1{\ifvmode\leavevmode\fi\vadjust{\vbox to0pt{\vss
\hbox to 0pt{\hskip\hsize\hskip1em%
\vbox{\hsize2cm\tiny\raggedright\pretolerance10000%
\noindent {\color{red}{#1}}\hfill}\hss}\vbox to8pt{\vfil}\vss}}}%

\theoremstyle{plain}
\newtheorem{propn}{Proposition}[section]
\newtheorem{thm}[propn]{Theorem}
\newtheorem{lemma}[propn]{Lemma}
\newtheorem{cor}[propn]{Corollary}

\newtheorem{claim}[propn]{Claim}

\theoremstyle{definition}
\newtheorem{defn}[propn]{Definition}

\newtheorem{exo}[propn]{Example}

\newtheorem{rem}[propn]{Remark}

\def \thetas{\theta^\sharp}

\def \V{\mathcal{V}}

\def \bbR{\mathbb{R}}

\def \bbC{\mathbb{C}}
\def \bbS{\mathbb{S}^1}

\def \del{\partial}
\numberwithin{equation}{section}

\newcommand{\om}{\omega}

\newcommand{\h}{\mathfrak{h}}

\newcommand{\gh}{\mathfrak{h}}
\newcommand{\cn}{\nabla^c}
\newcommand{\tM}{\widetilde{M}}
\newcommand{\tP}{\widetilde{P}}
\newcommand{\tN}{\widetilde{N}}
\newcommand{\tZ}{\widetilde{Z}}
\newcommand{\tQ}{\widetilde{Q}}
\newcommand{\tn}{\widetilde{\nabla}}

\renewcommand{\tilde}{\widetilde}

\newcommand{\1}{\sqrt{-1}}
\newcommand{\cntrct}                
{\hspace{2pt}\raisebox{1pt}{\text{$\lrcorner$}}\hspace{2pt}}
\newcommand{\arrow}{{\:\longrightarrow\:}}

\DeclareMathOperator{\Aut}{Aut}
\DeclareMathOperator{\aut}{\mathfrak{aut}}
\DeclareMathOperator{\Li}{\mathscr{L}}
\DeclareMathOperator{\D}{D}

\DeclareMathOperator{\Ham}{Ham}
\DeclareMathOperator{\Iso}{Iso} 

\DeclareMathOperator{\tw}{Tw}
\DeclareMathOperator{\T}{T\mkern-2mu} 
\DeclareMathOperator{\di}{d\mkern-1.4mu } 
\def \H{\mathcal{H}}
\def \W{\mathcal{W}}
\def \F{\mathcal{F}} 
\def \L{\mathscr{L}} 
\renewcommand{\span}{\mathrm{span}}
\renewcommand{\phi}{\varphi}

\begin{document}

\bibliographystyle{plain}

\title[Conformal foliations, K\"ahler twists, Weinstein construction]{Conformal foliations, K\"ahler twists and the Weinstein construction}
\subjclass[2010]{53C12, 53C43, 53C55}
\keywords{Intrinsic torsion, conformal and homothetic foliations, complex structure, Hamiltonian action, harmonic morphism, balanced metric, K\"ahler metric, twist, constant scalar curvature}
\date{\today}
\author[Paul-Andi Nagy]{Paul-Andi Nagy}
\author[Liviu Ornea]{Liviu Ornea}
\thanks{The research of Paul-Andi Nagy is supported by the Institute for Basic Science (IBS-R032-D1); during the early stages of this work P.-A.N. was partially supported by the Research Institute of the University of Bucharest (ICUB). 
L.O. was partially supported by a grant of Ministry of Research and Innovation, CNCS - UEFISCDI,
	project number PN-III-P4-ID-PCE-2016-0065, within PNCDI III}
	
\address{Paul-Andi Nagy\\
Center for Complex Geometry, Institute for Basic Science (IBS)\\
55 Expo-ro, Yuseong-gu, 34126 Daejeon, South Korea
}
\email{paulandin@ibs.re.kr}

\address[Liviu Ornea]{University of Bucharest, Faculty of Mathematics and Informatics, 14 Academiei Str.,\newline Bucharest, Romania, {\it and}\newline
	Institute of Mathematics ``Simion Stoilow" of the Romanian Academy, 21,
	Calea Grivitei Street,\\ 010702, Bucharest, Romania}
\email{lornea@fmi.unibuc.ro} \email{liviu.ornea@imar.ro}

\begin{abstract}
We classify both local and global K\"ahler structures admitting totally geodesic homothetic foliations with complex leaves. The main building blocks are  related to Swann's twists and are obtained by applying Weinstein's method of constructing symplectic bundles to  K\"ahler data. As a byproduct we obtain new classes of: holomorphic harmonic morphisms with fibres of arbitrary dimension from compact K\"ahler manifolds; non-K\"ahler balanced metrics conformal to K\"ahler ones (but compatible with different complex structures).  
\end{abstract}

\maketitle 

\tableofcontents
\section{Introduction} \label{intro}
A foliation $\F$ on a Riemannian manifold $(M,g)$ is called conformal if
$$\L_V g = \theta(V )g$$
on $\T \F^{\perp}$ for any vector field $V$ tangent to the leaves where $\theta$ is a 1-form on $M$ that vanishes on $\T \F^{\perp}$. Here $\L_V$ denotes the Lie derivative in direction of $V$. The foliation $\F$ is called homothetic if $\theta$ is closed and globally homothetic when $\theta$ is exact. Introduced in \cite{Vais}, these natural classes of foliations occur in the study of harmonic morphisms and conformal submersions \cite{Wood,bg,B1,Svensson}. 

In this work we are interested in instances when the foliated manifold is K\"ahler and $\F$ has complex leaves. In the special case when $\dim_{\bbR}\F=2$, homothetic foliations relate to additional geometric structures $(M,g,J)$ may carry. Those include K\"ahler structures of Calabi-type \cite{hamf1} as well as ambik\"ahler conformal classes of Riemannian metrics \cite{MMP}.

The first question we address is finding sufficient conditions for a conformal $\F$ to be homothetic or globally homothetic. The extreme dimensional cases are well understood. When $\dim_{\bbR}\F=2$  the local structure of K\"ahler structures carrying homothetic foliations by complex curves has been completely described in \cite{BLMS}; in particular any such foliation is either holomorphic or totally geodesic and Riemannian.  When 
$\mathrm{codim}_{\bbR} \F=2$ and $\F$ is a conformal foliation with complex leaves then $\F$ must be holomorphic by well known algebraic reasons. For higher codimensions we prove  
\begin{thm} \label{thm0-int}
Let $\F$ be a holomorphic and conformal foliation, $\mathrm{codim}_{\bbR} \F \geq 4$, on a K\"ahler manifold $Z$. Then
\begin{itemize}
\item[(i)]  $\F$ is homothetic
\item[(ii)] If $Z$ is compact, the foliation $\F$ is globally homothetic.
\end{itemize}
\end{thm} 
To prove this we first observe that for holomorphic foliations the property of being conformal is equivalently described by an exterior differential system, which easily yields (i), see Proposition \ref{cf1}. For the proof of (ii) we combine Hodge theory for the restriction of the K\"ahler form to the leaves of $\F$ with the global $\partial \bar \partial$-Lemma, see Theorem \ref{glob-loc}.

In view of this result we have attempted to classify homothetic foliations with complex leaves on K\"ahler manifolds. It turns out that such structures unify several themes in K\"ahler geometry as described below.

\subsection{K\"ahler structures of Weinstein type}
In order to obtain examples with leaves of arbitrary dimension and codimension 
a natural assumption is to consider foliations $\F$ with totally geodesic leaves. 
These carry a natural cohomological invariant, the twist class $\tw(\F)$ (see Section \ref{tls} for details), which partly encodes the geometry. We prove
\begin{thm} \label{thm1-in}
Let $Z$ be a compact K\"ahler manifold endowed with a totally geodesic, holomorphic and conformal foliation $\F$ ( $\mathrm{TGHH}$ for short). Assume  the twist class $\tw(\F)$ is integral. Then either $\F$ is Riemannian (and $(g,J)$ is locally a Riemannian product) or the universal cover $\tZ$ with the pulled-back K\"ahler structure is 
\begin{itemize}
	\item[(i)]obtained from the Weinstein construction

	\noindent or satisfies
	
	\item[(ii)]the pull-back of the class $\tw(\F)$ vanishes  and $\tZ$ is the product $\tM \times \tN$ of two simply connected complete K\"ahler manifolds with $\tM$ exact. The K\"ahler structure on $\tZ$ is given by the local Weinstein construction.
\end{itemize}
In particular, if $Z$ is simply connected, then $\F$ is obtained by the Weinstein construction.	
\end{thm}
This result is proved in Subsection \ref{actiunii},  Theorem \ref{global-carac} (for details on the action of the fundamental group $\pi_1(Z)$ on $\tZ$ the reader is referred to Section \ref{actiuni} of the paper). 

The Weinstein construction above pins down some of the K\"ahler aspects in the general construction of symplectic bundles with structure group $G$ described in \cite{W}, see also \cite{McD-S}. This involves K\"ahler starting data regarded as fibre and base. The former is endowed with an isometric and Hamiltonian action of the circle whilst the latter is polarised within the K\"ahler class. The desired Weinstein-type manifold is then an associated bundle, naturally equipped with a K\"ahler structure coming from the base and the fibre, and with an holomorphic and Hamiltonian circle action, see Subsection \ref{wei_con}. The conformal setup in Theorem \ref{thm1-in} forces $G=\bbS$ in (i) and $G=\bbR$ in (ii). 

Next we investigate up to which extent a TGHH foliation with respect to a fixed K\"ahler structure
$(g,J)$ on a compact manifold $Z$ must be unique. By Theorem \ref{thm1-in} we can essentially assume $(Z,g,J)$ of Weinstein type for some fibre $N$ and polarised base $M$. We actually show how different couples $N,M$ as above yet produce the same result. The idea is to take $N,M$ obtained by the Weinstein construction. This requires the explicit computation of the isometry Lie algebra of a Weinstein type structure $Z$ which is performed in Subsection \ref{aut}; in fact we also determine the Lie algebra of holomorphic vector fields. Establishing the relation between the Picard groups of the fibre and $Z$ in Subsection \ref{pic} allows exhibiting an explicit polarisation of the latter from a polarisation on the fibre.

Based on these we iterate the Weinstein construction by taking at each step the resulting Wein\-stein type manifold as a new base, whilst keeping the fibre arbitrary though polarised. This produces examples of K\"ahler structures admitting up to $3$ distinct TGHH foliations. In turn, a sign change in $J$ in  direction of these foliations gives rise to up to 3 mutually commuting integrable complex structures, orthogonal w.r.t. the K\"ahler metric, see Subsection \ref{uni}
\subsubsection{\bf Classification results} \label{ssc}
Theorem \ref{thm1-in} is the first step towards the more general classification result below. 
\begin{thm} \label{thm2-int2}
Let $(Z,g,J)$ be K\"ahler and equipped with a totally geodesic, homothetic foliation $\F$ with complex leaves.
Around any smooth point in $Z \slash \F$ the conformal submersion $Z \to Z\slash \F$ factorises as the product of two submersions $\pi_1 \circ \pi_2$ where 
\begin{itemize}
	\item[(i)] $\pi_1:Z_1 \to Q$ where $Z_1$ is K\"ahler and $\pi_1$ is Riemannian with totally geodesic and complex fibers
	\item[(ii)] $\pi_2:Z \to Z_1$ is conformal, holomorphic with totally geodesic fibres which are locally obtained by the Weinstein construction.
\end{itemize}
\end{thm}
In general, $Q$ is not K\"ahler. Compact examples of foliated K\"ahler manifolds $Z$ as above are obtained using input from the classification of nearly-K\"ahler structures \cite{Na1}. Moreover deforming $(g,J)$ in direction of $\F$ produces non-integrable $\mathcal{G}_1$ structures, see \cite{gh,Fr-Iv,Na} as well as \cite{Strom} for definitions and main properties. These classes of examples have algebraically generic intrinsic torsion.

Theorem \ref{thm2-int2}  is proved in Section \ref{abs-c}. The first step 
consists in constructing an intermediate foliation $\F^1$ which is TGHH and such that $Z\slash \F^1$ is K\"ahler and supports a totally geodesic Riemannian foliation with complex leaves (see Theorem \ref{tg-main}). This is achieved by finding the connection that prolongs the differential system involving the intrinsic torsion tensor of the $\mathcal{G}_1$-structure. It is worth noting this connection is not metric but conformal. The final step consists in using the classification of totally geodesic Riemannian foliation on K\"ahler manifolds from \cite{Na1} and Theorem \ref{thm1-in} in order to recover the original foliation $\F$ together with the desired factorisation.

\subsection{Swann's twist machinery} \label{twmach} 
The twist construction assigns to a manifold $S$ with a circle action a new manifold $Z$ in the following way. Consider a  principal circle bundle $Q \to S$ which admits a principal connection with curvature form $\Omega$ w.r.t. the $\bbS$-action on $S$ is Hamiltonian. This set of data allows lifting  the circle action from $S$ to $Q$; in case this is free the twist $Z$ is the quotient of $Q$ by the lifted action (see \cite{Swann}).   

It turns out that the above Weinstein construction of $Z$ from $N$ and $M$ is a specific instance of twisting $N \times M$, see Proposition \ref{ws-tw}. This allows using the machinery developed by Swann, in particular the algebraic correspondence between invariant metrics or complex structures on $S$ and the same type of object on $Z$.  Via this correspondence it is easy to describe the K\"ahler structure of $Z$, the space of its harmonic forms, and its Picard group. 

More generally we determine in Subsection \ref{l-a} what type of Hermitian structure on $S$ corresponds to a K\"ahler structure on $Z$ under the twist correspondence. 

The twist construction also provides a key idea for proving Theorem \ref{thm1-in}, see Remark \ref{idea}. Having the twist class integral identifies the natural candidate for a twist bundle $Q$. Deforming the initial K\"ahler metric on $Z$ and applying to it the twist correspondence equips $Q$ with a new 
Riemannian metric. We show the latter has reduced holonomy and 
recover the geometry from de Rham's splitting theorem. These considerations do not require an $\bbS$-action on $Z$.

\subsection{Relations to other geometries}
\subsubsection{\bf Riemannian properties} The above construction has several applications to problems with a strong Riemannian flavour.
Indeed, the homothetic and totally geodesic foliation set-up allows the construction of holomorphic harmonic morphisms with arbitrary dimension of the fibre (Subsection \ref{ham}). As far as we know, these are the first such examples between compact K\"ahler manifolds.

Examples of a slightly different nature can be constructed when the  base K\"ahler manifold carries a Riemannian foliation with complex, totally geodesic leaves.  For an appropriate choice of fibre the Weinstein construction yields examples of conformal submersions with totally geodesic, non-holomorphic fibres, see Proposition \ref{fact-ex}.   

\subsubsection{\bf Hermitian properties} 
One can define a new (integrable) complex structure $I$ on the total space by changing the sign of $J$ in direction of the  foliation: $I=-J$ on $\T \F$ and $I=J$ on $\T \F^\perp$. Moreover, families of new  metrics can be defined, which are Hermitian w.r.t to $I$ and can be balanced or K\"ahler, according to the dimension of the foliation and upon different choices of some parameters. In particular,  it is possible to have a K\"ahler metric conformal to a balanced (also known as semi-K\"ahler) metric on the same compact manifold, w.r.t. different complex structures, see Subsection \ref{I-subs}: 

\begin{thm} 
	Let $(Z,g,J)$ be a compact K\"ahler manifold  obtained by the Weinstein construction and $z:Z\to (0,\infty)$ a  momentum map associated to the canonical circle action. Then the Hermitian structure $(z^{-\frac{2m}{m+n-1}}g,I)$ is balanced non-K\"ahler if $n \geq 2$ and K\"ahler if $n=1$.
\end{thm}

\section{Complex homothetic foliations} \label{loc-globs}

\subsection{General observations} \label{lp}
Assume that the K\"ahler manifold $(Z^{2m},g,J), m\geq 2$, is equipped with a 
foliation $\mathcal{F}$ with leaf-tangent distribution $\D_{+}$. Throughout this paper $\mathcal{F}$ is assumed to be {\it{complex}}, that is 
\begin{equation*}
J\D_{+}=\D_{+}.
\end{equation*}
At the purely algebraic level a complex foliation $\mathcal{F}$ induces a $g$-orthogonal splitting  
\begin{equation} \label{sp-F}
\T Z=\D_{+} \oplus \D_{-}
\end{equation}
which certainly satisfies $J\D_{-}=\D_{-}$. The K\"ahler form $\omega=g(J \cdot, \cdot)$ splits accordingly as 
\begin{equation} \label{sp-om}
\omega=\omega_{+}+\omega_{-}.
\end{equation}
Consider the almost Hermitian structure $(g,I)$ on $Z$ where 
\begin{equation}\label{i}
I_{\vert \D_{+}}:=-J_{\vert \D_{+}}, \ I_{\vert \D_{-}}:=J_{\vert \D_{-}}
\end{equation}
together with its canonical Hermitian connection 
\begin{equation} \label{itor}
\nabla^c=\nabla^g+\eta
\end{equation}
where the intrinsic torsion tensor $\eta:=\frac{1}{2}(\nabla^gI)I$ and $\nabla^g$ is the Levi-Civita connection of the metric $g$. The connection $\nabla^c$ is metric and Hermitian, that is $\nabla^cg=\nabla^cI=0$. Because 
$\nabla^gJ=0$ and $IJ=JI$ it follows that 
\begin{equation} \label{com-J}
\eta_UJ=J\eta_U, \ U \in \T Z,
\end{equation}
that is $\eta$ belongs to  $\Lambda^1Z \otimes \mathfrak{u}(\T Z,g,J)$. Here  $\mathfrak{u}(\T Z,g,J)$ denotes skew-symmetric (w.r.t. $g$) endomorphisms of $\T Z$ which commute with $J$; via the metric, $\mathfrak{u}(\T Z,g,J)$ is isomorphic to $\Lambda^{1,1}Z$, according to the convention $F \mapsto g(F \cdot, \cdot)$.

Property \eqref{com-J} has two direct consequences. The first is that $\nabla^c$ preserves the splitting \eqref{sp-F}. Since $\eta_UI+I\eta_U=0, U \in \T Z$ it also follows that 
$$\eta_U \D_{\pm} \subseteq\D_{\mp},\qquad  U \in   \T Z.$$ 
A routine verification which we leave to the reader shows that $\nabla^c$ is the orthogonal projection onto \eqref{sp-F} of the Levi-Civita connection 
$\nabla^g$. Explicitly 
\begin{lemma} 
We have 
\begin{equation} \label{nablac}
\nabla^c:=2\nabla^g+P_{+}\nabla^g P_{+} +P_{-}\nabla^g P_{-},
\end{equation}
where $P_{\pm}: \T Z\rightarrow \D_{\pm}$ are the orthogonal projection maps onto $\D_{\pm}$. 
\end{lemma}
Many geometric features of the foliation $\mathcal{F}$ can be read off the algebraic symmetries of the intrinsic torsion tensor $\eta$. For instance the integrability of $\D_{+}$ makes that the restriction of $\eta$ to $\D_{+} \times \D_{+}$ is symmetric
\begin{equation*}
\eta_{V_1}V_2=\eta_{V_2}V_1
\end{equation*}
in particular 
\begin{equation} \label{hol-T0}
\eta_{JV_1}JV_2=-\eta_{V_1}V_2, \quad \eta_{JV}X=-J\eta_VX, \quad X\in \D_{-}
\end{equation} 
by taking into account that $[\eta_U,J]=0,U \in \T Z$. Since $\eta$ can be identified with the second fundamental form of the leaves of $\mathcal{F}$, the first equality in \eqref{hol-T0} yields the following well known
\begin{propn} \label{min}
Any foliation with complex leaves on a K\"ahler manifold is minimal. 
\end{propn}
Yet another preliminary result we will need is the following 
\begin{lemma} \label{do-gen}
We have 
\begin{equation*}
\di \omega_{+}=0 \ \mbox{on} \ \Lambda^3\D_{+} \oplus (\Lambda^2\D_{+} \wedge \Lambda^1 \D_{-}) \oplus \Lambda^3\D_{-}, \ \di \omega_{+}(X,Y,V)=\omega_{+}(\eta_{X}Y-\eta_YX,V),
\end{equation*}
for all $X,Y\in\D_-$ and $V\in \D_+$.
\end{lemma}
\begin{proof}
The component of $\di\omega_{+}$ in $\Lambda^3\D_{+}$ vanishes since $\omega_{+}=\omega$ on $\D_{+}$ and $\D_{+}$ is integrable. Finally $\di\omega_{+}$ has vanishing projection on $\Lambda^3\D_{-}$ since $\omega_{+}(\D_{-})=0$. Now pick $V_1,V_2 \in \D_{+}$ and expand 
\begin{equation*}
(\di\omega_{+})(V_1,V_2,X)=(\nabla^g_{V_1}\omega_{+})(V_2,X)-(\nabla^g_{V_2}\omega_{+})(V_1,X)+
(\nabla^g_X\omega_{+})(V_1,V_2).
\end{equation*}
The last summand vanishes 
since $\omega=\omega_{+}$ on $\D_{+}$ and $\nabla^g \omega=0$. As $\omega_{+}$ vanishes on $\D_{-}$ it follows that $(\nabla^g_{V_1}\omega_{+})(V_2,X)=-\omega_{+}(V_1,\nabla^g_{V_2}X)=\omega(\nabla^g_{V_1}V_2,X)$. Therefore $\di\omega_{+}(V_1,V_2,X)=\omega([V_1,V_2],X)=0$ 
since $\D_{+}$ is integrable. In other words $\di\omega_{+}$ is fully determined by its component on $\Lambda^2\D_{-} \wedge \Lambda^1\D_{+}$ which is computed as follows. Expanding 
the exterior derivative $\di$ by means of the Lie bracket we get 
$$\di\omega_{+}(X,Y,V)=-\omega_{+}([X,Y],V)=\omega_{+}(\eta_XY-\eta_YX,V)$$
and the claim is proved.
\end{proof}
\subsection{Holomorphic and conformal foliations} \label{holcf}
In the remainder of this section we work out the algebraic type of $\eta$ for the two main classes of complex foliations of interest in this paper, as introduced below.
\begin{defn} \label{def-h}
A complex foliation $\mathcal{F}$ is called holomorphic provided it satisfies $(\Li_VJ)\T Z \subseteq \D_{+}$ for all $V \in \D_{+}$. 
\end{defn}
Having $\F$ holomorphic ensures, by the complex Frobenius theorem, that $\D_{+}$ is locally spanned by holomorphic vector fields. The second class of foliations relevant to this work is introduced by the following   
\begin{defn} \cite{Vais} \label{def-cf}
The foliation $\mathcal{F}$ is called conformal provided that 
\begin{equation} \label{cf}
(\mathscr{L}_Vg)_{\vert \D_{-}}=\theta(V) g_{\vert \D_{-}}
\end{equation}
 for all $V \in \D_{+}$, where 
$\theta \in \Lambda^1Z$ satisfies $\theta(\D_{-})=0$. Moreover $\mathcal{F}$ is called
\begin{itemize}
\item[(i)] homothetic if $\di\theta=0$, 
\item[(ii)] globally homothetic if $\theta$ is exact.
\end{itemize}
\end{defn}
The form $\theta$ above is called the Lee form of the conformal foliation; it is required to vanish on $\D_{-}$ in order to render it unique. When $\theta=0$ we recover the more familiar notion of Riemannian foliation. Conformal foliations, not necessarily defined on K\"ahler manifolds, are conformally invariant under 
$C^{\infty}_{\mathcal{F}}(Z):=\{f \in C^{\infty}(Z):\di f(\D_{-})=0\}$ in the following sense. If $\mathcal{F}$ is conformal w.r.t. $g$, it remains so w.r.t. 
$e^fg$, for all $f \in C^{\infty}_{\mathcal{F}}(Z)$, with Lee form given by $\theta+\di f$.

To explain the local structure of conformal foliations recall that a locally defined function $\lambda:Z \to \bbR$ is called a local 
dilation for $\F$ provided that 
$$\theta=-(\di \ln \lambda^2)_{\D_{+}}$$ where the subscript indicates orthogonal projection onto 
$\D_{+}$ w.r.t. $\Lambda^1Z=\Lambda^1\D_{+} \oplus \Lambda^1\D_{-}$. If $\lambda$ is globally defined, it is called a global dilation. Local dilations do exist and are obtained from basic vector fields, that is vector fields $X \in \D_{-}$ such that $[\D_{+},X] \subseteq \D_{+}$. 
Indeed, if $X$ is basic, equation \eqref{cf} entails that $\lambda=(g(X,X))^{-\frac{1}{2}}$ is a local dilation. Choosing a local 
dilation $\lambda $ has the effect that the foliation $\F$ becomes Riemannian w.r.t. the locally defined metric $\lambda^2g$. As $\mathcal{F}$ is locally equivalent with a submersion $Z\rightarrow Z/\mathcal{F}$ on the (local) leaf space it follows that conformal foliations are in local correspondence with conformal submersions, in the sense of the definition below. 
Note that this correspondence will be used systematically in this paper.
\begin{defn} \label{def-subm}
Let $(Z,g_Z), (M,g_M)$ be Riemannian manifolds. We call a submersion $\pi:Z \to M$ conformal provided that 
$$\pi^{\star}g_M=\lambda^2 g_Z \ \mbox{on} \ \D_{-}$$
for some nowhere vanishing function $\lambda:Z \to \bbR$.
\end{defn}

To complete the dictionary between notions record that above $\D_{-}$ denotes the orthogonal complement w.r.t. $g_Z$ of the vertical distribution $\D_{+}=\ker(\di \pi) \subseteq \T Z$ of the submersion 
$\pi$. In this situation $\D_+$ is a conformal foliation with Lee form $\theta=-(\di \ln \lambda^2)_{\D_{+}}$. The function $\lambda$ above is called the dilation function of the conformal submersion. Such objects will be referred to as $\pi:(Z,g_Z,\lambda) 
\to (M,g_M)$ in shorthand notation. 
%

In the rest of this section we restrict attention to complex foliations and begin with determining the algebraic structure of 
the intrinsic torsion tensor $\eta$.
\begin{lemma} \label{tg-e}
Let $(Z,g,J)$ be a K\"ahler manifold equipped with a complex  foliation $\mathcal{F}$. 
\begin{itemize}
\item[(i)] The foliation $\mathcal{F}$ is conformal with Lee form $\theta$ if and only if, for all $X,Y\in\D_-$ and $V\in\D_+$, we have 
\begin{equation} \label{d1}
\begin{split}
& \eta_XY=\Psi_XY+\frac{1}{2}(g(X,Y)\theta^\sharp+\omega(X,Y)J\theta^\sharp)\\
& \eta_XV=\Psi_XV-\frac{1}{2}(\theta(V)X-\theta(JV)JX)
\end{split}
\end{equation} 
for some tensor field $\Psi:\D_{-} \to \mathfrak{u}(TZ,g,J), X \mapsto \Psi_X$ which  satisfies $\Psi_XY+\Psi_YX=0$ and
$\Psi_X(\D_{\pm}) \subseteq \D_{\mp}$.
\item[(ii)] $\mathcal{F}$ is holomorphic if and only if  $\Psi=0$.
\end{itemize}
\end{lemma}
\begin{proof}
(i) 
Split $\eta_XY=\Psi_XY+Q_XY$ into $J$-anti-invariant, respectively $J$-invariant parts, that is $\Psi_{JX}JY=-\Psi_XY$ and $Q_{JX}JY=
Q_XY$. That $\D_{+}$ is conformal reads $g(\eta_XY+\eta_YX,V)=\theta(V)g(X,Y)$; as the $J$-anti-invariant component of the latter vanishes we get 
$g(\Psi_XY+\Psi_YX,V)=0$, thus $\Psi_XY+\Psi_YX=0$ as claimed. Considering the $J$-invariant respectively the $J$-anti-invariant parts  
in $(X,Y)$ of $\eta_XJY=J\eta_XY$ shows that $\Psi_XJY=J\Psi_XY$ and $Q_XJY=JQ_XY$. Therefore changing $(V,Y) \mapsto (JV,JY)$ in 
$g(Q_XY+Q_YX,V)=\theta(V)g(X,Y)$ yields $g(Q_XY-Q_YX,V)=\theta(JV)g(X,JY)$ since $Q_{JY}X=-Q_Y(JX)=-JQ_YX$. It follows that 
$2g(Q_XY,V)=\theta(V)g(X,Y)+\theta(JV)g(X,JY)$ and the first line in \eqref{d1} is proved. The second line therein follows by orthogonality.
The converse follows from the general formula 
\begin{equation}\label{L_ca_eta}
	\Li_Vg(X,Y)=g(\eta_XY+\eta_YX,V)
\end{equation}
(ii) Combining the general formula $((\Li_VJ)X)_{{\D_{-}}}=\eta_{JX}V-J\eta_XV$ with the parametrisation of $\eta$ in \eqref{d1} leads to 
$((\Li_VJ)X)_{{\D_{-}}}=-2J\Psi_XV$ and the claim follows.
\end{proof}

\begin{rem} \label{codim}
When $\dim_{\bbR}\D_{-}=2$ the space $\{\alpha \in \Lambda^2\D_{-}:J\alpha=-\alpha\}=0$ showing that $\Psi=0$ in \eqref{d1}. Thus 
any complex conformal foliation of real codimension  $2$ is holomorphic. The same algebraic reason leads to the well-known result 
that any holomorphic foliation of real codimension  $2$ on a K\"ahler manifold is automatically conformal.
\end{rem}

We now begin the study of the almost complex structure $I$ defined in \eqref{i}. 

\begin{lemma} \label{nij}
Let $(Z,g,J)$ be a K\"ahler manifold equipped with a complex, conformal foliation $\mathcal{F}$.  The Nijenhuis tensor of  
the almost complex structure $I$
is given by 
\begin{equation} \label{nij-e}
\begin{split}
& N^I(V,W)=0, \quad N^I(V,X)=-4\eta_VX+4\Psi_XV, \quad N^I(X,Y)=-8\Psi_XY.
\end{split}
\end{equation}
\end{lemma}
\begin{proof}Since $\D_{+}$ is integrable and $N^J=0$ we have ${N^I}_{| \D_{+} \times \D_{+}}=0$. 
We compute  
\begin{equation*}
\begin{split}
N^I(V,X)=&[V,X]+[JV,JX]+I(-[JV,X]+[V,JX])\\
=&2(\nabla^g_VX)_{\D_+}-2(\nabla^g_XV)_{\D_-}-2J(\nabla^g_{JX}V)_{\D_-}+2J(\nabla^g_{JV}X)_{\D_+}
\end{split}
\end{equation*}
where to obtain the last line we have expanded the Lie bracket by means of the Levi-Civita connection of $g$. By \eqref{itor} we get further 
$N^I(V,X)=-2\eta_VX+2\eta_XV+2J\eta_{JX}V-2J\eta_{JV}X$ and the desired final result follows from \eqref{hol-T0} and \eqref{d1}. Taking into account that $N^J=0$ we derive 
\begin{equation*}
\begin{split}
N^I(X,Y)=&[X,Y]-[JX,JY]+I([JX,Y]+[X,JY])\\
=&(I-J)([JX,Y]+[X,JY])=-2J([JX,Y]_{\D_+}+[X,JY]_{\D_+}).
\end{split}
\end{equation*}
However $[JX,Y]_{\D_+}+[X,JY]_{\D_+}=-\eta_{JX}Y+\eta_Y(JX)-\eta_X(JY)+\eta_{JY}X=-4\Psi_X(JY)$ by using successively \eqref{d1} and the algebraic 
properties of $\Psi$. The last part of \eqref{nij-e} now follows.
\end{proof}
As a direct consequence we have the following  
\begin{propn} \label{tg}
With the same hypothesis as in Lemma \ref{nij}, 
the almost complex structure $I$ is integrable if and only if $\D_{+}$ is holomorphic and totally geodesic w.r.t the metric $g$.
\end{propn}
\begin{proof}
Follows from \eqref{nij-e} and (ii) in Lemma \ref{tg-e}, since having $\D_{+}$ totally geodesic is equivalent to $\eta_{|_{\D_{+}\times \D_+}}=0$.
\end{proof}
The next Lemma 
is needed to obtain a simple exterior differential characterisation of  holomorphic conformal foliations on K\"ahler manifolds. Its proof is straightforward from Lemma  \ref{tg-e}.
\begin{lemma} \label{do-ggg}
Assume that the foliation $\F$ is complex and conformal with Lee form $\theta$.  Then:
\begin{equation} \label{do+gg}
\di\omega_{+}=-\theta \wedge \omega_{-}+2\Phi
\end{equation}
where $\Phi \in \Lambda^3Z$ is determined from 
\begin{equation*}
\Phi=0 \ \mbox{on} \ \Lambda^3\D_{+} \oplus (\Lambda^2\D_{+} \wedge \Lambda^1 \D_{-}) \oplus \Lambda^3\D_{-}, \ \Phi(V,X,Y)=g(J\Psi_XV,Y).
\end{equation*}
\end{lemma}
Note that $\Phi \in \lambda^3_I:=\{\alpha \in \Lambda^3Z:\alpha(IU_1,IU_2,U_3)=-\alpha(U_1,U_2,U_3)\}$; when $\F$ is totally geodesic this is proportional to the torsion form of the canonical $\mathcal{G}_1$-structure (see below, Definition \ref{g1}). 
\begin{propn} \label{cdiff}
Let $(Z,g,J)$ be a K\"ahler manifold equipped with a holomorphic foliation $\mathcal{F}$. The following are equivalent
\begin{itemize}
\item[(i)] We have 
\begin{equation} \label{do+}
\di\omega_{+}=-\theta \wedge \omega_{-} 
\end{equation}
where $\theta \in \Lambda^1Z$ vanishes on $\D_{-}$
\item[(ii)] The foliation $\mathcal{F}$ is conformal, with Lee form $\theta$.
\end{itemize}
\end{propn}
\begin{proof}

(i)  $\Rightarrow$ (ii). From \eqref{do+} we get $g(\eta_XY-\eta_YX,JV)=\theta(V)\omega_{-}(X,Y)$; because $\F$ is holomorphic we know that $\Psi=0$, i.e.  $\eta_{JX}Y=-J\eta_XY$ 
thus the variable change $(X,V) \mapsto (JX,JV)$ yields $g(\eta_XY+\eta_YX,V)=\theta(JV)g(X,Y)$. It follows that $2g(\eta_XY,V)=\theta(V)g(X,Y)-
\theta(JV)g(JX,Y)$ thus $\F$ is conformal with Lee form $\theta$ by Lemma \ref{tg-e}.\\
(ii) $\Rightarrow$ (i). Since $\F$ is holomorphic we have $\Psi=0$ from Lemma \ref{tg-e} (ii). Hence $\Phi=0$ in Lemma \ref{do-ggg} and the claim follows.
\end{proof}
\begin{rem} \label{str-eqsn}
\begin{itemize}
\item[(i)] Assume that $\F$ is holomorphic and conformal with Lee form $\theta$. Since $\di\omega=0$ equation \eqref{do+} is further equivalent to 
\begin{equation} \label{do-}
\di\omega_{-}=\theta\wedge\omega_{-}.
\end{equation}
Either of \eqref{do+} or \eqref{do-} will be referred to as the structure equations 
of the holomorphic distribution $\D_{+}$. 
\item[(ii)] Equation  \eqref{do+gg} is not  sufficient to ensure that 
$\F$ is conformal, for it does not fully determine the algebraic type of $\eta$. 

\end{itemize}
\end{rem}
\begin{propn} \label{cf1}
Let $\mathcal{F}$ be a holomorphic and conformal foliation on a K\"ahler manifold. If $\mathrm{codim}_{\bbR} \mathcal{F}\geq 4$, then  
$\mathcal{F}$ is homothetic. 
\end{propn}
\begin{proof}
Differentiate in $\di\omega_+=-\theta\wedge\omega_-$ and use the second equation in \eqref{do+} to obtain 
\begin{equation}\label{do--}
\di\theta\wedge\omega_{-}=0.
\end{equation}
As $\Lambda^2Z=\Lambda^2\D_{+} \oplus (\Lambda^1\D_{+} \wedge \Lambda^1\D_{-}) 
\oplus \Lambda^2\D_{-}$ it follows that $\di \theta$ has vanishing component 
on $\Lambda^2\D_{+}$. If $V \in \D_{+}$, equation \eqref{do--} ensures that 
$(V \cntrct  \di \theta) \wedge \omega_{-}=0$. 
But $V \cntrct \di \theta \in \Lambda^1\D_{-}$ and the assumption on the codimension of the foliation implies injectivity of the exterior multiplication with $\omega_-$ on $\Lambda^1\D_{-}$, showing that $\di \theta \in \Lambda^2\D_{-}$. But $\theta \wedge \om_{+}^n=0$ where 
$2n=\dim_{\bbR}\D_{+}$; differentiating and using again \eqref{do+} yields $\di \theta \wedge \om_{+}^n=0$. This entails the vanishing of $\di \theta$, since the latter belongs to $\Lambda^2\D_{-}$.
\end{proof}

\begin{rem} \label{ccurv-1}
If $\mathcal{F}$ is a homothetic foliation by complex curves, i.e. 
$J\D_{+}=\D_{+}$ and $\dim_{\bbR}\D_{+}=2$, much stronger results are available. For, $\mathcal{F}$ is then either holomorphic 
or totally geodesic and Riemannian according to \cite[Proposition 2.1]{BLMS}.
\end{rem}

\subsection{Global aspects} \label{loc-glob}
We investigate the extent up to which homothetic foliations on a compact K\"ahler manifold $(Z,J,g)$ must be globally homothetic. While this is based on the $\partial \overline{\partial}$-lemma for K\"ahler manifolds some preliminaries are required. We denote with $L$ the exterior 
multiplication with $\omega$ and with $L^{\star}$ its adjoint w.r.t. the metric $g$. Explicitly 
\begin{equation*} 
L^{\star}\alpha =\frac{1}{2}\sum \limits_{i} Je_i \cntrct  e_i \cntrct  \alpha
\end{equation*} 
for $\alpha \in \Lambda^{\star}Z$, where $\{e_i\}$ is some local $g$-orthonormal basis in $\T Z$. By straightforward algebraic computation it follows that 
\begin{equation} \label{l-prod}
L^{\star}(\gamma \wedge \alpha)=\gamma \wedge L^{\star}\alpha+J\gamma^{\sharp} \cntrct  \alpha
\end{equation}
with $\gamma \in \Lambda^1Z$ and $\alpha \in \Lambda^{\star}Z$. Below and in the rest of the paper we let the complex structure $J$ act on 
$1$-forms $\alpha \in \Lambda^1Z$ by composition, $J\alpha:=\alpha \circ J$.
\begin{lemma} \label{cdff}
Let $\mathcal{F}$ be conformal, with Lee form $\theta$ and $\dim_{\bbR}\D_{-}=2m$. Then
\begin{equation}\label{cd-1}
\begin{split}
\di^{\star}\omega_{+}&=m J \theta.
\end{split}
\end{equation}
\end{lemma}
\begin{proof}
Recall that on forms of type $(1,1)$ we have the K\"ahler identity $[L^{\star},\di]=J\di^{\star}$. (see e.g. \cite[Lemma 14.5]{Mor})
Therefore
$$J\di^{\star}\omega_+=L^{\star}\di\omega_{+}-\di L^{\star}\omega_{+}=-L^{\star}(\theta \wedge \omega_{-})-\di L^{\star}\omega_{+}
$$
by also using \eqref{do+}. The identity \eqref{l-prod}
yields $L^{\star}(\theta \wedge \omega_{-})=m\theta$ and since $L^{\star}\omega_+=\vert \omega_+ \vert^2=n$ where $\dim_{\bbR}\D_{+}=2n$ the claim is proved.
\end{proof}
\begin{thm} \label{glob-loc}
Let $(Z^{2m},g,J)$ be a compact K\"ahler manifold. A conformal  
and holomorphic foliation of real codimension at least 4 is globally homothetic.
\end{thm}
\begin{proof}
By Proposition \ref{cf1} we know that $\di \theta=0$. Because $(g,J)$ is K\"ahler it follows that $\di(J\theta) \in \Lambda^{1,1}_JZ$. By the global $\partial \overline{\partial}$-lemma we get 
$\di (J\theta)=\di J\di f$, for some function $f \in C^{\infty}Z$, thus 
$J (\theta-\di f)$ is closed. It is also co-exact since 
\begin{equation*}
J\theta-J\di f=\frac{1}{m}\di^{\star}\omega_{+}-\di^{\star}(f\omega).
\end{equation*}
Indeed, $J\theta=\frac{1}{m}\di^{\star}\omega$ by  \eqref{cd-1}, and also $J\di f=\di^\star(f\omega)$ due to the local expression $\di^\star=-\sum_i e_i\cntrct  \nabla_{e_i}^g$. 
Now integration over $Z$ shows that $J\theta-J\di f=0$ and the claim is proved.
\end{proof}
Without any assumption on the codimension of the foliation, the same argument yields
\begin{cor} \label{cor-h}
An homothetic 
and holomorphic foliation on a compact K\"ahler manifold  $(Z^{2m},g,J)$ is globally homothetic.
\end{cor}

The above result is specific to K\"ahler geometry. In the Riemannian context, examples of conformal and non-homothetic foliations were given e.g. in \cite{NS}.
\subsection{Factorisation} \label{red-h}
Part of the main focus in this work is on the classification of conformal foliations with totally geodesic and complex leaves.
Thus let the K\"ahler manifold $(Z^{2m},g,J),m \geq 2$, be equipped with a homothetic foliation $\mathcal{F}$ whose leaves are complex and totally geodesic. As above we denote with $\theta$ the Lee form of $\mathcal{F}$.
To pin down the type of almost Hermitian geometry prompted out by this set-up recall the following 
\begin{defn} \label{g1}
An almost Hermitian structure $(h,I)$ on the manifold $Z$ has type $\mathcal{G}_1$ provided the Nijenhuis tensor of $I$ is totally skew-symmetric with respect to $h$, that is $\Phi:=h(N^I(\cdot, \cdot), \cdot) \in \Lambda^3Z$.
\end{defn}

\begin{propn} \label{g1-f}
The almost Hermitian structure $(h:=\frac{1}{2}g_{\vert \D_{+}}+g_{\vert \D_{-}},I)$ has type $\mathcal{G}_1$. 
\end{propn}
\begin{proof}
Since $\D_{+}$ is totally geodesic we have $\eta_{\D_{+}}=0$ and the claim follows from \eqref{nij-e}.
\end{proof}

\begin{rem}
	Almost Hermitian structures of type $\mathcal{G}_1$ can be caracterised \cite{Fr-Iv} as those admitting an Hermitian connection $D$ with totally skew symmetric torsion. The connection $D$ is unique and explicit. In Gray-Hervella's classification (\cite{gh}), the class $\mathcal{G}_1$ corresponds to the class $\W_1+\W_3+\W_4$, where $\W_1$ is the class of nearly-K\"ahler manifolds, $\W_3$ is the class of almost Hermitian structures with vanishing Lee form, characterised by the vanishing of the Lee form, and $\W_4$ is the class of locally conformally almost K\"ahler manifolds. Proposition \ref{g1-f} can be used to construct new examples of $\mathcal{G}_1$ manifolds (see Section \ref{non-h}). As in our case the Lee form does not vanish (and $I$ is not integrable if $\D_+$ is only complex, and not holomorphic), neither the pure type $\W_3$ nor the locally conformally K\"ahler case can appear. 
\end{rem}

\smallskip

The first objective in this subsection is to close the differential system on $(\Psi,\theta)$(see \eqref{d1} for the definition of $\Psi$) in order to obtain structure results pertaining to the foliation $\mathcal{F}$. One way to proceed consists in using a result of Nagy \cite{Na}[Theorem 3.1] which computes 
the covariant derivative $DN^I$ in terms of the exterior derivative of $d\Phi$. This approach has several disadvantages as $N^I$ will not be parallel w.r.t. to $D$; moreover specific properties of $J\thetas$  such as being proportional to a holomorphic Killing vector field are not accessible in this way.

Thus we proceed from first principles, using essentially the same technique as in \cite{Na}, but for a different connection. Let 
$\tn$ be the linear connection in $\T Z$ given by  
\begin{equation} \label{nat}
\begin{split}
\tn_{U_1}U_2&=\nabla^g_{U_1}U_2+\frac{1}{2}(g(U_1,U_2)\thetas-\theta(U_1)U_2-\theta(U_2)U_1) \\
&+\frac{1}{2}(\omega(U_1,U_2)J\thetas+\theta(JU_1)JU_2+\theta(JU_2)JU_1) \\
&+\Psi_{U_1}U_2.
\end{split}
\end{equation}
where by a slight abuse of notation $\Psi$ has been extended to $\T Z$ via $\Psi_{\D_{+}}=0$. Using \eqref{d1} it is easy to check that 
\begin{equation*}
\tn \D_{\pm} \subseteq \D_{\pm}.
\end{equation*}
Direct computation shows that $\tn$ is the unique linear connection satisfying 
\begin{equation} 
\begin{split}
&\tn g=\theta \otimes g, \ \tn J=0\\
& \widetilde{T}=\omega \otimes J\thetas+\widetilde{\Psi}
\end{split}
\end{equation}
where $\widetilde{T}$ is the torsion tensor of $\tn$ and $\widetilde{\Psi}(U_1,U_2)=\Psi_{U_1}U_2-\Psi_{U_2}U_1$. 

Indicate with $T^c$, respectively $R^c$, the torsion, respectively  the curvature tensor\footnote{For the curvature tensor $R$ of a connection $\nabla$ we use the convention $R(X,Y)=-[\nabla_X,\nabla_Y]+\nabla_{[X,Y]}$.} of $\cn$ and record the general comparison formula 
\begin{equation} \label{comp-c}
R^c(U_1,U_2)=R^g(U_1,U_2)-d^{\cn}\eta(U_1,U_2)+[\eta_{U_1},\eta_{U_2}]-\eta_{T^c(U_1,U_2)},
\end{equation}
for $U_1,U_2 \in \T Z$, which follows from \eqref{itor} by direct computation. Here 
$$\di^{\cn}\!\eta(U_1,U_2):=(\cn_{U_1}\eta)_{U_2}-
(\cn_{U_2}\eta)_{U_1}.$$
The prolongation of the differential system imposed on $(\Psi,\theta)$ by requiring $\F$ to be totally geodesic is completed below.
\begin{lemma} \label{prol0}
Assume that $\F$ is complex, totally geodesic and homothetic, $\di \theta=0$. Then: 
\begin{itemize}
\item[(i)] $\theta(\Psi_XY)=J\theta(\Psi_XY)=0$,
\item[(ii)] $\cn_X \thetas=0$,
\item[(iii)] $\tn_X\Psi=0$.
\end{itemize}
\end{lemma}
\begin{proof}
(i) follows from \eqref{d1} via $0=\di\theta(X_1,X_2)=-\theta[X_1,X_2]$.\\
(ii) follows from $\di\theta(V,X)=0$ and having $\D_{+}$ totally geodesic.\\
(iii) As $\eta_{\D_{-}}\D_{\pm} \subseteq \D_{\mp}$ and $\eta_{\D_{+}}=0$, we get from \eqref{comp-c} that 
\begin{equation*}
R^g(X_1,X_2,X_3,V)=g((\nabla^c_{X_1}\eta)_{X_2}X_3-(\nabla^c_{X_2}\eta)_{X_1}X_3,V),
\end{equation*}
for all $X_1,X_2,X_3\in \D_-$ and $V\in \D_+$. At the same time, using (ii) and \eqref{d1} yields $(\nabla^c_{X_1}\eta)_{X_2}X_3=(\nabla^c_{X_1}\Psi)_{X_2}X_3$. Thus from the algebraic Bianchi identity 
for $R^g$ we get further $(\nabla^c_{X_1}\Psi)_{X_2}X_3-(\nabla^c_{X_2}\Psi)_{X_1}X_3+(\nabla^c_{X_3}\Psi)_{X_1}X_2=0$, by taking into account that $\Psi$ is skew-symmetric on $\D_{-} \times \D_{-}$. Hence the curvature formula above reads  $R^g(X_1,X_2,X_3,V)=-g((\nabla^c_{X_3}\Psi)_{X_1}X_2,V)$; combining 
$R^g(JX_1,JX_2,X_3,V)=R^g(X_1,X_2,X_3,V)$ and $\Psi_{JX_1}JX_2=-\Psi_{JX_1}JX_2$ yields $(\nabla^c_{X_1}\Psi)_{X_2}X_3=0$. This proves the claim by taking into account that $\tn_{X_1}=\cn_{X_1}$. 
\end{proof}
\begin{propn} \label{prol}
Assume that $\F$ is complex, totally geodesic and homothetic, that is $\di\theta=0$. We have
\begin{itemize}
\item[(i)] $\Li_{J\thetas}g=\theta \otimes J\theta+J\theta \otimes  \theta$ 
\item[(ii)] $\tn \Psi=0$.
\end{itemize}
\end{propn} 
\begin{proof}
In what follows, $X,X_i\in \D_-$ and $V, V_i\in \D_+$ for all indices $i$. Since $\T Z=\D_{+} \oplus \D_{-}$ is $\nabla^c$-parallel we have $R^c(V_1,X_1,V_2,X_2)=0$.  Using that $\eta_{|_{\D_{+}}}=0$ and \eqref{comp-c} leads to 
\begin{equation} \label{Rc1}
R^g(V_1,X_1,V_2,X_2)=-g((\nabla^c_{V_1}\eta)_{X_1}X_2,V_2)-g(\eta_{\eta_{X_1}V_1}V_2,X_2).
\end{equation}
From the symmetry in pairs of $R^g$, that is $R^g(V_1,X_1,V_2,X_2)=R^g(V_2,X_2,V_1,X_1)$, it follows that 
\begin{equation} \label{prol1}
g((\nabla^c_{V_1}\eta)_{X_1}X_2,V_2)-g((\nabla^c_{V_2}\eta)_{X_2}X_1,V_1)+g(\eta_{\eta_{X_1}V_1}V_2,X_2)-g(\eta_{\eta_{X_2}V_2}V_1,X_1)=0.
\end{equation}
A purely algebraic computation based on expanding $\eta$ according to \eqref{d1} yields
\begin{equation} \label{prol2}
\begin{split}
g(\eta_{\eta_{{X_1}V_1}}V_2,X_2)=&-g(\Psi_{X_1}V_1,\Psi_{X_2}V_2)\\
&+\frac{1}{2}(\theta(V_1)g(V_2,\Psi_{X_1}X_2)+\theta(V_2)g(V_1,\Psi_{X_1}X_2))\\
&+\frac{1}{2}(\theta(JV_1)g(JV_2,\Psi_{X_1}X_2)-\theta(JV_2)g(JV_1,\Psi_{X_1}X_2))\\
&+\frac{1}{4}(\theta \otimes \theta-J\theta \otimes J \theta)(V_1,V_2)g(X_1,X_2)\\
&-\frac{1}{4}(\theta \otimes J\theta+J\theta \otimes  \theta)(V_1,V_2)g(JX_1,X_2).
\end{split}
\end{equation}
Using \eqref{com-J} and \eqref{d1} the $J$-anti-invariant piece of \eqref{prol1} in the variables $(X_1,X_2)$ reads
\begin{equation} \label{prol3}
\begin{split}
&g((\nabla^c_{V_1}\Psi)_{X_1}X_2,V_2)+g((\nabla^c_{V_2}\Psi)_{X_1}X_2,V_1)\\
+&\theta(V_1)g(V_2,\Psi_{X_1}X_2)+\theta(V_2)g(V_1,\Psi_{X_1}X_2)=0.
\end{split}
\end{equation} 
At the same time, taking the $J$-invariant part in $(X_1,X_2)$ of \eqref{prol1} whilst using again \eqref{d1} and \eqref{prol2} shows that 
\begin{equation*}
\begin{split}
&g(X_1,X_2)\di\theta(V_1,V_2)+\omega(X_1,X_2)\Li_{J\thetas}g(V_1,V_2)\\
=&(\theta \otimes J\theta+J\theta \otimes  \theta)(V_1,V_2)g(JX_1,X_2).
\end{split}
\end{equation*}
Since $\theta$ is exact on the leaves of $\D_+$, we know that $\di\theta$ vanishes on $\Lambda^2\D_{+}$, hence the above equation is equivalent to 
\begin{equation} \label{prol4}
\Li_{J\thetas}g(V_1,V_2)=(\theta \otimes J\theta+J\theta \otimes  \theta)(V_1,V_2).
\end{equation}
To finish the proof of (i) write first
$$\Li_{J\thetas}g(V,X)=g(\nabla^g_VJ\thetas, X)+g(V,\nabla^g_XJ\thetas)=g(V,\nabla^g_XJ\thetas),$$
since $\D_{+}$ is totally geodesic. We now express $\nabla^g$ as $\nabla^c-\eta$ and use $\cn_X(J\thetas)=0$ by (ii) in Lemma \ref{prol0}, and $\eta_X(\D_+)\subset \D_{-}$. This  implies $\Li_{J\thetas}g(V,X)=0$; similarly 
$\Li_{J\thetas}g(X,Y)=0$ by \eqref{d1}, showing that $\Li_{J\thetas}g=\theta \otimes J\theta+J\theta \otimes  \theta$, as claimed in (i).

To prove (ii) combine the algebraic Bianchi identity 
\begin{equation*}
R^g(V_1,V_2,X_1,X_2)=R^g(V_1,X_1,V_2,X_2)-R^g(V_2,X_1,V_1,X_2)
\end{equation*}
 and \eqref{Rc1} to get 
\begin{equation*}
R^g(V_1,V_2,X_1,X_2)=-g((\nabla^c_{V_1}\eta)_{X_1}X_2,V_2)+g((\nabla^c_{V_2}\eta)_{X_1}X_2,V_1)-g(\eta_{\eta_{X_1}V_1}V_2-\eta_{\eta_{X_1}V_2}V_1,X_2).
\end{equation*}
Since on K\"ahler manifolds  $R^g(V_1,V_2,X_1,X_2)=R^g(V_1,V_2,JX_1,JX_2)$, it follows by a short computation based on \eqref{d1} and \eqref{prol2} that 
\begin{equation} \label{prol5}
\begin{split}
&g((\nabla^c_{V_1}\Psi)_{X_1}X_2,V_2)-g((\nabla^c_{V_2}\Psi)_{X_1}X_2,V_1)\\
+&\theta(JV_1)g(JV_2,\Psi_{X_1}X_2)-\theta(JV_2)g(JV_1,\Psi_{X_1}X_2)=0.
\end{split}
\end{equation}
From \eqref{prol3} and \eqref{prol5} we get 
\begin{equation*}
\begin{split}
-2g((\nabla^c_{V_1}\Psi)_{X_1}X_2,V_2)=&\theta(V_1)g(V_2,\Psi_{X_1}X_2)+\theta(V_2)g(V_1,\Psi_{X_1}X_2)\\
+&\theta(JV_1)g(JV_2,\Psi_{X_1}X_2)-\theta(JV_2)g(JV_1,\Psi_{X_1}X_2).
\end{split}
\end{equation*}
Using successively the definition of $\tn$ in \eqref{nat}, the equality  
$\nabla^c_{V_1}=\nabla^g_{V_1}$ 
, an algebraic computation shows that the last displayed equation is  equivalent with $(\tn_{V_1}\Psi)_XY=0$. That $\tn_{V_1} \Psi=0$ follows from having $\Psi_X$ skew-symmetric w.r.t. $g$ and $\tn g=\theta \otimes g$.
\end{proof}
With this information in hands we can describe how $\F$ induces a holomorphic and homothetic foliation with totally geodesic leaves. We shall call such a foliation TGHH.  
Consider the distributions 
\begin{equation*}
\D_{+}^2:=\span \{\Psi_XY\} \subseteq \D_{+}, \quad \D_{-}^2:=\span \{\Psi_XV\} \subseteq \D_{-}.
\end{equation*} 
Both are parallel w.r.t. $\tn$, in particular they must have constant rank over $Z$. Due to $\tn g=\theta \otimes g$, the orthogonal complements $\D_{\pm}^1$ of $\D_{+}^2$ respectively $\D_{-}^2$ in $\D_{+}$ respectively $\D_{-}$ are also $\tn$-parallel. Consider the $J$-invariant distributions 
$$ \D^1:=\D^1_{+} \oplus \D^1_{-}, \quad \D^2:=\D^2_{+} \oplus \D^2_{-}.$$
\begin{thm} \label{tg-main}
Let $(Z,g,J)$ be K\"ahler and equipped with a totally geodesic, homothetic foliation $\F$ with complex leaves. Then
\begin{itemize}
\item[(i)] The $g$-orthogonal splitting $\T Z=\D^1 \oplus \D^2$ defines a $\mathrm{TGHH}$  foliation $\F^1$ with leaves tangent to $\D^1$ and Lee form $\theta$
\item[(ii)] Any integral manifold of $\F^1$ carries a $\mathrm{TGHH}$ foliation with leaves tangent to $\D^1_{+}$,  and Lee form given by the restriction of $\theta$
\item[(iii)] Around any of its smooth points, the quotient $Z \slash \F^1$ is K\"ahler and carries a totally geodesic Riemannian foliation with complex leaves tangent to the projection of $\D^2_+$. 
\end{itemize}
\end{thm}
\begin{proof}
(i) From the definitions we have
\begin{equation*}
\Psi_{\T Z}\D^{1}=0, \ \theta(\D^2)=0.
\end{equation*}
The last equality follows from (i) in  Lemma \ref{prol0}. In particular $\thetas \in \D^1$ and also $\tn_{U_1}U_2-\nabla^g_{U_1}U_2 \in \D^1$ for all $U_1,U_2 \in \D^1$ by using \eqref{nat}. But $\D^1$ is $\tn$-parallel thus $\D^1$ is totally geodesic. When $U_1,U_2 \in \D^2$ equation \eqref{nat} reads 
\begin{equation*}
\tn_{U_1}U_2=\nabla^g_{U_1}U_2+\frac{1}{2}(g(U_1,U_2)\thetas+\omega(U_1,U_2)J\thetas)+\Psi_{U_1}U_2.
\end{equation*} 
Since $\Psi_{\D^2}\D^2 \subseteq \D^2$ by orthogonality and the distribution $\D^2$ is $\tn$-parallel it follows that 
$\nabla^g_{U_1}U_2+\frac{1}{2}(g(U_1,U_2)\thetas+\omega(U_1,U_2)J\thetas) \in \D^2$. As $\thetas \in \D^1$,  Lemma \ref{tg-e} shows that $\D^1$ is holomorphic 
and homothetic, with Lee form $\theta$.\\
(ii) follows by arguments entirely similar to those above by taking into account that $\D^1_{\pm}$ are $\tn$-parallel, $\Psi_{\D^1}\D^1=0$ and 
$\theta(\D^1_{-})=0$.\\
(iii) Locally $Z$ is the total space of a submersion $\pi:Z \to M$ with fibres tangent to $\D^1$. Choose a locally defined function $z>0$ with $\theta=\di \ln z$; because $\D^1$ is homothetic with Lee form $\theta$ the restriction of the $z^{-1}g$ to $\D^2$ projects onto a Riemannian metric $g_M$ on $M$, that is $(z^{-1}g)_{\vert \D^2}=\pi^{\star}g_M$. The holomorphy of $\D^1$ ensures that $J$ projects onto an integrable complex structure $J_M$, orthogonal 
w.r.t. $g_M$. From the structure equation \eqref{do-} it follows that the pair $(g_M,J_M)$ is K\"ahler. To prove the claim it suffices to show that 
$\D^{2}_{+}$ projects onto $M$, tangent to the leaves of a complex Riemannian foliation w.r.t. $(g_M,J_M)$. An orthogonality argument based on 
the definition of $\D_{-}^1$ and having $\Psi:\D_{-} \times \D_{-} \to \D_{+}$ skew symmetric shows that 
\begin{equation*}
\Psi_{\D_{-}^1}\D_{+}=0.
\end{equation*}
It follows that 
\begin{equation} \label{tor1}
\Psi_{\D^1}\D_{+}^2=0.
\end{equation}
The distribution $\D_{+}^2$ is $\widetilde{\nabla}$-parallel. Record that $g(\D^1,\D_{+}^2)=0, \omega(\D^1,\D_{+}^2)=0$ and that $\theta$ respectively 
$J\theta$ vanish on $\D_{+}^2$. From the expression for $\widetilde{\nabla}$ in \eqref{nat} and \eqref{tor1} it follows that 
\begin{equation*}
\nabla^g_{\D^1}\D_{+}^2 \subseteq \D_{+}^2.
\end{equation*}
Similarly, using that $\D^1$ is $\widetilde{\nabla}$-parallel and $\Psi_{\D_{+}^2}\D^1=0$ we get $\nabla^g_{\D_{+}^2}\D^1 \subseteq \D^1 \oplus \D_{+}^2$. These facts lead to 
$[\D^1,\D_{+}^2]_{\D^2} \subseteq \D_{+}^2$
where the subscript indicates orthogonal projection. Thus $\D_{+}^2$ projects onto a distribution on $M$, denoted by  $\D_{+}^M \subseteq \T M$, and hence $\D_{-}^2$ projects onto 
the orthogonal complement  $\D_{-}^M$ of $\D_{+}^M$ in $\T M$. Whenever $U$ is a vector field on $M$ indicate with $\widetilde{U}$ its horizontal lift to $\D^2$. Pick $U_1 \in \D_{+}^M$ and 
$U_2, U_3 \in \D_{-}^M$. Because $\D_{+}$ is homothetic and $\theta$ as well as $dz$ vanish on $\D_{+}$ we get 
\begin{equation*}
0=(\Li_{\widetilde{U}_1}g)(\widetilde{U}_2, \widetilde{U}_3)=z(\Li_{\widetilde{U}_1} \pi^{\star}g_M)(\widetilde{U}_2, \widetilde{U}_3)=z \pi^{\star}((\Li_{U_1}g_M)(U_2,U_3))
\end{equation*} 
by using the well know formula $[\widetilde{U}, \widetilde{V}]_{\D^2}=\widetilde{[U,V]}$ for all $U,V \in \T M$. We have showed that $\D_{+}^M$ induces a complex Riemannian foliation w.r.t. $g_M$.

Pick $U_1,U_2 \in \D_{+}^M$; because $\D_{+}^2$ is $\widetilde{\nabla}$-parallel and $\Psi_{\D_{+}^2}\D_{+}^2=0$ it follows from \eqref{nat} that 
\begin{equation*}
\nabla^g_{\widetilde{U}_1}\widetilde{U}_2+\frac{1}{2}(g(\widetilde{U}_1, \widetilde{U}_2)\thetas+\omega(\widetilde{U}_1, \widetilde{U}_2)J\thetas) \in \D_{+}^2.
\end{equation*}
As $\thetas \in \D_{+}^1$ this leads to $(\nabla^g_{\widetilde{U}_1}\widetilde{U}_2)_{\D^2} \in \D_{+}^2.$ Taking into account that a quick check based on Koszul's formula  shows that $(\nabla^g_{\widetilde{U}_1}\widetilde{U}_2)_{\D^2}=\widetilde{\nabla^{g_M}_{U_1}U_2}$ we end up with 
$\nabla^{g_M}_{U_1}U_2 \in \D_{+}^M$. In other words the latter distribution is totally geodesic and the claim is proved.
\end{proof}
Geometrically this is a local factorisation result for conformal submersions, see section \ref{non-h} for details and examples. Combined with the local classification Theorem for TGHH foliations in Section \ref{abs-c} this result fully gives the local structure of K\"ahler metrics admitting totally geodesic, homothetic foliations. It also allows constructing large classes 
of compact examples, see Section \ref{uni}. Before moving on, we set up some notation and recall some basic material 
on symmetries of K\"ahler metrics.

Let $(N^{2n},g_N,J_N)$ be a compact K\"ahler manifold with K\"ahler form $\om_N=g_N(J_N \cdot ,\cdot)$. Indicate with 
$$\mathfrak{aut}(N,g_N,J_N)=\{X\in \T N : \Li_Xg_N=0 \ \mathrm{and} \ \Li_XJ_N=0\}$$
the Lie algebra of holomorphic Killing vector fields on $N$.  We recall that Killing vector fields are automatically holomorphic, i.e. $\mathfrak{iso}(N,g_N)=\mathfrak{aut}(N,g_N,J_N).$ 
If $X$ belongs to $\mathfrak{aut}(N,g_N,J_N)$, then $\Li_X\omega_N=0$, thus, by Cartan's formula, $\di (X\cntrct\omega_N)=0$; when $X\cntrct\omega_N$ is exact, $X$ is called Hamiltonian. A map $z_X:N \to \bbR$ with $X\cntrct \omega_N=\di z_X$ is called a {momentum map}; it is unique up to addition of a real constant.  Let 
$$\aut_0(N,g_N,J_N):=\{X \in \T N : \  \L_Xg_N=0 \ \mathrm{and} \ [X \cntrct \omega_N]=0 \ \mathrm{in} \ H^1_{dR}N \}$$ 
be the Lie algebra of Hamiltonian Killing vector fields. Note that vector fields $X$ in $\aut_0(N,g_N,J_N)$ are automatically holomorphic, $\L_XJ_N=0$. 

We also recall the important fact that 
$\mathfrak{aut}_0(N,g_N,J_N)$ coincides with the Lie subalgebra of Killing vector fields with zeroes. For further use we denote with $\Aut(N,g_N,J_N)$ the group of holomorphic isometries and with $\Aut^0(N,g_N,J_N)$ its connected component; the closed connected subgroup of $\Aut^0(N,g_N,J_N)$ with Lie algebra $\aut_0(N,g_N,J_N)$ will be denoted by $\Aut_0(N,g_N,J_N)$. 

With the preliminaries thus set, the following immediate consequence of Proposition \ref{prol} is the starting point for obtaining the classification results in Section \ref{abs-c}. The notation used below is the same as in Theorem \ref{tg-main} and its proof.
\begin{propn} \label{detect-k}Let $\F$ be a totally geodesic, homothetic foliation on $(Z,g,J)$ with complex leaves and Lee form $\theta$.
\begin{itemize}
\item[(i)] For any locally defined function $f$ with $\theta=\di f$, the vector field 
$$K:=-e^f J \thetas\in \D_+^1$$ is  holomorphic and Killing, with (local) momentum map $z=e^f>0$
\item[(ii)] If $Z$ is compact, then there exists a map $z:Z \to (0,\infty)$ with $\theta=\di \ln z$. In particular, $K\in\mathfrak{aut}_0(Z,g,J)$ is globally defined. 
\end{itemize}
\end{propn}
\begin{proof}
(i) That $\Li_Kg=0$ follows from (i) in Proposition \ref{prol}. Since $\Li_K \omega=0$ we also have $\Li_KJ=0$.\\
(ii) Because $Z$ is compact and the foliation induced by $\D^1$ is holomorphic and conformal with Lee form $\theta$ the latter is exact by Corollary \ref{cor-h}, thus one can choose a globally defined primitive $f$ for $\theta$. 
\end{proof}
This is the starting point for obtaining the classification results in Section \ref{abs-c}. 

\section{ The K\"ahler geometry of twists} \label{gen-constr}

\subsection{ Hamiltonian forms} \label{l-a}
We begin with the following very general 
\begin{defn} \label{ham-O}
Let $S$ be a connected  manifold and let $\alpha \in \Lambda^2S$ be closed. The space of Hamiltonian vector fields w.r.t. $\alpha$ is 
\begin{equation*}
\Ham(S,\alpha)=\{X: [X\cntrct\alpha]=0 \ \mbox{in} \  H^1_{dR}S \}.
\end{equation*}
Moreover, a circle action with tangent vector field $K_S$ is Hamiltonian with respect to $\alpha$ provided that $K_S \in \Ham(S,\alpha)$. 
\end{defn}
The space of closed $2$-forms with respect to which a given vector field $X$ is Hamiltonian is denoted by 
\begin{equation*}
\Ham^{2}(S, X):=\{\alpha \in \Lambda^2S : \di\alpha=0, \ [X\cntrct\alpha]_{dR}=0\}.
\end{equation*}
A map $z_{\alpha}$ with $X \cntrct \alpha=\di z_{\alpha}$ is called a momentum map for $X$ with respect to the closed form $\alpha$; it is only defined up to addition of constants. 

\begin{rem}\label{fix_z}
	When we need to fix a momentum map $z_\alpha$, and in addition $S$ is compact, we assume that $\int_Sz_\alpha=0$. 	
	The thus normalised momentum map is then $\underline{z}_{\alpha}=z_{\alpha}-\tfrac{\int_S z}{vol(S)}$; this notation will be used 
	throughout the rest of the paper.
\end{rem}
The main fact we wish to prove in this section is that on compact K\"ahler manifolds an isometric Hamiltonian action is automatically $\Omega$-Hamiltonian whenever $\Omega$ is harmonic. The upshot is that such actions can therefore be lifted to the total space of circle bundles according to Proposition \ref{lift} in the next section. 

To that extent we let $(N^{2n},g_N,J_N), n \geq 1$ be a compact K\"ahler manifold with K\"ahler form $\omega_N=g_N(J_N \cdot,\cdot)$. In addition 
to the symmetry algebras defined at the end of the previous section 
it is also useful to consider the complex version of Hamiltonian vector fields as follows. Whenever $\Omega \in \Lambda^{1,1}N$ is closed let 
\begin{equation*}
\Ham^c(N,\Omega):=\{X \in \T N:X \cntrct \Omega=\di a+J_N\di b\}.
\end{equation*}
Equivalently $X \in \Ham^c(N,\Omega)$ if and only if $[X_{01} \cntrct \Omega]=0$ in $H^{1,0}N$. Finally, we indicate with $\H^p(N,g_N)$ the space of harmonic $p$-forms with respect to the metric $g_N$ and prove the following 
\begin{propn} \label{ham-K}
Let $(N,g_N,J_N)$ be a compact K\"ahler manifold and let $X$ belong to $\aut_0(N,g_N,J_N)$. Then 
\begin{itemize}
\item[(i)]  $ \H^2(N,g_N)\subset \Ham^2(N,X)$
\item[(ii)] if $\ \Omega \in \Lambda^{1,1}N$ satisfies $\di \Omega=0$ we have $\aut_0(N,g_N,J_N) \subseteq \Ham^c(N,\Omega)$.
\end{itemize}
\end{propn}  
\begin{proof}
(i) Let $\Omega \in \H^2(N,g_N)$. Because 
$X$ is a holomorphic Killing vector field it is well known that $\L_{X}\Omega=0$. In particular $\di (X \cntrct \Omega)=0$ which allows writing $X \cntrct \Omega=\sigma+\di f$
with $\sigma \in \H^1(N,g_N)$. Choose a momentum map $z_X$ for $X$, that is $X \cntrct  \omega_N=\di z_X$. We compute the $L^2$-norm
\begin{equation*}
\begin{split}
\Vert \sigma \Vert^2_{L^2}=&\langle \sigma, X \cntrct \Omega-\di f \rangle_{L^2}=
\langle \sigma, X \cntrct \Omega\rangle_{L^2}
=\langle X^{\flat} \wedge \sigma, \Omega \rangle_{L^2}=\langle J_NX^{\flat} \wedge J_N\sigma, J_N\Omega \rangle_{L^2}\\
=&-\langle \di z_X \wedge J_N\sigma, J_N\Omega \rangle_{L^2}=-\langle \di\, (z_XJ_N\sigma), J_N\Omega \rangle_{L^2}=0
\end{split}
\end{equation*}
since $J_N\sigma$ is closed and $J_N\Omega$ is co-closed. Therefore $\sigma=0$ and the claim is proved.\\
(ii) The global $\partial \overline{\partial}$-Lemma allows writing $\Omega=\H(\Omega)+\di J_N \di f$ where $\H(\Omega) \in \H^{1,1}(N,g_N)$ and 
$f:N \to \bbR$. By (i) the form $X \cntrct  \H(\Omega)$ is exact; since $X$ is holomorphic Cartan's formula yields 
$X \cntrct (\di J_N \di f)=J_N \di \L_Xf- \di \L_{JX}f$ whence the result.
\end{proof}

\subsection{Lifting circle actions and twists} \label{gen-consn}
From now on assume that $S$ is comes equipped with the following set of data.
\begin{itemize} 
\item[$\bullet$] a circle action $(R^S_{\lambda})_{\lambda \in \bbS}$ with infinitesimal generator $K_S$ which is Hamiltonian with respect to $\Omega \in \Ham^2(S,K_S)$ 
\item[$\bullet$] a principal circle bundle $Q \xrightarrow{\pi_Q} S$ equipped with a principal connection form $\Theta$ in $Q$ with curvature form $\Omega$, that is $-\di \Theta=\pi^{\star}_Q\Omega$ and $\Theta(T_Q)=1$ where $T_Q$ is tangent to the principal circle action on $Q$. 
\end{itemize}
Indicate with 
$\Lambda^{\star}_{\bbS}S$ the space of forms which are invariant under the circle action.
Cartan's formula makes that $\Ham^2(S, K_S)\subset \Lambda^2_{\bbS}S$. Hamiltonian actions can be lifted \cite{Swann}(see Proposition 2.3) to the total space of circle bundles such as $Q$. The proof given therein is purely topological; below we quickly outline a proof when the total space of the circle bundle is compact, mostly to explain the infinitesimal lift of vector fields (see also \cite{MiR}, \cite{bh} for related information). 

Choose a momentum map $z_{\Omega}$ for the curvature form $\Omega$. The infinitesimal lift of $K_S$ to $Q$ (which depends on the choice of $z_\Omega$) is defined by 
\begin{equation}\label{ilift}
\widetilde{K_S}:=
(z_{\Omega} \circ \pi_Q) T_Q+K_S^{\H}
\end{equation}
 where $K_S^{\H}$ denotes 
the lift of $K_S$ to $\H:=\ker(\Theta)$. 
\begin{propn} \label{lift}
When $S$ is compact there exists a choice of momentum map such that the flow of $\widetilde{K_S}$ induces  
a circle action $ \mathbb{S}^1\ni \lambda \mapsto \widetilde{R}_\lambda$ on 
$Q$ satisfying 
\begin{itemize}
\item[(i)] $\pi_Q \circ \widetilde{R}_\lambda=R^S_\lambda \circ \pi_Q$
\item[(ii)] $\widetilde{R}_{\lambda_1} \circ R_{\lambda_2}^Q=R_{\lambda_2}^Q \circ \widetilde{R}_{\lambda_1} $ 
where $\lambda \mapsto R^{Q}_\lambda$ is the principal action on $Q$
\item[(iii)] $\widetilde{R}_\lambda^{\star}\Theta=\Theta$.
\end{itemize}
\end{propn}
\begin{proof}
Consider the infinitesimal lift $\widetilde{K_S}$ with respect to 
some momentum map $z_{\Omega}$. Using Cartan's formula leads to $\L_{\widetilde{K_S}}\Theta=0$. Since we also 
have $[T_Q,\widetilde{K_S}]=0$ and $(\di \pi)\widetilde{K_S}=K_S$ the 
flow $(\psi_t)_{t \in \bbR}$ of $\widetilde{K_S}$ satisfies 
\begin{equation*}
(\psi_t)^{\star}\Theta=\Theta, \quad \psi_t \circ R^Q_\lambda=R^Q_\lambda \circ \psi_t, \quad 
\pi \circ \psi_t=\phi_t \circ \pi
\end{equation*}
where $(\phi_t)_{t \in \bbR}$ is the flow of $K_S$. Because $\phi_{2\pi}=1_S$ 
it follows that $\psi_{2\pi}$ is a gauge transformation preserving $\Theta$, thus 
$\psi_{2\pi}=R_{\lambda_0}^P$ with $\lambda_0 \in \mathbb{S}^1$. The flow $\left(\psi_t \circ R^Q_{\exp(-\1\, tt_0)}\right)_{t \in \bbR}$, where $e^{2\pi \1\, t_0}=\lambda_0$ is $2\pi$-periodic and clearly satisfies 
(i)-(iii); moreover, this is the flow of $\widetilde{K_S}+t_0T_Q$ which is the Hamiltonian lift corresponding to the momentum map 
$z_{\Omega}+t_0$.
\end{proof}
\begin{rem} \label{lift-u} We emphasize that:
\begin{itemize}
\item[(i)] Lifted actions as above are not unique. It is easy to see that any other action 
satisfying (i)-(iii) is of the form $\mathbb{S}^1 \ni \lambda\mapsto \widetilde{R}_\lambda \circ R_{\lambda^n}^Q$ with $n \in \mathbb{Z}$. In particular, the momentum map $z_\Omega$ is defined only up to addition of an integer. We shall always choose this integer such that $z_\Omega>0$. 
\item[(ii)] The non-zero real multiples of the curvature form $\Omega$ are the only elements in $\mathrm{Ham}(S,K_S)$ for which we don't normalize the momentum map to have zero integral.
\end{itemize}
\end{rem}
In light of this remark it is convenient to clarify the choices made for the momentum map of $\Omega$ in the following
\begin{defn} \label{adm-mo}
Assume that $S$ is compact. An admissible momentum map for $\Omega$ is a momentum map $z_{\Omega}$ such that Proposition \ref{lift} holds and which moreover satisfies $z_{\Omega}>0$ over $S$.
\end{defn}
Lifted circle actions feature prominently in the following construction.
\begin{defn} (\cite{Swann}) \label{def-tw}
Let $Q \xrightarrow{\pi_Q} S$ be a principal circle bundle over the compact manifold $S$, equipped a principal connection form $\Theta$ in $Q$ with curvature form $\Omega$. Consider an $\Omega$-Hamiltonian circle action $ \mathbb{S}^1\ni \lambda \mapsto R^S_\lambda$ on $S$. The twist $Z$ of $S$ is the quotient $\pi_Z: Q\to Z:=Q \slash \bbS$
with respect to the lifted action $(\widetilde{R}_{\lambda})_{\lambda \in \bbS}$ from Proposition \ref{lift}. 
\end{defn}
In this context, $Q$ is called a twist bundle of $S$. If $Z$ is a smooth 
manifold the twist is called smooth. The principal circle action on $Q$ projects onto an $\bbS$-action $(R^Z_{\lambda})_{\lambda \in \bbS}$ on $Z$, with infinitesimal generator to be denoted by $K_{Z}$.
\subsection{Twist correspondence for Hamiltonian forms} \label{twist-co}
Consider a smooth twist $Z$ obtained from $S$ as in Definition \ref{def-tw}.
The space of $\bbS$-invariant forms on $S$ respectively $Z$ are denoted by $\Lambda_{\bbS} ^\star S$ respectively $\Lambda_{\bbS} ^\star Z$. According to Remark \ref{lift-u}, (i) we choose the momentum map $z_\Omega$ to be admissible;
this can be always arranged for whenever $S$ is compact. 

The main objective of this section is to compare, in an explicit way, the spaces $\Ham^2(S,K_S)$ and $\Ham^2(Z,K_Z)$.
To achieve this, we first consider the differential operator $\di_{\mkern+1.6mu  \Omega}:\Lambda^{\star}S \to \Lambda^{\star}S$ which acts according to 
		\begin{equation*}
		\di_{\mkern+1.6mu \Omega}\alpha=\di\alpha+z_{\Omega}^{-1}\Omega\wedge(K_S\cntrct\alpha).
		\end{equation*}
		In order to explain in more detail some aspects of the twist construction we record the following 
\begin{lemma}\label{deomega}
The following hold
	\begin{itemize}
		\item[(i)] the differential operator $\di_{\mkern+1.6mu  \Omega}$
		satisfies
		\begin{equation*}
		\di_{\mkern+1.6mu\Omega}^2\alpha=z_\Omega^{-1}\Omega\wedge\L_{K_S}\alpha
		\end{equation*}
		as well as $\di_{\mkern+1.6mu \Omega}\Lambda_{\bbS} ^\star S \subseteq \Lambda_{\bbS} ^\star S$
		\item[(ii)] there exists an exterior algebra isomorphism 
		$\iota_\Omega :\Lambda_{\bbS} ^\star S\to \Lambda_{\bbS}^\star Z$ satisfying 
		\begin{equation}\label{tw_forms}
		\begin{split}
		\pi_Z^\star(\iota_\Omega(\alpha))&=\pi_Q^\star\alpha-z_\Omega^{-1}\Theta\wedge\pi_Q^\star(K_S\cntrct\alpha)
		\end{split}
		\end{equation}
		\item[(iii)] we have $\di \circ \iota_\Omega=\iota_\Omega \circ \di_{\mkern+1.6mu \Omega}.$
	\end{itemize}
\end{lemma}
\begin{proof}
One proves (i) by direct computation,  (ii) follows from \cite[Lemma 3.4]{Swann}, while (iii) follows from \cite[Corollary 3.6]{Swann}.
\end{proof}
To keep track of the degree of the forms involved we denote with $\iota^p_\Omega$ the restriction of $\iota_\Omega$ to $\Lambda^p_{\bbS}S$. We can now describe the geometry of the fibration $\pi_Z : Q \to Z$ in more detail as follows.
The 2-form $\Omega$ induces a closed 2-form $\Omega_Z$ on $Z$ with respect to which the $\bbS$-action induced by the principal action on $Q$ is Hamiltonian; explicitly 
\begin{itemize}
\item[$\bullet$] we have $\Omega_Z=\iota^2_{\Omega}(z^{-1}_{\Omega}\Omega)$. Since $K_Z$ is induced by the projection of $T_Q$ onto 
$Z$, the form $\Omega_Z$ is Hamiltonian with respect to $K_Z$, with momentum map $K_Z \cntrct \Omega_Z=\di z^{-1}_{\Omega}$
\item[$\bullet$] the form $\Theta_Z:=(z_{\Omega} \circ \pi_Q)^{-1}\Theta \in \Lambda^1Q$ is a principal connection form in $Q$, with respect 
to the lifted action $(\tilde{R}_{\lambda})_{\lambda \in \bbS}$; in addition we have $-\di \Theta_Z=\pi_Z^{\star}\Omega_Z$.
\end{itemize}
\begin{rem}\label{symmetry}
$S$ and $Z$, in case the latter is smooth, are dual to each other in the following sense.  
As outlined above the closed 2-form $\Omega_Z$ is Hamiltonian with respect to the action $(R^Z_{\lambda})_{\lambda \in \bbS}$. The twist of $Z$ by this action is precisely $S$, see \cite[Section 3.2]{Swann}.
\end{rem}

Furthermore we define the translation map according to 
\begin{equation}\label{tau}
\tau_\Omega:\Ham^2(S,K_S)\longrightarrow\Lambda^2_{\bbS}S,\qquad \tau_\Omega(\alpha)=
\alpha-{\underline{z}_\alpha}{z_{\Omega}^{-1}}\Omega.
\end{equation}

\begin{propn}\label{tw_ham}
The translation map satisfies the following 
	\begin{itemize}
		\item[(i)] $\mathrm{im}(\iota_\Omega^2 \circ\tau_\Omega)\subset \Ham^2(Z,K_Z)$
		\item[(ii)] if $\om\in \Ham^2(S,K_S)$ is such that $ \tau_\Omega(\omega)$ is non-degenerate, then so is $\iota^2_\Omega \circ\tau_\Omega(\omega)$. 
	\end{itemize}
\end{propn}
\begin{proof}
	By direct computation one derives that $\di_{\mkern+1.6mu \Omega}\circ\tau_\Omega=0$ which, combined with (iii) in Lemma \ref{deomega}, proves (i). Claim (ii) follows from having $\iota_\Omega$ an exterior algebra isomorphism.
\end{proof}
These considerations enter the following correspondence between Hamiltonian forms on $S$ respectively $Z$, which is the main result of this section.
\begin{propn} \label{iso-ham}Assume that $S$ is compact and choose an admissible momentum map $z_{\Omega}$. The map 
\begin{equation*}
		\iota_\Omega^2 \circ \tau_\Omega : \Ham^2(S,K_S) \to \Ham^2(Z,K_Z)
		\end{equation*}
		is a vector space isomorphism.
\end{propn}
\begin{proof}
First we prove that the map in the statement is surjective. Pick $\beta \in \Ham^2(Z,K_Z)$; since 
$\Ham^2(Z,K_Z) \subseteq \Lambda^2_{\bbS}Z$, using Lemma \ref{deomega},(ii) ensures that $\beta=\iota_{\Omega}^2\gamma$ with 
$\gamma \in \Lambda^2_{\bbS}S$. By \eqref{tw_forms} this reads 
\begin{equation*}
\pi_Z^{\star}\beta=\pi_Q^\star\gamma-z_\Omega^{-1}\Theta\wedge\pi_Q^\star(K_S\cntrct\gamma).
\end{equation*}
Since $K_Z\cntrct\beta=\di \underline{z}_{\beta}$ and $(\mathrm{d} \pi_Z)T_Q=K_Z$ it follows that 
$z_{\Omega}^{-1}K_S \cntrct \gamma=-\di \underline{z}_{\beta}$. 

Furthermore, from having $\beta$ closed, we conclude that $\di_{\Omega} \gamma=0$ by Lemma \ref{deomega}, (iii). Thus, 
\begin{equation*}
\di \gamma=\di_{\Omega}\gamma-z_{\Omega}^{-1}\Omega \wedge (K_S \cntrct \gamma)=\Omega \wedge \di(\underline{z}_{\beta} \circ \pi_Z).
\end{equation*}
Equivalently, the form $\alpha:=\gamma-(\underline{z}_{\beta} \circ \pi_Z)\Omega$ is closed, that is $\di \alpha=0$. At the same time 
$K_S \cntrct \alpha=-\di((\underline{z}_{\beta} \circ \pi_Z)z_{\Omega})$, in particular $\alpha$ belongs to $\Ham^2(S,K_S)$. Therefore the normalised momentum map for $\alpha$ reads 
$\underline{z}_{\alpha}=-(\underline{z}_{\beta} \circ \pi_Z)z_{\Omega}+t$ where $t \in \bbR$ hence $\gamma=\tau_{\Omega}\alpha+tz_{\Omega}^{-1}\Omega$. At the same time 
\begin{equation} \label{i-Om}
\tau_{\Omega}\Omega=t_0 z_{\Omega}^{-1}\Omega
\end{equation}
where the constant $t_0:=z_{\Omega}-\underline{z}_{\Omega}>0$. Summarising, we have showed that the form $\gamma=\tau_{\Omega}(\alpha+\tfrac{t}{t_0}\Omega)$ and the surjectivity claim is proved. To finish the proof it suffices to establish injectivity for the translation map. Indeed, if $\alpha \in \Ham^2(S,K_S)$ satisfies $\tau_{\Omega}\alpha=0$ it is easy to derive that $\underline{z}_{\alpha}=cz_{\Omega}$ for some $c \in \bbR$. Integration over $S$ whilst taking into account that $z_{\Omega}>0$ then forces $c=0$ and the claim is fully proved.
\end{proof}
\subsection{K\"ahler twists} \label{k-twts}
We determine precise sets of data on $S$ which via the twist construction yield K\"ahler structures on the twist $Z$. The set-up here is again a manifold $S$ equipped with the same type of twist data as in the beginning of section \ref{gen-consn}. We first recall how the following correspondence between almost complex structures on $S$ respectively $Z$.
\begin{defn} (\cite{Swann})\label{jec}
	 The almost complex structures $J_Z, J_S$ on $Z$, respectively $S$ are equivalent, $J_Z\sim_{\H}J_S$, if they have the same horizontal lift to $\H:=\ker \Theta$ in $\T Q$.
\end{defn}  

If $J_S$ is integrable, by \cite[Proposition 3.8]{Swann} $J_Z$ is integrable as well. In this situation we now outline how to compute parts of the K\"ahler cone $\mathcal{K}(Z,J_Z)$ of $(Z,J_Z)$, based  on Proposition \ref{iso-ham}. We define the twisted K\"ahler cone $\mathcal{K}^{\bbS}_{\Omega}(S,J_S) \subseteq \Lambda^2_{\bbS}S$ of $(S,J_S)$ according to 
$$\mathcal{K}^{\bbS}_{\Omega}(S,J_S):=\{ \tau_{\Omega}\alpha : \ \alpha\in \Ham^2(S,K_S), \ \tau_{\Omega}\alpha \in \Lambda_{J_S}^{1,1}S \ 
\mathrm{and} \ (\tau_\Omega \alpha)(\cdot, J_S \cdot)>0 \}.$$ 
Since the translation map is injective (see end of proof of Proposition \ref{iso-ham}) the twisted K\"ahler cone may be identified with a subset of $\Ham^2(S,K_S)$.
Furthermore, we indicate with $\mathcal{K}^{\bbS}(Z,J_Z)$ the circle invariant K\"ahler cone of $(Z,J_Z)$ and record that 
\begin{equation} \label{k-incl}
\mathcal{K}(Z,J_Z) \cap \Ham^2(Z,K_Z) \subseteq \mathcal{K}^{\bbS}(Z,J_Z).
\end{equation}
The twisted K\"ahler cone measures up to which extent twist duality does not preserve K\"ahler type, as showed below.
\begin{propn}\label{tw_kahler}
Let $J_S$ be a $\bbS$-invariant complex structure on $S$ and consider the induced $\bbS$-invariant complex structure $J_Z$ on $Z$. 
The following hold
\begin{itemize}
\item[(i)] the map $\iota^2_\Omega \circ \tau_{\Omega}$ induces a bijection
\begin{equation*}
\mathcal{K}^{\bbS}_{\Omega}(S,J_S) \longrightarrow \mathcal{K}(Z,J_Z) \cap \Ham^2(Z,K_Z)
\end{equation*}
\item[(ii)] if $K_Z$ has zeroes or if $H^1_{dR}Z=0$ the inclusion in \eqref{k-incl} is an equality
\item[(iii)] if $\Omega$ is non-degenerate the twisted K\"ahler cone does not contain K\"ahler forms
\item[(iv)] assuming that $\Omega(\cdot ,J_S \cdot)>0$ we have that $z_{\Omega}^{-1}\Omega \in \mathcal{K}^{\bbS}_{\Omega}(S,J_S)$.
\end{itemize} 
\end{propn}
\begin{proof}
(i) follows from the isomorphism in Proposition \ref{iso-ham} combined with part (ii) in Proposition \ref{tw_ham}.\\
(ii) letting $\alpha$ belong to the K\"ahler cone $\mathcal{K}(Z,J_Z)$ the definitions ensure that the pair $(g_{\alpha}:=
\alpha(\cdot ,J_Z \cdot),J_Z)$ is K\"ahler. If in addition  $\alpha$ is $\bbS$-invariant and $K_Z$ has zeroes it follows that $K_Z \in \aut_0(Z,g_{\alpha},J_Z)$; 
hence the K\"ahler form $\alpha$ must be hamiltonian with respect to $K_Z$. This proves equality in \eqref{k-incl} when $K_Z$ has zeroes.
When $H^1_{dR}Z=0$ equality follows from $\Lambda^2_{\bbS}Z \cap \ker \di=\Ham^2(Z,K_Z)$.\\
(iii) assume that $\tau_{\Omega}\alpha=\alpha-z_{\Omega}^{-1}\underline{z}_{\alpha}\Omega$ with $\alpha\in \Ham^2(S,K_S) $ belongs to the twisted K\"ahler cone. Clearly, $\tau_{\Omega}\alpha$ is closed if and only if we have $\di (z_{\Omega}^{-1}\underline{z}_{\alpha})\wedge\Omega=0$. Since $\Omega$ is non-degenerate this yields $\underline{z}_{\alpha}=cz_{\Omega}$ where $c \in \bbR$. After integration, it follows that $0=c\int_Sz_\Omega$; since $z_{\Omega}$ is admissible, it has non-zero integral hence $c=0$ and further $\underline{z}_{\alpha}=0$. It follows that 
$K_S \cntrct \alpha=0$ and since $\tau_{\Omega}\alpha=\alpha$ we obtain a contradiction with having $\tau_{\Omega}\alpha$ non-degenerate.\\
(iv) follows from \eqref{i-Om}.
\end{proof}
Part (iii) in Proposition \ref{tw_kahler} hightlights the different geometric nature of the twisted K\"ahler cone; part (iv) 
in the same proposition provides exmaples of instances when the twisted K\"ahler cone is non-empty.


\subsection{The Weinstein construction for K\"ahler manifolds}\label{wei_con} Let $(M^{2m},g_{M},J_{M})$ where $ m \geq 1$, be a K\"ahler manifold with K\"ahler form $\omega_{M}=g_{M}(J_{M} \cdot ,\cdot)$. Let $(L,h) \to M$ be a complex Hermitian line bundle such that 
$c_1(L)=[\om_M]$, in other words $M$ is a Hodge manifold polarised by $L$. In the sequel $(M,L)$ will denote a polarised K\"ahler manifold.  Chern-Weil theory allows choosing a linear connection $D$ in $L$ such that  
\begin{equation} \label{cD}
Dh=0, \ R^D=\1 \, \om_M \otimes 1_L.
\end{equation}

The sphere bundle $P:=S(L)$
induces a principal $\mathbb{S}^1$-fibration $\mathbb{S}^1 \hookrightarrow P \xrightarrow{\ p\ } 
M$. The action of $\mathbb{S}^1$ on $P$ is $(m,s_m) \lambda=(m, \lambda s_m)$ for $\lambda \in \mathbb{S}^1$, where in the r.h.s. complex multiplication is meant. Then $L\simeq P \times_{\mathbb{S}^1} \bbC$ with respect to the action given by $(u, w)\lambda:=(R^P_\lambda(u), w\lambda^{-1})$; the isomorphism is given by the map $[u,w] \mapsto (m,ws_m)$ whenever $u=(m,s_m) \in P$ and $w \in \bbC$, which is easily checked to be well defined. The connection $D$ induces a principal connection $1$-form $\Theta$ in $\Lambda^{1}(P)$ which satisfies $-\di\Theta=p^{\star}\om_M$ (see  \cite[Section 5]{Mor} for details on the correspondence between $D$ and $\Theta$).

Now choose a K\"ahler manifold $(N^{2n},g_N,J_{N}), n \geq 1$ equipped with a circle action (on the right)  
$\bbS \ni \lambda \mapsto R^N_{\lambda} \in \Aut_0(N,g_N,J_N)$; we denote by $K_{N} \in \aut_0(N,g_N,J_N)$ its infinitesimal generator and by $z_{N}$ a momentum map, thus $K_N \cntrct \omega_N=\di z_N$.
Furthermore we assume that 
\begin{equation} \label{norm-N}
z_N>0
\end{equation}
on $N$. When $N$ is compact this can be done w.l.o.g., as the momentum map is uniquely determined only up to addition of constants.

With respect to the free $\mathbb{S}^1$-action on 
$P \times N$ given by 
\begin{equation} \label{act1}
(u,n)\lambda=(u \lambda, R^N_{\lambda^{-1}}n ).
\end{equation}
form the quotient 
$$ Z(N,M,L)=P \times_{\mathbb{S}^1} N=(P \times N) \slash \mathbb{S}^1.
$$
In order to perform explicit calculations on $Z(N,M,L)$ it is very convenient to observe it is a particular case of the twist construction from Definition \ref{def-tw}.
\begin{propn} \label{ws-tw}
	Let $S=M \times N$ and consider the product circle bundle 
	\begin{equation} \label{a-bd}
	Q=P \times N  \xrightarrow{p \times 1_N} S 
	\end{equation}
	equipped with a principal connection form induced from $\Theta$.
	Endow $S$ with the circle action given by 
	\begin{equation} \label{a-example}
	(m,n)\lambda:=(m,R^N_{\lambda^{-1}}n). 
	\end{equation}
	Then $Z(N,M,L)$ is isomorphic with the twist of $S$ by $Q$ and moreover the curvature of $\Theta$ is $\Omega=\om_M$ and $z_\Omega=1$ is a momentum map with respect to $\Omega$.
\end{propn}	
\begin{proof}
By construction, the curvature form $\Omega$ in \eqref{a-bd} is $\omega_M$.
Hence, the circle action in \eqref{a-example} is Hamiltonian w.r.t. $\Omega$ since it is generated by $K_S=-K_N$ and $K_S \cntrct  \Omega=0$. The lift of \eqref{a-example} to $P \times N$ granted by Proposition \ref{lift} is therefore 
$ (u,n)\lambda=(R^P_\lambda(u),R^N_{\lambda^{-1}}n)$
with momentum map $z_{\Omega}=1$. It follows that the twist of the action \eqref{a-example} is precisely the associated fibre bundle 
$P \times_{\mathbb{S}^1}N$. 
\end{proof}
Note that the momentum map $z_{\Omega}=1$ is admissible in the sense of Definition \ref{adm-mo}; thus results in section 
\ref{twist-co} can be used for the twist pair $(M \times N, Z(N,M,L))$.

In the rest of this section we simply indicate $Z(N,M,L)$ with $Z$. The manifold $Z$ comes equipped with a natural circle action induced by 
\begin{equation} \label{act3}
((u,n),\lambda) \mapsto (u, R^N_{\lambda}n)
\end{equation}
on $P \times N$. Its infinitesimal generator will be denoted by $K$. We have a fibre bundle 
$$ N \hookrightarrow Z \xrightarrow{\ \pi\ } M.
$$
At the same time $P\times N$ becomes a $\mathbb{S}^1$-principal bundle over $Z$, with canonical projection denoted by 
$\tilde{p}: P \times N \to Z$. These fibrations enter the following commutative diagram

\begin{equation} \label{diagrama}
\begin{tikzpicture}[baseline=(current bounding box.center)]
    \node (N21) at (-1.5,0) {};
    \node (N21 label) at (-1.7,0.1) {$\mathbb{S}^1$};

    \node (N22) at (0,0) {$P\times N$};

    \node (N23) at (2,0) {$Z$};

    \node (N13) at (2,1) {$N$};

    \node (N32) at (0,-1.5) {$M\times N$};

    \node (N33) at (2,-1.5) {$M$};

    \draw[-stealth] (N22) -- (N23)  node[midway,above] {\small $\tilde{p}$};

    \draw[-stealth] (N22) -- (N32)  node[midway,left] {\small $p\times\operatorname{id}$};

    \draw[-stealth] (N23) -- (N33)  node[midway,right] {\small $\pi$};

    \draw[-stealth] (N32) -- (N33)  node[midway,above] {\small $\operatorname{pr}_1$};

    \draw[right hook-stealth] (N13) -- (N23);

    \draw[right hook-stealth] (N21) -- (N22);
    


\end{tikzpicture}
\end{equation}
\subsubsection{\bf The K\"ahler structure on $Z$} 
Start with the splitting 
\begin{equation}\label{teta}
\T P=\span \{T_P\} \oplus \H 
\end{equation}
where $T_{P}$ is tangent to the principal $\mathbb{S}^1$-action on $P$ (thus $\Theta(T_P)=1$) and $\H=\ker \Theta$. 
We have a direct sum 
decomposition 
$$ \T(P \times N)=\tilde{\V} \oplus \tilde{\H}
$$
where $\tilde{\V}=\ker (\di \tilde{p})=\span \{T_P-K_N\}$ and $\tilde{\H}=\H \oplus \T N$. Hence,  $\Theta$ in \eqref{teta} becomes a  principal connection form in $\mathbb{S}^1 \hookrightarrow P \times N \to Z$. Since $\T N$ and $\H$ 
are horizontal and invariant under the action \eqref{act1} they project onto 
$\D_{+},\D_{-} \subseteq \T Z$, that is  
\begin{equation*}
\begin{split}
\D_{+}=\di \tilde{p}(\T N)=\ker (\di \pi) \ \mathrm{and} \ \D_{-}&=\di\tilde{p}(\H)
\end{split}
\end{equation*}
giving rise to the direct sum decomposition 
$$ \T Z=\D_{+} \oplus \D_{-}.
$$
The map $\iota^2_{\om_M}: \Lambda_{\mathbb{S}^1}^{\star}N \to \Lambda^{\star}Z$ allows pushing-forward invariant forms from $N$ to $Z$ which, moreover, are vertical in the sense of the following

\begin{defn}
A form $\alpha$ in $\Lambda^{\star}Z$ is called vertical if $X \cntrct \alpha=0$ for all $X$ in $\D_{-}$.
\end{defn}

Denote by $X_{-}$ the horizontal lift to $\D_{-}$ of the vector field $X$ on $M$. If $V$ is an $\mathbb{S}^1$-invariant vector field on $N$ we denote by 
$V_{+} \in \D_{+}$ its projection onto $Z$. The latter is invariant under the circle action in \eqref{act3} and satisfies $[V_{+},X_{-}]=0$.

\smallskip

{\noindent\bf The symplectic form.} Let $\om:=\iota^2_{\omega_M}(z_N\om_M+\om_N)$, see Lemma \ref{deomega} (ii).  The splitting of $\omega$ according to $\T Z=\D_+\oplus\D_-$ is
\begin{equation}\label{omega-def}
\omega:=\omega_+ + \omega_-,
\end{equation}
with components explicitly given by 
\begin{equation*}
\omega_{+}=\iota_{\om_M}^2(\omega_N) \ \mathrm{and} \ \omega_-:=\iota_{\om_M}^2(z_N\om_M)=z\pi^\star\om_M
\end{equation*}
where $z$ is the projection onto $Z$ of the momentum map $z_N$ on $N$. By (i) in Proposition \ref{tw_ham}, and using $\tau_{\om_M}(\om_N)=\om_N+z_N\om_M$ it follows that  $\omega$ is closed. Since $z_N\om_M+\om_N$ is non-degenerate, Proposition \ref{tw_ham} (ii) ensures that $\om$ is non-degenerate as well. By (iii) in Lemma \ref{deomega} we have
\begin{equation} \label{curv-d}
\di \omega_{+}=-\pi^{\star}\om_M \wedge \di z.
\end{equation}

\begin{rem} \label{symp}
In symplectic geometry this recipe to construct symplectic forms is well known and due to Weinstein \cite{W}(see also \cite{McD-S}); it works in the more general set-up when $N$ is equipped with an Hamiltonian action of some Lie group $G$ and $G \hookrightarrow P \to M$ is a principal bundle with a fat connection.
\end{rem}

{\noindent\bf The Riemannian metric.} 
Define 
\begin{equation}\label{def_g}
g:=g_++z\pi^\star g_M,
\end{equation}
where $g_+$ is a Riemannian metric on $\D_+$ given by  
\begin{equation*}
\tilde{p}^{\star}g_{+}=g_{N}+ \Theta \otimes g_{N}(K_N ,\cdot) + g_{N}(K_N ,\cdot) \otimes \Theta+ g_{N}(K_N,K_N) \Theta \otimes \Theta.
\end{equation*}
Note that the symmetric $(2,0)$-tensor on $P \times N$ in the r.h.s
is projectable down to $Z$ since it is invariant under both circle actions and since it vanishes on $T_P-K_N$.

\smallskip

{\noindent\bf The  complex structure.} Let 
$$ J=J_{+}+J_{-},$$ 
 where $J_{+}: \D_{+} \to \D_{+}$ is defined by  
$$ J_{+}:=(g_{+})^{-1}\omega_{+},$$
and  $J_{-}$ is the lift to $\D_{-}$ of $J_{M}$. Since $J\sim_{\tilde \H} J_M+J_N$, it is a 
complex structure on $Z$.

\begin{propn} \label{pro1-c}
\begin{itemize}
\item[(i)] The structure $(g,J,\om)$ is K\"ahler and circle invariant. Moreover the infinitesimal generator of the action  $K=(K_N)_+\in \aut_0(Z,g,J)$ and has momentum map $z$.
\item[(ii)] $\pi:(Z, J) \to (M, J_{M})$ is a holomorphic submersion and $\D_{+}$ is a holomorphic distribution w.r.t. $J$.
\end{itemize}
\end{propn}
\begin{proof}
(i) is Proposition \ref{tw_kahler} applied to the data described in Remark \ref{ws-tw}.

(ii) That $\pi$ (hence $\D_+=\ker (d \pi)$) is holomorphic follows from the construction of $J$ and the fact that $\D_{+}$ is $J$-invariant.  
\end{proof}

The next result shows that the naturally defined distribution $\D_+$ is in fact a foliation with special properties.

\begin{propn} \label{kaehler}
With respect to $g$, the distribution $\D_{+}$ is 
\begin{itemize}
\item[(i)] homothetic with Lee form $\theta=\di \ln z$
\item[(ii)] totally geodesic.
\end{itemize}
\end{propn}
\begin{proof} 
	(i) By the very definition \eqref{def_g} we see that $\pi:Z\to M$ is a conformal submersion. The claim follows since the dilation factor is constant on the base.\\
	(ii) Let $I$ be the almost complex structure defined in \eqref{i}. Then $I\sim_{\tilde\H} J_M-J_N$, and hence $I$ is integrable. The claim then follows from Proposition \ref{tg}. 
\end{proof}

\begin{defn} \label{def-w1}
The K\"ahler manifold $(Z,J,g,\omega)$ is said to be obtained by the Weinstein construction. We shall also denote the underlying manifold by $Z=Z(N,M,L)$ to refer to the fibre, the base and the polarisation thereof.
\end{defn}

\begin{exo} \label{proj}
	Take $N=\mathbb{P}^n$ endowed with the Fubini-Study metric, and with the circle action $[z_0: \ldots :z_n]\lambda=[z_0\lambda : \ldots :z_n]$. Then $Z=\mathbb{P}(L \oplus 1_{\bbC^n})$. We thus recover the Calabi construction as used for example in \cite{hamf1}.
\end{exo}

\begin{rem} \label{rigid-actn}
	Note that when using the Calabi construction in \cite{hamf1}, the $\mathbb{S}^1$-action on $Z$ and hence that of the circle on $N$ are assumed to be {rigid} in the sense that  $\di x_N \wedge \di z_N=0$, where $x_N=g_N(K_N,K_N)$. This is also the case in Example \ref{proj}. However our construction does not need rigidity of the action.
	
	In fact, every rigid action gives rise to a non-rigid action in the following way. Assume that $(N,g_N,J_N)$ admits 
	a rigid action. Pick $f:N \to \bbR$ invariant under the circle such that $\omega_f=
	\omega_N+\di J\di f>0$. Then $(\omega_f,J_N)$ is a K\"ahler structure, w.r.t. which the action of the circle is Hamiltonian with momentum map $z_f=z+\L_{JK}f$; then $x_f=x+\L_{JK}^2f $, is not, in general, only a function of $x_f$. For instance take $f=z^2H$ where $\L_KH=\L_{JK}H=0$,
	$H$ is not constant and such that $\omega_f$ remains positive. 
\end{rem}

\subsection{The exact Weinstein construction} \label{loc_wein_def}
The local version of the above is obtained in a local trivialisation chart of the fibre bundle $Z\to M$. We present a slightly more general situation which will occur naturally in the final part of our discussion. 

Let $(M, J_M,g_M)$ be a simply connected K\"ahler manifold, with exact fundamental form  $\om_M=\di \alpha_M$. The 1-form $\alpha_M$ is unique only up to transformations of the type $\alpha_M\mapsto\alpha_M+\di f$. Let $(N,J_N,g_N)$ be a K\"ahler manifold admitting a non-trivial vector field $K_N\in\aut_0(N,J_N,g_N)$. Furthermore we assume there is a positive momentum map $z_N$ for $K_N$. On the product $M\times N$ consider:
\begin{equation}\label{local_wein}
\begin{split}
\tilde \om_{\tM\times\tN}&=z_{\tN}\om_{\tM}+\om_{\tN}-\di z_{\tN}\wedge\alpha_{\tM},\\
\tilde J_{\tM\times\tN}&=J_{\tM}+J_{\tN}+J_{\tM}\alpha_{\tM}\otimes K_{\tN}-\alpha_{\tM}\otimes J_{\tN}K_{\tN}.
\end{split}
\end{equation}
Direct computation shows these define a K\"ahler structure on $M\times N$ with Riemannian metric $g_{M\times N}$ w.r.t. which the distributions 
\begin{equation}\label{local_wein_distr}
\begin{split}
\D_+&=\T N\\
\D_-&=\{X+\alpha_{M}(X)K_{N} : \ X\in\T M\}
\end{split}
\end{equation}
define a TGHH foliation with Lee form $\theta:=\di \ln z_N$.
 
\subsection{The Picard group of $Z$}\label{pic}

Later on in the paper (see Section \ref{uni}) we will need to explicitly describe 
how holomorphic line bundles over $N$ can be pushed forward to $Z$. To that extent 
we need some preliminary material.

\begin{propn}
The map 
\begin{equation} \label{deRham-i}
\H^2(N,g_N) \to H^2_{dR}(Z), \quad \alpha \mapsto [\iota^2_{\om_M}\circ\tau_{\om_M}(\alpha)].
\end{equation}
is well-defined and induces 	 an injective linear map $\iota^2:H^2_{dR}(N) \to H^2_{dR}(Z)$.
\end{propn}
\begin{proof}
 Since $K_N\in \mathfrak{aut}_0(N,g_N,J_N)$, Proposition \ref{ham-K} (i) ensures that we have an inclusion $\H^2(N,g_N)\subseteq \Ham^2(M\times N, K_N)$.  Therefore the map in \eqref{deRham-i} is well-defined by Proposition \ref{tw_ham} (i). The induced map $\iota^2$  between cohomology groups is injective since by construction $\tilde p_u^{\star}\iota^2(c)=c$ whenever $c \in H^2_{dR}(N)$ and $u \in P$. Here we have indicated with $\tilde p_u:=\tilde{p}(u, \cdot):N \to Z, u \in P$ the fibre inclusion map.
\end{proof}

  The counterpart of \eqref{deRham-i} at the level of Picard groups is explicitly described in the following 
\begin{propn} \label{picard}
Assuming that $N$ and $M$ are compact we have a natural injective morphism 
$$ \mathrm{Pic}(N,J_N) \to \mathrm{Pic}(Z,J), \quad l \mapsto l_{+}$$
such that $c_1(l_{+})=\iota^2(c_1(l))$. 
\end{propn}
\begin{proof}
Let $l\in \mathrm{Pic}(N,J_N)$ be represented by a principal circle bundle $Q \xrightarrow{\pi_Q} N$; consider the harmonic representative $\Omega_Q \in \Lambda^{1,1}_{J_N}N$  in the class $c_1(Q)$. Choose 
a principal connection 
form $\Theta_Q$ in $Q$ with $-\di \Theta_Q=\pi^{\star}\Omega_Q$. 
Proposition \ref{ham-K} (i) shows that the circle action on $N$ is $\Omega_Q$-Hamiltonian. Lift the action on $N$ to an action $\widetilde{R}_\lambda$ on $Q$, according to Proposition \ref{lift}. With respect to the free circle action on 
$P \times Q$ given by  
\begin{equation} \label{a-lift}
((u,q),\lambda) \mapsto (R^P_\lambda(u), \widetilde{R}_{\lambda^{-1}}q)
\end{equation} 
define  $Q_{+}:=(P \times Q) \slash \mathbb{S}^1$. Letting $p_{Q_{+}}:P \times Q \to Q_{+}$ be the canonical projection we briefly describe the principal bundle structure in 
$Q_{+}$. The projection $\pi_{Q_{+}}:Q_{+} \to Z$ respectively the circle action $R^{Q_{+}}:Q_{+} \times \mathbb{S}^1 \to Q_{+}$ are uniquely determined 
by requiring the diagrams
\begin{equation}\label{com-co}
\begin{tikzcd}
&P\times Q \arrow[d,swap,"p_{Q_{+}}"] \arrow[r,"1_{P} \times \pi_Q"] &P \times N \arrow[d,"\tilde{p}"]\\
&Q_{+} \arrow[r,"\pi_{Q_{+}}"]&Z\\
\end{tikzcd}
\end{equation}
respectively 
\begin{equation}\label{com-ac}
\begin{tikzcd}
&(P\times Q) \times \bbS \arrow[d,swap,"p_{Q_{+}} \times 1_{\bbS}"] \arrow[r,"1_{P} \times R^Q"] &P \times Q \arrow[d,"p_{Q_{+}}"]
\\
&Q_{+} \times \bbS \arrow[r,"R^{Q_{+}}"]&Q_{+}\\
\end{tikzcd}
\end{equation}
be commutative. 

Let $z_{\Omega_Q}$ be the momentum map determined from $\di z_{\Omega_Q}=K_N \cntrct 
\Omega_{Q}$ and let $\widetilde{K_N}$ be the infinitesimal lift of $K_N$ to $Q$ constructed in \eqref{ilift}. The infinitesimal generator of \eqref{a-lift} is $T_P-\widetilde{K_N}$. The form $(\pi_Q^\star z_{\Omega_Q})\Theta_P+\Theta_Q$ in $\Lambda^1(P \times Q)$ is invariant under \eqref{a-lift}(use (iii) in Proposition \ref{lift}) and vanishes on $T_P-\widetilde{K_N}$. Thus there exists 
$\Theta_{Q_{+}} \in \Lambda^1Q_{+}$ such that 
$$p_{Q_{+}}^{\star}\Theta_{Q_{+}}=(\pi_Q^\star z_{\Omega_Q})\Theta_P+\Theta_Q.$$ As the latter form is invariant under the principal circle action on $Q$ the commutativity of \eqref{com-ac} entails that $\Theta_{Q_{+}}$ is a principal connection form in $Q_{+}$.

From 
$$-\di (\Theta_Q+z_{\Omega_Q}\Theta_P)=\pi_Q^{\star}\Omega_Q+\Theta_P \wedge 
\pi^{\star}(K_N \cntrct \Omega_Q)+z_{\Omega} \pi_P^{\star}\om_M$$
 we get  
$$\Omega_{Q_{+}}=(\iota^2_{\om_M}\circ\tau_{\om_M})(\Omega_Q),$$
  thus 
$c_1(Q_{+})=\iota^2(c_1(Q))$, by \eqref{deRham-i}. 

If $\Theta^{\prime}_Q$ is another principal connection form with $-\di\Theta^{\prime}_Q=\pi^{\star}\Omega_Q$ the lifted circle action on $Q$ becomes $\bbS \ni \lambda \mapsto F \circ \widetilde{R}_\lambda \circ F^{-1}$ where $F:Q \to Q$ is a gauge transformation. The map $1_P \times F$ conjugates the circle actions on $P \times Q$ thus the isomorphism class of $Q_{+}$ remains unchanged. By standard arguments one now proves 
\begin{equation}\label{tens}
(Q_1 \otimes Q_2)_{+}=(Q_1)_{+} \otimes (Q_2)_{+}.
\end{equation}
This shows that if $Q_1$ and $Q_2$ represent the same class in $\mathrm{Pic}(N,J_N)$, then $(Q_{1})_+$ is isomorphic to $(Q_{2})_+$ and represents a class $l_+\in \mathrm{Pic}(Z,J)$ and the association $l\mapsto l_+$ is a group morphism.

Injectivity amounts to showing that if $Q_{+}$ is trivial then so is $Q$. This in turn follows easily from \eqref{com-co}.
\end{proof}

\subsection{The automorphism Lie algebras} \label{aut}
In  this subsection, unless otherwise indicated,  
$M$ and $N$ will be compact. We give here the description of the Lie algebra 
$\mathfrak{aut}(Z,J)=\{X\in \T Z : \L_XJ=0\}$ of holomorphic vector fields on $(Z,J)$. The ideal of holomorphic vector fields with zeroes will be denoted $\mathfrak{aut}_0(Z,J)$. Any $X\in \mathfrak{aut}_0(Z,J)$ admits an holomorphy potential $f_X:Z\to\bbC$ determined from $X^{01}\cntrct\omega=\del f_X$. Here $X^{01}=\frac{1}{2}(X+iJX)$ in $T^{01}Z$. As in the real setup, holomorphy potentials can be normalised by $\int_Z f_X=0$.

Consider the complex vector bundle $E:=P \times_{\bbS} \aut(N,J_N) \to M$ where the circle acts on $P \times \aut(N,J_N)$ according to 
\begin{equation*}
((u,X), \lambda) \mapsto (R^P_\lambda(u), (R^N_{\lambda^{-1}})_{\star}X).
\end{equation*}
The Lie algebra structure in $\aut(N,J_N)$ induces a Lie algebra structure 
\begin{equation*}
[\cdot, \cdot] : \Gamma(E) \times \Gamma(E) \to \Gamma(E).
\end{equation*}
We denote by $D^E$ the connection induced by $D$ in $E$ and by $\del^E$ the associated Dolbeault operator.
\begin{propn} \label{lifts} Let $\D^{01}_+=\D_+\cap \T^{0,1}Z$. There exists  a natural injective map 
	\begin{equation*}
	\Phi:\Gamma(E) \to \D_{+}^{01}, \ b \mapsto \Phi(b) 
	\end{equation*} 
such that
\begin{itemize}
\item[(i)] for $b\in \Gamma(E)$ we have $\Phi(b)\in \mathfrak{aut}(Z,J)$ if and only if  
$$ b \in \mathcal{E}:=\ker\left(\del^E: \Gamma(E) \to \Lambda^{10}(M,E)\right)$$
\item[(ii)]  $\Phi$ is a Lie algebra morphism. 
\end{itemize}
\end{propn}
\begin{proof}
The map $\Phi$ is constructed as follows.  A section $b$ in $E$ canonically induces a map $V^b:P \to \h_N$ such that $V^b(R^P_\lambda(u))=(R^N_{\lambda^{-1}})_{\star}V^b(u)$. Therefore the vector field $X^b$ on $P \times N$ given by $(u,n) \mapsto (V^b(u))_{n}$ is invariant w.r.t. the circle action \eqref{act1} and thus projects onto a vector field $\Phi(b) \in \D_{+}^{01}$.\\
(i) Choosing a basis $\{X^j\}$ in $\h_N$ allows writing $X^b=\sum \limits_{i}f_jX^j$ with $f_j:P \to \bbC$. Clearly $\Phi(b)$ is holomorphic if and only if  
$[X^b,JX^{\H}]=J[X^{b},X^{\H}]$ in $\T (P \times N)$ whenever $X \in \T M$. Since 
\begin{equation} \label{brack1} [X^{\H},X^b]=\sum \limits_{i}(\L_{X^{\H}}f_j) X^j 
\end{equation}
it follows that $\L_{(JX^{\H})}f_j=-\1\, \L_{X^{\H}}f_j$; equivalently $\L_{(JX^{\H})}V=-\1\, \L_{X^{\H}}V$ and the claim is proved.
\\
(ii) Because $X^{b_i} \in \T N, i=1,2$, we have $[X^{b_1},X^{b_2}]_{(u,n)}=[V^{b_1}(u), V^{b_2}(u)]_n$ on $P \times N$ and the claim follows by projection 
onto $Z$.  
\end{proof}
 
In order to extend $\Phi$ to the space $\mathcal{E} \oplus \aut_0(M,J_M)$ we establish the following
\begin{lemma} \label{lie-inl} 
\begin{itemize}
\item[(i)]The map 
$$\Phi:\aut_0(M,J_M) \to \aut(Z,J) \ \mathrm{given \ by} \ \Phi(X):=X_{-}+(f_X \circ \pi) K^{01}$$ is a well defined Lie algebra monomorphism.
\item[(ii)] We have $[\Phi(X),\Phi(b)]=\Phi(\rho(X)b)$, where the Lie algebra representation 
$$ \rho:\aut_0(M,J_M) \to \mathrm{Der}(\mathcal{E}) \ \mbox{is given by} \ \rho(X)b=D_X^Eb-f_X[b,K^{01}].
$$
\end{itemize} 
\end{lemma}
\begin{proof}
(i) Observe that  $\L_{X_{-}}J=2\1\, \pi^{\star}(X \cntrct \om_{M}) \otimes K^{01}$ for all $X \in \aut(M,J_M)$. This can be easily verified 
using the first formula in \eqref{d1} and the fact that $(\L_{X_{-}}J)V=0$ for all $V \in \D_{+}$. If moreover $X \in \aut_0(M,J_M)$, the claim follows from the identity 
$\L_{(f_XK^{01})}J=-2\1\, \pi^{\star}(\del f_X) \otimes K^{01}$.\\
(ii) Working on $P \times N$, it is enough to compute the commutator $[X^{\H}+(f_X \circ \pi_P)K_N^{01},X^b]$. Writing $X^b=\sum f_jX^j$ as in the proof of 
(ii) in Proposition \ref{lifts} we find 
\begin{equation*}
[X^{\H}+(f_X \circ \pi_P)K_N^{01},X^b]=\sum \limits_{i}(\L_{X^{\H}}f_j) X^j+f_j(f_X \circ \pi_P)[K^{01}_N, X^j]
\end{equation*}
by using \eqref{brack1} and the fact that $f_X \circ \pi_P$, and $f_j$ are functions on $P$. The claimed commutation relation follows now easily. At the same time it shows that 
$\rho(X)$ preserves $\mathcal{E}$ (this can be also checked directly from the definition of $\rho$). The Jacobi identity in $\h_Z$ ensures that $\rho$ is a representation of Lie algebras.
\end{proof}

\begin{thm} \label{holvf} The map 
\begin{equation*}
\Phi : \mathcal{E} \rtimes_{\rho} \aut_0(M,J_M) \to \aut(Z,J)
\end{equation*}
is a Lie algebra isomorphism.
\end{thm}
\begin{proof} By Lemma \ref{lie-inl}, it is enough to see that $\Phi$ is surjective. Pick $H_0=V_0+X_0$ in $\aut(Z,J)$, split according to $\T Z=\D_{+} \oplus \D_{-}$. The equation $\L_{H_0}J=0$ has four components we will 
now describe. Working in direction of $\D_{-}$ we get $(\L_{V_0}J)Y_-+(\L_{X_0}J)Y_-=0$, for all $Y\in \T M$.  Since $\D_{+}$ is holomorphic, by projecting onto $\D_{-}$ respectively 
$\D_{+}$ we get 
\begin{equation} \label{hvvf-1}
((\L_{X_0}J)Y_-)_{\D_-}=0
\end{equation}
as well as $(\L_{V_0}J)Y_-+((\L_{X_0}J)Y_-)_{\D_+}=0$ for all $Y\in \T M$. But 
\begin{equation*}
\begin{split}
((\L_{X_0}J)Y_-)_{\D_+}&=(-\nabla^g_{JY_-}X_0+J\nabla^g_{Y_-}X_0)_{\D_+}\\
&=\eta_{JY_-}X_0-J\eta_{Y_-}X_0=2\1\, \pi^{\star}\om_{M}(X_0,Y_-)K^{01}
\end{split}
\end{equation*} 
by the first formula in \eqref{d1} and $\om_{-}=z\pi^{\star}\om_{M}, K=-zJ\zeta$. Therefore 
\begin{equation} \label{hvvf-2}
(\L_{V_0}J)Y_-=-2\1 \pi^{\star}\om_{M}(X_0,Y_-)K^{01}
\end{equation}
for all $Y \in \T M$. In direction of $\D_{+}$ we have  $(\L_{V_0}J)V+(\L_{X_0}J)V=0$; because $\D_{+}$ is holomorphic, $(\L_{V_0}J)V$ belongs to 
$\D_{+}$. Moreover $(\L_{X_0}J)V=-\nabla^g_{JV}X_0+J\nabla^g_VX_0$ belongs to $\D_{-}$ since $\D_{+}$ is totally geodesic. It follows that 
\begin{equation} \label{hvvf-3}
(\L_{V_0}J)V=0
\end{equation}
as well as 
\begin{equation*} \label{hvvf-4}
(\L_{X_0}J)V=0
\end{equation*}
whenever $V \in \D_{+}$.  
\begin{claim}
	The vector field $X_0$ is projectable,  i.e. $X_0=X_{-}$ for some vector field $X$ on $M$.
\end{claim}
Indeed, let $f:M\rightarrow \bbC$ be a {\em local} holomorphic function, defined on some open set $U\subset M$. Consider the function $F:=\L_{Y_0}(f\circ p):p^{-1}(U)\times N\rightarrow \bbC$, where $Y_0\in\H$ is the horizontal lift of $X_0$ to $P\times N$. Here $p:P\times N\longrightarrow M$ is considered as $p(u,n)=p(u)$, by a slight abuse of notation. We have:
\begin{equation*}
\begin{split}
\L_{JV}F&=\L_{JV}(\L_{Y_0}(f\circ p))=\L_{Y_0}(\L_{JV}(f\circ  p))+\L_{[JV,Y_0]}(f\circ p)\\
&=\L_{[JV,Y_0]}(f\circ p)\ \ \text{since}\ \ JV\in \D_+\\
&=\L_{J[V,Y_0]}(f\circ p) \ \ \text{by} \ \eqref{hvvf-4}\\
&=\1\, \L_{[V,Y_0]}(f\circ p) \ \ \text{since $f$ is holomorphic}. 
\end{split}
\end{equation*}
This proves that $F$ is holomorphic in the second argument, hence $F$ only depends on $p^{-1}(U)$, as $N$ is compact. Therefore, $Y_0$ does not depend on $N$ and the claim is proved.

\smallskip

Then \eqref{hvvf-1} guarantees 
that $X \in \aut(M,J_M)$. To deal with \eqref{hvvf-3} we consider the horizontal lift $\widetilde{V}$ to the circle fibration $P \times N \to Z$ in diagram  \eqref{diagrama}. It determines a family $\widetilde{V}_u, u\in P$, of vector fields in $T^{01}N$ such that $\widetilde{V}_{R^P_\lambda(u)}=(R^N_{\lambda^{-1}})_{\star}\widetilde{V}_u$ and moreover $\widetilde{V}_u \in \aut(N,J_N)$ for $(u,\lambda) \in P \times \bbS$. 
Choose now a basis $\{X^j\}$ in $\aut(N,J_N)$ with $X^1=K^{01}$ and write $V=\sum \limits_{j}f_jX^j$ with $f_j:P \to \bbC$. Using \eqref{hvvf-2}
yields 
$$ \sum \limits_{j} (\di f_j-\1\, J\di f_j)(Y) X_j=2\pi^{\star}\om_{M}(X_0,Y)X_1.
$$
It follows that $\pi^{\star}(X \cntrct \om_M)=\del f_1$ and $\di f_j-\1\, J\di f_j=0$ on $\H \subseteq TP$ for $j \geq 1$.
It is easy to see that up to a constant $f_1=f_{X} \circ \pi$ where $f_{X}:M \to \bbC$ satisfies $\int_M f_{X}=0$. We have showed that $X \in \aut_0(M,J_M)$. Furthermore $V-f_1K^{01}:P \to \aut(N,J_N)$ has the same invariance properties as $\tilde{V}$ hence $V_0-f_1K^{01}=\Phi(b),b \in \mathcal{E}$, by using Proposition \ref{lifts}. That is $H_0=\Phi(b)+X_{-}+(f_{X} \circ \pi)K^{01}=\Phi(b+X)$ and the proof is complete.
\end{proof}

An explicit way of determining $\mathcal{E}$ is outlined below.
The circle action on $N$ induces a circle action on $\gh_N$ via $ \bbS\ni \lambda  \mapsto (R^N_\lambda)_\star$. Let $k_p \in \mathbb{Z}, 0 \leq p \leq d$, be the weights of the representation $(\bbS, \aut(N,J_N))$, with the convention $k_0=0, \ k_p \neq k_j$ for $0 \leq p \neq j \leq d$. We have a splitting 
\begin{equation}\label{descomp}
 \aut(N,J_N)=\bigoplus_{p=0}^d  \gh_p, 
\end{equation}
where $(R^N_\lambda)_\star X=\lambda^{k_p}X$ whenever $X\in\gh_p$.
 
Denote by $L^i, i\in \mathbb{Z}$, the $i$-th tensor power of $L$, and by $\partial^{i}:\Gamma(L^i)\to\Lambda^{1,0}(M)\otimes L^{-i}$ the  Dolbeault operator of the holomorphic bundle $L^i$.

\begin{propn} \label{e}
We have an isomorphism
\begin{equation*}
\mathcal{E}\simeq \h_0 \oplus  \left(\bigoplus \limits_{1 \leq p \leq d, k_p <0}\ker \del^{k_p}   \otimes \h_p\right).
\end{equation*}
\end{propn}
\begin{proof}
Let $\tilde{\h}_0$, respectively $\tilde{\h}_p$ denote the trivial bundle $M\times \h_0 \to M$, respectively $M\times \h_p\to M$. From the definition of $E$ and \eqref{descomp} we have a vector bundle isomorphism 
$ E=\tilde{\h}_0 \oplus  \left(\bigoplus \limits_{1 \leq p \leq d} L^{k_p} \otimes \tilde{\h}_p \right)
$ 
which in turn induces an isomorphism 
$$\ker(\del^E)=\h_0 \oplus  \left(\bigoplus \limits_{1 \leq p \leq d}\ker \del^{k_p}  \otimes \h_p\right).$$ Because $L$ is positive, $\ker \del^{k_p} =0$ for $k_p \geq 0$ by Kodaira's vanishing theorem, and the claim follows.
\end{proof}
\begin{rem}
\begin{itemize}
\item[(i)]When $N=\mathbb{P}^1$ so that $Z=\mathbb{P}(L \oplus 1)$ the result above has been proved in \cite{ACGT}, see also \cite{Ma} for the case when $M$ is a Riemann surface.
\item[(ii)] Proposition \ref{e} still works when no assumption is made on the curvature form $R^D$ of the Hermitian connection $D$ in $L$, i.e we do not assume the curvature condition $R^D=\1\, \om_M\otimes 1_L$. Going through the proof above we see that the only adjustment needed is to replace $\aut_0(M,J_M)$ with the space 
$$\{X \in \aut(M,J_M)\ :\  X \cntrct R^D=\del f_X, f_X:M \to \mathbb{C}, \int_{M}f_X=0 \}.$$ Moreover it suffices to assume that $M$ only carries a complex structure and a volume form.
\end{itemize}
\end{rem}
The holomorphy potentials for elements in $\aut_0(Z,J)$ are determined as follows. If $ b \in \ker \del^{k_p} \otimes \h_p$  is of the form
$b=\sum_j s_j \otimes X^j$, let $\Phi(b)\in \aut(Z,J)$ be the corresponding vector field. Let  $f_j$ be the lift of $s_j$ to $P$ which 
thus satisfy $f_j(R^P_\lambda(u))=\lambda^{-k_p}f_j(u)$. If $X^j$ has a potential, i.e. $X^j\in \gh_N^0$, let $F_{X^j}$ be its holomorphy potential. Since $X^j$ belongs to $\h_p$, $(R^N_\lambda)_\star X^j=\lambda^{k_p}X^j$. Then  $f_jF_{X^j}$ is $\bbS$-invariant on $P\times N$ with respect to the action \eqref{act1}, and hence it projects on a function $F_{\Phi(b)}$ on $Z$ which is the holomorphy potential of  $\Phi(b)$; explicitly  
\begin{equation} \label{HP-s}
F_{\Phi(b)}\circ\pi=\sum_j f_j F_{X^j}.
\end{equation}
To finish this section we determine the Lie algebra $\mathfrak{iso}(Z,g)$ based on the following preliminary 
\begin{lemma} \label{hamc}
We have $\Ham^c(Z,\pi^{\star}\omega_M)=\Gamma(\D_{+}) \oplus \pi^{\star}\Ham^c(M,\omega_M)$.
\end{lemma}
\begin{proof}
Pick $X \in \Ham^c(Z,\pi^{\star}\omega_M)$. According to the definition of the latter space in section \ref{k-twts} we have $X \cntrct  \pi^{\star}\omega_M=\di a+J\di b$ 
for some functions $a,b:Z \to \bbR$. In particular $\di a+J\di b$ vanishes on $\D_{+}$; lifted to $P \times N$ the function $a+ib$ is holomorphic in direction of $N$ thus it must be constant in those directions. It follows that $a+ib=(f+ih) \circ \pi$ for some function 
$f+ih:M \to \bbC$. Therefore 
$\di a+J\di b=\pi^{\star}(\di f+J_M \di h)$
and the claim follows easily.
\end{proof}
From Theorem \ref{holvf} it follows that $\Phi(\mathfrak{iso}_{\bbS}(N,g_N)\oplus \mathfrak{aut}_0(M,g_M,J_M)) \subseteq \mathfrak{iso}(Z, g)$ where 
the lower index $\bbS$ denotes invariant vector fields w.r.t. the circle action. 
It turns out that 
\begin{thm} \label{comp-iso}
If $Z$ is compact the restriction 
$$\Phi: \mathfrak{iso}_{\bbS}(N,g_N)\oplus \mathfrak{aut}_0(M,g_M,J_M)\to \mathfrak{iso}(Z, g)$$
is a Lie algebra isomorphism. 
\end{thm}
\begin{proof}
{\it{Step 1:}} $K$ acts trivially on $\mathfrak{iso}(Z,g)$.\\
Since $\pi^{\star}\omega_M \in \Lambda^{1,1}Z$ is closed we have that $\aut_0(Z,g,J) \subseteq \Ham^c(Z,\pi^{\star}\omega_M)$ by
part (ii) in Proposition \ref{ham-K}. Let $X \in \aut_0(Z,g,J)$; by Lemma \ref{hamc} it follows that $[K,X] \in \D_{+}$. On the other hand $[K,X]$ is Hamiltonian, $[K,X] \cntrct \omega=\di c$ with $c:Z \to \bbR$ hence $\di c$ vanishes on $\D_{-}$. Since $[X,Y]_{\D_+}=z^{-1}\omega(X,Y)K$  whenever $X,Y \in \D_{-}$ it follows that $(\di c)K=0$. In other 
words the momentum map $\omega([K,X],K)$ for $[K,[K,X]]$ vanishes identically.  Therefore $\mathrm{ad}_K^2X=[K,[K,X]]=0$ in the Lie algebra $\mathfrak{iso}(Z,g)$; the latter is reductive since it is the Lie algebra of a compact Lie group hence $[K,X]=0$. If $X \in \mathfrak{iso}(Z,g)$ the Lie bracket $[K,X] \in \aut_0(Z,g,J)$ must therefore commute with $K$ thus $[K,X]=0$ by the same argument as above.\\
{\it{Step 2:}} $\Phi$ is surjective.\\
Let $X \in \mathfrak{aut}(Z,g,J)$. Since $[K,X]=0$ and $X$ is holomorphic Theorem \ref{holvf} yields   
$$ X=\Phi(V+Y)=V_{+}+(Y_{-}+(f_1 \circ \pi) K+(f_2 \circ \pi)JK)
$$
with $(V,Y) \in  \mathfrak{aut}^{\bbS}(N,J_N)\oplus \mathfrak{aut}_0(M,J_M)$,
according to the definition of the map $\Phi$. Let $f_Y=f_1-if_2$ be the normalised holomorphy potential for $Y$, which thus satisfies $Y \cntrct \omega_M=\di f_1-J_M \di f_2$. 
From the definition of the map $\iota_{\om_M}$ we get $ X \cntrct \om=\iota^1_{\om_M}(\alpha_X)$ where 
$$\alpha_X=V \cntrct \om_N+z_N(Y \cntrct \om_M)+f_1\di z_N-f_2 J_N\di z_N \in \Lambda^1(M \times N).$$
Then $\di (X \cntrct \om)=0$ is equivalent with $\di_{\mkern+1.8mu\om_M}\alpha_X=0$; in expanded form this reads 
\begin{equation} \label{eq-k}
\di \alpha_X=\alpha_X(K_N)\omega_M.
\end{equation}
Because $\L_{K_N}\alpha_X=0$ contracting with $K_N$ leads to $\di (\alpha_X(K_N))=0$, which forces $\alpha_X(K_N)=0$ since $[\om_M] \neq 0 $ in $H^2_{dR}(M \times N)$. From the definition of $\alpha_X$ this entails $\om_{N}(V,K_N)-x_N f_2=0$. Differentiating in direction of $M$ we get further $x_N \di f_2=0$ thus 
$f_2=0$ by taking into account that $K_N$ is not identically zero and $\int_Mf_2=0$. We have showed that $Y \in \aut_0(M,g_M,J_M)$; as $f_2$ vanishes 
we have $\alpha_X=V \cntrct \om_N+\di(z_Nf_1)$ thus \eqref{eq-k} reduces to $\di(V \cntrct \om_N)=0$ showing that $V \in \mathfrak{iso}_{\bbS}(N,g_N)$. This completes the proof.
\end{proof}

Recall that  a compact complex manifold with positive first Chern class admits a K\"ahler metric of constant scalar curvature only if the Lie algebra of its holomorphic vector fields is reductive (\cite{lichne}). It follows that the above structure results can be used to determine explicit obstructions for the existence of constant scalar curvature metrics compatible with $J$.

\begin{thm} \label{csc-obs}
Assume that the Lie algebra $\aut(Z,J)$ is reductive. Then
\begin{equation*}
H^0(M,L^{-k_p}) \otimes \h_p=0 \ \mbox{for all} \ 1 \leq p \leq d \ \mbox{such that} \ k_p<0.
\end{equation*}
In particular, if $H^0(M,L) \neq 0$ the algebra $\aut(Z,J_Z)$ is reductive if and only 
\begin{itemize}
\item[(i)] the action of $\bbS$ on $(N,J_N)$ has only positive weights
\item[(ii)] the Lie algebras $\h_0$ and $\aut_0(M,J_M)$ are reductive.
\end{itemize}
In this case $\aut(Z,J)=\h_{0} \oplus \aut_0(M,J_M)$, a direct sum of Lie algebras.
\end{thm}
\begin{proof}
Let $p_0$ be the greatest $1 \leq p \leq d$ such that $k_p<0$. Then $[\h_{p},\h_{p_0}]=0$ for all $1 \leq p \leq d$ such that $k_p<0$ and 
$[\h_0,\h_{p_0}] \subseteq \h_{p_0}$. By Theorem \ref{holvf} and Proposition \ref{e} it follows that $\ker(\del^{p_0}) \otimes \h_{p_0}$ is an abelian ideal in  $\aut(Z,J)$ thus contained in its center. As $K_N \in \h_0$ acts on 
$\ker(\del^{p_0}) \otimes \h_{p_0}$ by multiplication with $ik_{p_0}$ we conclude that $\ker(\del^{p_0}) \otimes \h_{p_0}=0$. The claim follows now by induction over the number of negative weights in $\aut(N,J_N)$. Finally, having $H^0(M,L) \neq 0$ forces $\h_{p}=0$ whenever $k_p<0$.
\end{proof}

\subsection{Non-uniqueness of the foliation $\D_+$} \label{uni}
Let $(Z,g,J)$ be K\"ahler and equipped with a holomorphic, globally homothetic foliation $\F$ with Lee form $\theta$. When $\F$ is regular, that is when the leaf space $M:=Z \slash \F$ is a smooth manifold we obtain a smooth submersion $\pi:Z \to M$. The holomorphy of $\F$ entails that $M$ comes equipped with a complex structure $J_M$ such that 
\begin{equation*}
\pi:(Z,J) \to (M,J_M)
\end{equation*}
is holomorphic. Choosing $z:Z \to (0,\infty)$ with $\theta=\di \ln z$ makes that $\D_{+}=\ker (\di \pi)$ defines a Riemannian foliations w.r.t. $z^{-1}g$, in particular $M$ carries a Riemannian metric $g_M$ such that 
\begin{equation*}
g_{\vert \D_{-}}=z \pi^{\star}g_M.
\end{equation*}
Clearly $(g_M,J_M)$ is K\"ahler, with K\"ahler form $\omega_M=g_M(J_M \cdot, \cdot)$ satisfying $\omega_{-}=z\pi^{\star}\omega_M$.

In the setup we investigate the extent up to which homothetic foliations 
compatible with the K\"ahler structure $(g,J)$ must be unique. Whenever 
$E \subseteq TM$ is a distribution we indicate with $\pi^{\star}E \subseteq \D_{-}$ 
its horizontal lift.
\begin{thm} \label{multi-ruled}
Assume that $(M,g_M,J_M)$ carries a
conformal foliation $\F_M$ with complex vertical distribution $\D_{+}^M$, 
Lee form $\theta_M$ and horizontal distribution $\D_{-}^M$. Decompose 
\begin{equation} \label{enlarge1}
\T Z=(\D_{+} \oplus\ \pi^{\star}\D_{+}^M) \oplus\ \pi^{\star}\D_{-}^M.
\end{equation}
Then 
\begin{itemize}
\item[(i)] the distribution $\D_{+} \oplus\ \pi^{\star}\D_{+}^M$ is tangent 
to the leaves of a conformal foliation with Lee form 
$\theta+\pi^{\star}\theta_M$;
\item[(ii)] $\D_{+} \oplus\ \pi^{\star}\D_{+}^M$ is holomorphic if and only if  $\ \D_{+}^M$ is holomorphic in $\T M$;
\item[(iii)] $\D_{+} \oplus\ \pi^{\star}\D_{+}^M$ is totally geodesic if and only if  $\ \D_{+}^M$ is totally geodesic in $\T M$ and the intrinsic torsion satisfies $\eta_{\D_{+}}( \pi^{\star}\D_{+}^M)=0$.
\end{itemize}
\end{thm}
\begin{proof}
From the assumptions made the splitting \eqref{enlarge1} is $g$-orthogonal and $J$-invariant. Since $\pi:(Z,g)\to (M,g_M)$ is a conformal submersion with dilation factor $z$, we have (e.g. \cite[Proposition 2.1.15 (ii)]{Wood}) 
\begin{equation} \label{lie3}
[X_{-},Y_{-}]=[X,Y]_{-}-z(\pi^{\star}\omega_{M})(X,Y) J\zeta
\end{equation}
whenever $X,Y \in \T M$. Here $X_{-}$ denotes the horizontal lift to $\D_{-}$ of the vector field $X \in \T M$. A quick verification using 
Koszul's formula, \eqref{lie3} and $\di z(\D_{-})=0$ shows that 
\begin{equation} \label{KoKo}
(\nabla^g_{X_{-}}Y_{-})_{\D_-}=(\nabla^{g_M}_XY)_{-}.
\end{equation}
(i) As $X_{-}$ is basic, that is $[\D_{+},X_{-}] \subseteq \D_{+}$, the integrability of $\D_{+} \oplus\ \pi^{\star}\D_{+}^M$ is a consequence of having 
$\D_{+}^M$ integrable, \eqref{lie3} and $\zeta \in \D_{+}$. Pick vector fields $Y,Z \in \D_{-}^M$ and $X \in \D_{+}^M$; using \eqref{KoKo} we compute 
\begin{equation*}
\begin{split}
(\L_{X_{-}}g)(Y_{-},Z_{-})=&g(\nabla^g_{X_{-}}Y_{-},Z_{-})+g(Y_{-},\nabla^g_{X_{-}}Z_{-})\\
=&g((\nabla^{g_M}_XY)_{-},Z_{-})+g(Y_{-},(\nabla^{g_M}_XY)_{-})=z\pi^{\star}\L_Xg_M(Y,Z)\\
=&z \pi^{\star}(\theta_M(X)g_M(Y,Z))=(\pi^{\star}\theta_M)(X_{-})g(Y_{-},Z_{-})
\end{split}
\end{equation*}
since $\D_{+}^M$ is conformal. If $V \in \D_{+}$ we have $(\L_Vg)(Y_{-},Z_{-})=\theta(V)g(Y_{-},Z_{-})$. These considerations prove the claim by taking into account that $\theta(\pi^{\star}\D_{+}^M)=0$ together with $(\pi^{\star} \theta_M)(\D_{+})=0$.\\
(ii) Since $\D_{+} \oplus \ \pi^{\star}\D_{+}^M$ is integrable the holomorphy of the distribution $\D_{+} \oplus\  \pi^{\star}\D_{+}^M$ amounts to having 
$(\L_{\D_{+} \oplus\ \pi^{\star}\D_{+}^M}J)\pi^{\star}\D_{-}^M \subseteq \D_{+} \oplus\ \pi^{\star}\D_{+}^M$. As $\D_{+}$ is holomorphic this is equivalent to 
the inclusion $(\L_{\pi^{\star}\D_{+}^M}J)\pi^{\star}\D_{-}^M \subseteq \D_{+} \oplus\ \pi^{\star}\D_{+}^M$. From \eqref{lie3} it is easy to see this happens if and only 
if $\D_{+}^M$ is holomorphic in $\T M$.\\
(iii) Pick $X \in \D_{+}^M$ and $V \in \D_{+}$. Then $g(\nabla^g_{X_{-}}V,Y_{-})=-g(V,\nabla^g_{X_{-}}Y_{-})=0$ for all $Y \in \D_{-}^M$ by using \eqref{d1}. Note that in the latter $\Psi=0$ since $\D_{+}$ is holomorphic. It follows that $\nabla^g_{X_{-}}V$ belongs to $\D_{+} \oplus 
\pi^{\star}\D_{+}^M$. Because $[V,X_{-}] \in \D_{+}$ we get that $\nabla^g_{X_{-}}V \in \D_{+} \oplus \
\pi^{\star}\D_{+}^M$ as well. Consequently, the latter distribution is totally geodesic if and only if  
\begin{equation*}
\nabla^g_VW \in \D_{+} \oplus\  
\pi^{\star}\D_{+}^M, \qquad \nabla^g_{X_{-}}Y_{-} \in \D_{+} \oplus\  
\pi^{\star}\D_{+}^M
\end{equation*}
whenever $V,W \in \D_{+}$ and $X,Y \in \D_{+}^M$. The first requirement is equivalent to having $\eta_{\D_{+}}\pi^{\star}\D_{+}^M=0$ whilst the second amounts to 
$\D_{+}^M$ totally geodesic in $\T M$ by \eqref{KoKo}.
\end{proof}
\begin{rem} \label{proof2}
When $\D_{+}^M$ is holomorphic, a somewhat shorter proof of (i) in Theorem \ref{enlarge1} is available. The K\"ahler form 
of $(g,J)$ splits orthogonally as 
\begin{equation*}
\om=(\om_{+}+z\pi^{\star}\om_{+}^M)+z\pi^{\star}\om_{-}^M.
\end{equation*}
according to \eqref{enlarge1}. We compute 
\begin{equation*}
\begin{split}
\di(\om_{+}+z\pi^{\star}\om_{+}^M)&=-\theta \wedge \om_{-}-z\pi^{\star}\theta_M \wedge 
\pi^{\star} \om_{-}^M+\di z \wedge \pi^{\star} \om_{+}^M\\
&=-(\di\ln z+\pi^{\star}\theta_M) \wedge z\pi^{\star}\om_{-}^M
\end{split}
\end{equation*}
by using the structure equations for $\D_{+}$ respectively $\D_{+}^M$ together with $\theta=\di \ln z$ and $\om_{-}=z\pi^{\star}\om_{+}^M+z\pi^{\star}\om_{-}^M$. As 
$\di\ln z+\pi^{\star}\theta_M$ vanishes on $\pi^{\star}\D_{-}^M$ the claim follows from Proposition \ref{cdiff}.\\
\end{rem}
This non-uniqueness result has a nice geometric interpretation in case 
the foliation on the base space is totally geodesic. Indeed we have the 
following two main classes of examples illustrating the multi-fibered structure 
$Z$ may have. 

\begin{propn} \label{no+cpx}
	Let $(M,g_M,J_M)=(M_1,g_{M_1},J_{M_1}) \times (M_2,g_{M_2},J_{M_2})$,  equipped with the polarisation given by $L=\mathrm{pr}_1^{\star}L_1 \otimes \mathrm{pr}_2^{\star}L_2$, where $L_i \to M_i$ are the polarisations of $M_i$, $i=1,2$. The K\"ahler structure $(g,J)$ obtained by the Weinstein construction on $Z=Z(N,M,L)$ has the following properties
	\begin{itemize}
		\item[(i)] it carries $3$ distinct $\mathrm{TGHH}$-foliations given by $\D_{+}, \D_{+} \oplus \pi^{\star}\T M_1, \D_{+} \oplus \pi^{\star}\T M_2$ with the same Lee form. The last two foliations correspond to explicit fibrations $Z(N,M_i,L_i)\hookrightarrow Z\xrightarrow{\pi_i} M_i, i=1,2$
		\item[(ii)] the metric $g$ admits $3$ distinct orthogonal complex structure $I_1,I_2,I_3$ obtained as in Proposition \ref{tg}. Those mutually commute and $(g,I_k), k=1,2,3,$ are not K\"ahler.
	\end{itemize} 
\end{propn}
\begin{proof} 
(i) We apply Theorem \ref{multi-ruled} to the product  foliation, $\D_{+}^M=\T M_1, \D_{-}^M=\T M_2$, for which, in particular, $\theta_M=0$. It only remains to determine the fibrations $\pi_i$. Choose Hermitian metrics $h_i$ in $L_i,i=1,2$; with respect to the canonical product metric the sphere bundle 
$P$ of $L$ is identified to $(P_1 \times P_2) \slash \mathbb{S}^1$ where 
$P_i=S(L_i, h_i)$  and $\mathbb{S}^1$ acts on 
$P_1 \times P_2$ according to $((p_1,p_2),\lambda) \mapsto (R^{P_1}_\lambda p_1,R^{P_2}_{\lambda^{-1}}p_2)$. Recall that $Z:=P \times_{\bbS} N$ where $(N,g_N,J_N)$ is K\"ahler and equipped with a Hamiltonian 
$\bbS$ action. The above identification for $P$ makes it straightforward to check that one can naturally identify 
$$Z\simeq (P_1 \times P_2 \times N) \slash \mathbb{T}^2$$ 
w.r.t. the free 
$\mathbb{T}^2=\bbS \times \bbS$ action on $P_1 \times P_2 \times N$ given by 
\begin{equation*} 
(p_1,p_2,n)(\lambda_1, \lambda_2):=(R^{P_1}_{\lambda_1} p_1,R^{P_2}_{\lambda_1^{-1}\lambda_2}p_2,R^N_{\lambda_2^{-1}}n).
\end{equation*}
Identifying 
$N \times (P_1 \times P_2) \simeq (N \times P_1) \times P_2 \to (N \times P_2) \simeq P_1$ via the map given by $(n,(p_1,p_2)) \mapsto ((n,p_1),p_2) \mapsto ((n,p_2),p_1)$
induces further identifications 
\begin{equation*}
Z\simeq (N \times_{\mathbb{S}^1} P_1) \times_{\mathbb{S}^1} P_2\simeq (N \times_{\mathbb{S}^1} P_2) \times_{\mathbb{S}^1} P_1.
\end{equation*}
Here the circle acts on $P_k \times_{\bbS} N, k=1,2$, according to \ref{act3}.
In other words, for $i=1,2$, the foliations $\D_{+} \oplus \pi^{\star}\T M_i$ constructed in Theorem \ref{multi-ruled} correspond to the fibrations 
\begin{equation}\label{pic2}
\begin{tikzcd}
&Z(N,M_2,L_2) \arrow[hookrightarrow,d]\\
Z(N,M_1,L_1) \arrow[hookrightarrow,r] &Z \arrow[d,swap,"\pi_1"] \arrow[r,"\pi_2"] &M_2. \\
&M_1\\ 
\end{tikzcd}
\end{equation}
The projection maps $\pi_i:Z \to M_i,i=1,2$ are induced by the bundle projection maps $\pi_{P_i}:P_i \to M_i, i=1,2$.\\
(ii) follows by using Proposition \ref{tg} for the 3 foliations constructed above. Note that ``distinct'' here means $I_j\neq \pm I_k$ for all $1\leq j\neq k\leq 3$. 
\end{proof}

\begin{rem}
	It is an open problem to determine the number of complex structures orthogonal with respect to a given Riemannian metric, see \cite{salamon}. Explicit examples were constructed on Lie groups, where this number is large, see \cite{joyce}.
\end{rem}

In the second class of examples below the geometric origin of the polarisation 
on the base  $M$ is of a different nature.

\begin{propn} \label{pol-b}
Let $M_1=Z(N_0,M_0,L_0)$ be obtained by the Weinstein construction with  $(N_0, l_0)$  polarised. Let  $(N_1,g_{N_1},J_{N_1})$ be  a K\"ahler manifold with a Hamiltonian circle action and build 
$Z_1=Z(N_1,M_1,L_1)$ where $L_1=(l_0)_{+}$ is the polarisation of $M_1$ constructed in Proposition \ref{picard}. Then
\begin{itemize}
	\item[(i)] $Z_1$ is obtained from the Weinstein construction for the fibre $Z(N_1, N_0, l_0)$ and  polarised base $(M_0, L_0)$, i.e. $Z_1=Z(Z(N_1, N_0, l_0), M_0, L_0)$
	\item[(ii)] the K\"ahler structure $(g,J)$ on $Z_1$ obtained by the Weinstein construction carries $2$ distinct $\mathrm{TGHH}$-foliations given by $\D_{+}$, $\D_{+} \oplus \pi^{\star}\D^{M_1}_1$  with  Lee forms $\theta$, respectively $\theta+\pi^\star\theta_{M_1}$. 
\end{itemize} 
\end{propn}
\begin{proof} (i) Let $P_0 \to M_0$ respectively $Q_0 \to N_0$ be the sphere bundles in $L_0$ respectively $l_0$. From the construction of the polarising 
bundle  in Proposition \ref{picard} applied for $M_1$ it follows that we can identify 
\begin{equation*} 
Z_1=(P_0 \times Q_0 \times N_1) \slash \mathbb{T}^2
\end{equation*}
 w.r.t. the free 
$\mathbb{T}^2=\bbS \times \bbS$ action on $P_0 \times Q_0 \times N_1$ given by 
\begin{equation}\label{tw-it2}
(p_0,q_0,n_1)(\lambda_1, \lambda_2):=(R^{P_0}_{\lambda_1}p_0,\ (\widetilde{R}_{\lambda_1^{-1}}\circ R^{Q_0}_{\lambda_2})q_0,\ R^{N_1}_{\lambda_2^{-1}}n_1).
\end{equation}
Here $(\widetilde{R}_{\lambda})_{\lambda \in \bbS}$ denotes, according to Proposition \ref{lift}, the lift of the Hamiltonian action of $\bbS$ on $N_0$ to $Q_0$. The explicit form of the action \eqref{tw-it2} makes it easy to check that we can further identify $Z_1=P_0 \times_{\bbS}Z(N_1,N_0,l_0)$ where 
the circle action on $Z(N_1,N_0,l_0)=Q_0 \times_{\mathbb{S}^1} N_1$ is induced from $((q_0,n_1),\lambda) \mapsto (\widetilde{R}_{\lambda}q_0,n_1)$. Note that the latter action is exactly the lift of the $\bbS$-action on $N_0$ to $Z(N_1,N_0,l_0)$.\\
(ii) To see that the foliation 
$\D_{+} \oplus \pi^{\star}\D_{+}^{M_1}$ has the required property, apply Theorem \ref{multi-ruled} to the canonical foliation $\D_+$ on $M_1$. It corresponds to the vertical fibration in 
\begin{equation}\label{pic3}
\begin{tikzcd}
&Z(N_1,N_0,l_0) \arrow[hookrightarrow,d]\\
N_1 \arrow[hookrightarrow,r] &Z_1 \arrow[d,swap,"\pi_1"] \arrow[r,"\pi"] &M_1. \\
&M_0\\ 
\end{tikzcd}
\end{equation}
\end{proof}
\begin{rem} 
	(i) Proposition \ref{pol-b} also explains what happens when the Weinstein construction is iterated. Above we have performed only a stage-2 iteration, as shown in the following diagram:
	\begin{equation}\label{diagramaa}
\begin{tikzcd}
N_1 \arrow[hookrightarrow,d] & N_1 \arrow[hookrightarrow,d]\\
Z(N_1,N_0,l_0) \arrow[d] \arrow[hookrightarrow,r] &Z_1 \arrow[d] \arrow[r] &M_0 \\
N_0 \arrow[hookrightarrow,r] & M_1=Z(N_0,M_0,L_0)\arrow[r] \arrow[d] &M_0 \\
&M_0\\ 
\end{tikzcd}
\end{equation}
However the 
procedure can be continued to construct towers of fibrations, prolongating the middle column above by iterating $n$-times.

(ii) A special instance of iteration in the Weinstein construction occurs in the case of Bott manifolds, see \cite{bct}.
\end{rem}

Yet another construction of different homothetic and holomorphic foliations compatible with the same metric relies on projecting such an object, under 
suitable compatibility conditions, from the fibre $N$ to $Z$. If $\T N=\D^1\oplus \D^2$ with $\D^i$, $i=1,2$, invariant for the $S^1$-action, then $\D^i$ is contained in the horizontal space $\tilde \H$ of the circle fibration $\tilde p : P\times N\to Z$, hence they project on $Z$ into $\D^i_+=\di \tilde p (\D^i)$. 
\begin{thm} \label{unique2}
Let $(N,g_N,J_N)$ be equipped with an homothetic and holomorphic foliation $\D^1$, 
with horizontal space $\D^2$ and Lee form $\theta^N$. Assume that 
\begin{itemize}
\item[(a)] $\D^1$ is circle invariant and
\item[(b)] we have $(\di \ln z_N)_{\vert \D^1}=\theta^N$,
\end{itemize}
and decompose 
\begin{equation} \label{shrink2}
\T Z=\D^1_{+} \oplus\ (\D^2_{+} \oplus \D_{-}).
\end{equation}
Then, with respect to the metric $g$,
\begin{itemize}
\item[(i)] the distribution $\D^1_{+} $ is tangent 
to the leaves of an holomorphic and homothetic foliation with Lee form 
$\iota_{{\om_M}}(\theta^N)$ (see \eqref{tw_forms}).
\item[(ii)] if $\ \D^1$ is totally geodesic w.r.t $g_N$, so is $\D_{+}^1 $.
\end{itemize}
\end{thm}
\begin{proof}
 That $\D^{1}_{+}$ defines a holomorphic foliation w.r.t. $J$ as well as part (ii) is proved by arguments similar to those used in Theorem \ref{multi-ruled}. To show that $\D^1_{+}$ is homothetic 
we split the K\"ahler forms of $(g_N,J_N)$ respectively $(g,J)$ according to 
$\om^N=\om^1+\om^2$ respectively
\begin{equation*}
\om=\iota_{\om_M}(\om^1)+(\iota_{\om_M}(\om^2)+\om_{-}).
\end{equation*}
Using (iii) in Proposition \ref{deomega} and the structure equation 
$\di \om^1=-\theta^N \wedge \om^2$ we compute
\begin{equation*}
\begin{split}
\di \iota_{\om_M}\om^{1}=-\iota_{\om_M}(\theta^N) \wedge \iota_{\om_M}(\om^2)-\pi^{\star}\om_M \wedge \iota_{\om_M}(K_N \cntrct \om_1).
\end{split}
\end{equation*}
The assumption in (b) guarantees that $z_N^{-1}K_N \cntrct \omega^1=\theta^N$ and the claim is proved. 
\end{proof}
This leaves quite some freedom in constructing new examples. Indeed, take  
$N=Z(N_0,M_0,l_0)$ equipped with its canonical totally geodesic foliation, the canonical $\mathbb{S}^1$-action as defined in \eqref{act3} certainly satisfies (a) and (b) above\footnote{and then $Z(N,M,l)$ fibers over $M_0 \times M$.}. However one can determine {\it{all}} circle actions on $N$ satisfying these assumptions and obtain many more examples. This is suggested by Example \ref{pol-b} where the circle action on $N$ is lifted from $M_0$.

\smallskip

 Directly from Theorem \ref{comp-iso} we derive:
\begin{propn} \label{act5} 
Assume that $N=Z(N_0,M_0,l_0)$ is equipped with its canonical foliation $\D^1$. Then any circle action $\bbS \subseteq \mathrm{Iso}(N,g_N)$ satisfying 
(a) and (b) in Theorem \ref{unique2} has tangent vector field of the form 
$$tK \ \mbox{or} \ Y_{-}+z_YK,$$
where $Y \in \mathfrak{iso}(M_0,g_{M_0})$ is Hamiltonian vector field with momentum map $z_{Y}$, and $t \in \bbR$.
\end{propn}

\section{Applications in Riemannian and Hermitian geometry}

We describe several implications of the special foliated structure of a K\"ahler manifold obtained by the Weinstein construction in Riemannian and Hermitian geometries.

\subsection{Holomorphic harmonic morphisms}\label{ham}

Recall that a harmonic morphism is a smooth map $f:(M_1,g_1)\rightarrow (M_2,g_2)$ between Riemannian manifolds, such that $\phi\circ f$ is a (local) harmonic map on $(M_1,g_1)$ for all  (local) harmonic maps $\phi$ on $(M_2,g_2)$. It was independently proven by Fuglede and Ishihara that this is equivalent with $f$ being a harmonic map which is horizontally weakly conformal (we refer to \cite{Wood} for details about harmonic morphisms).
Harmonic morphisms thus produce conformal foliations off the critical point set. Conversely, one can look for conditions under which given foliations produce harmonic morphisms. It was proven in \cite[Corollary 2.6]{bg} that a minimal, conformal  foliation  locally  produces harmonic morphisms. Thus 

\begin{propn}
	The  holomorphic submersion $\pi:(Z,J) \to (M,J_M)$ with totally geodesic fibres is a harmonic morphism from $(Z,g) \to (M,g_M)$.
\end{propn}

\begin{rem}
	\begin{itemize}
		\item[(i)] 	Both $Z$ and $M$ can be choosed compact and the dimension of the fibers respectively the codimension  are arbitrary.   As far as we know, up to now no examples of holomorphic harmonic morphisms from compact K\"ahler manifolds, with arbitrary fibre dimension and codimension , were known. 
		\item[(ii)] It was already observed in \cite{BLMS} that the Calabi construction on line bundles yields holomorphic harmonic morphisms with fibres of complex dimension 1.
		\end{itemize}
\end{rem}

\subsection{Conformal submersions with complex fibres} \label{non-h}
Based on the structure result in Theorem \ref{tg-main} we construct in this section examples of totally geodesic, non-holomorphic foliations 
on compact K\"ahler manifolds. This recipe actually covers all local examples. 

Start with a compact K\"ahler manifold $(M,g_M,J_M)$ which carries a totally geodesic, Riemannian foliation  with complex leaves. Let $\T M=
\D_{+}^M \oplus \D_{-}^M$ be the associated splitting (that we assume not to be $\nabla^{g_M}$-parallel, i.e. not giving rise to a local Riemannian product).  We also assume that $[\omega_M] \in H^2(M,\mathbb{Z})$ thus $M$ is polarised by a circle bundle 
$P \to M, c_1(P)=[\omega_M]$.

In addition we consider a K\"ahler manifold $N:=Z(N_0,M_0,l_0)$ obtained by the Weinstein construction. This carries its canonical homothetic foliation 
$$\T N=\D_{+}^N \oplus \D_{-}^N
$$
with Lee form $\theta^N=\di \ln z_N$. The Hamiltonian circle action $\bbS \subseteq \Aut_0(N,g_N,J_N)$ induced by the action on $N_0$(see \eqref{act3}) satisfies the compatibility condition 
in Theorem \ref{unique2}, (b). 
With respect to this circle action we form the K\"ahler manifold
\begin{equation} \label{ex-no3}
Z:=P \times_{\bbS} N.
\end{equation}
The canonical foliation on $Z$ comes from the splitting $\T Z=\D_{+} \oplus \D_{-}$. Using Theorem \ref{unique2}, we push forward the canonical foliation of $N$ to $Z$ and decompose $\D_+=\D_+^1\oplus\D_-^2$. Moreover, by Theorem \ref{multi-ruled}, we pull-back the foliation $\D_{+}^M$ from $M$ to $Z$ and obtain $\D_-=\D_+^2\oplus\D_-^2$. 

\begin{thm} \label{tg-no3}
On the K\"ahler manifold $(Z,g,J)$ described in \eqref{ex-no3}, consider the splitting 
$$\T Z=(\D_+^1\oplus\D_+^2)\oplus (\D_-^1\oplus\D_-^2).$$ 
Then
\begin{itemize}
\item[(i)] The distribution $\D_+^1\oplus\D_+^2$ induces a complex, totally geodesic, homothetic foliation which is not holomorphic.
\item[(ii)] The canonical deformation $(h,I)$ is a non-integrable almost Hermitian structure of type $\mathcal{G}_1$, see Proposition \ref{g1-f}. 
\item[(iii)] All K\"ahler manifolds supporting a totally geodesic homothetic foliation with complex leaves are obtained locally by the above construction.
\end{itemize}   
\end{thm}
\begin{proof}
(i) The foliated structure of $Z$ is implied by combining Theorems \ref{multi-ruled} and \ref{unique2}. If the foliation were holomorphic, part (ii) in  Theorems \ref{multi-ruled} leads to $\D_+^M$  holomorphic in $\T M$. Since $\D_+^M$ is already totally geodesic and Riemannian, it will be $\nabla^g$-parallel, contradicting our assumption.  
(ii) follows from (i) and Proposition \ref{g1-f}, while 
(iii) has been established during the proof of Theorem \ref{tg-main}.
\end{proof}

\begin{rem}
(i) In particular, when all manifolds involved in the above construction are compact, then $Z$ is compact too, therefore we obtain examples of totally geodesic, conformal and non-holomorphic submersions from compact K\"ahler manifolds. 

(ii) Complete K\"ahler manifolds $(M,g,J)$ supporting a totally geodesic, Riemannian foliation with complex leaves have been classified in \cite{Na1}: up to a finite cover and Riemannian products, they are twistor spaces of quaternion K\"ahler manifolds of positive scalar curvature or certain classes of compact homogeneous K\"ahler manifolds. In the latter case, the fibrations producing the desired Riemannian foliations have been fully described in \cite{span}.
\end{rem}
Based on these we illustrate how factorisation 
works for the foliation constructed in Theorem \ref{tg-no3}, (i). Let $g_{FS}$ and $g_{qK}$ be the canonical metrics of the complex, respectively quaternionic projective spaces, normalised such that the Hopf fibration $\mathbb{P}^{2n+1} \xrightarrow{\ \pi_{h}\ }\mathbb{H}P^n$ is a Riemannian submersion.
\begin{propn} \label{fact-ex}
In Theorem \ref{tg-no3}, assume $(M,g_M)=(\mathbb{P}^{2n+1}, g_{FS})$ is polarised by the canonical class. Then:
\begin{itemize} 
\item[(i)] The foliation $\D_{+}^1 \oplus \D_{+}^2$ is tangent to the fibres of a conformal submersion 
$Z \xrightarrow{\pi_0} M_0 \times \mathbb{H}P^n$. The base is equipped with product metric $g_{M_0}+g_{qK}$ and the fibres of $\pi_0$ are $Z(N_0, \mathbb{P}^1, \mathcal{O}(-1))$;
\item[(ii)] $\pi_0$ factorises according to the commutative diagram 
\begin{equation*}
\begin{tikzpicture}

\node (N11) at (-1.5,0) {$N_0$};

\node (N12) at (0,0) {$Z$};

\node (N13 anchor) at (2.8,-.15) {};
\node (N13) at (3,0) {$M_0 \times \mathbb{P}^{2n+1}$};

\node (N23 anchor) at (2.8,-1.35) {};
\node (N23) at (3,-1.5) {$M_0 \times \mathbb{H}P^{n}$};

\draw[right hook-stealth] (N11) -- (N12);

\draw[-stealth] (N12) -- (N13) node[midway,above] {\small $\pi_1$};

\draw[-stealth] (N13 anchor) -- (N23 anchor) node[midway,right] {\small $1_{M_0} \times \pi_h$};

\draw[-stealth] (N12) -- (N23) node[midway,below left] {\small $\pi_0$};
\end{tikzpicture}
\end{equation*}
where the submersion $\pi_1$ is holomorphic, totally geodesic and conformal, whilst the submersion $1_{M_0} \times \pi_{h}$ is totally geodesic and Riemannian, with complex fibers.
\end{itemize}
\end{propn}

\begin{proof} 
 By \eqref{pic2} we already have a holomorphic and conformal submersion
$$N_0\hookrightarrow (Z,g, J) \xrightarrow{\ \pi_1\ } M_0\times \mathbb{P}^{2n+1},$$
where the base is equipped with the product metric and complex structure. Let $\pi_0$  be the composition $(1_{M_0}\times \pi_{h})\circ \pi_1$. It is  conformal since $\pi_1$ is conformal and $\pi_h$ is Riemannian. The distribution $\D_{+}^1 \oplus \D_{+}^2$ is by construction tangent to the direct sum $\ker \di \pi_0=\ker \di \pi_1\oplus \pi^\star_1(\T \mathbb{P}^1)$.   To prove both claims there remains to determine the fibres of $\pi_0$. From the diagram we see that the fibres  $\pi_0^{-1}(\{m_0,q\})=\pi_1^{-1}\left(\{m_0\}\times \pi_h^{-1}(q)\right).$ 
Recall that $Z=P\times _{\bbS} N_0$, where $P$ is the tensor product of the polarisations of $M_0$ and $\mathbb{P}^{2n+1}$. The latter corresponds to the circle fibration $\bbS\hookrightarrow \mathbb{S}^{4n+3}\xrightarrow{\ p_h\ } \mathbb{P}^{2n+1}$. As $p_h^{-1}(\mathbb{P}^1)\simeq \mathbb{S}^3$, the identification being equivariant when $\mathbb{S}^3$ is acted on by $\bbS$ with quotient $\mathbb{P}^1$, the claim follows.
\end{proof}

\begin{rem}
	More generally, in the above statement one can replace the Hopf fibration by the twistor fibration of any quaternion K\"ahler manifold of positive scalar curvature.
\end{rem}
\begin{rem}
	From Theorem \ref{tg-no3} (ii) and Proposition \ref{fact-ex} we obtain  examples of conformal submersions from $\mathcal{G}_1$ structures, namely  $\pi_0:(Z,h,I)\to M_0\times\mathbb{H}P^n$. Note that here the canonical variation is taken w.r.t. the distribution $\D^1_+\oplus\D^2_+$ tangent to the fibres of $\pi_0$. Moreover, the almost Hermitian structure induced by $(h,I)$ on the fibres is K\"ahler.  
\end{rem}

\subsection{Hermitian geometry on $Z$} \label{I-subs}
By (ii) in Proposition \ref{kaehler} that the almost complex structure $I$ given by 
$$ I=-J \ \mbox{on} \ \D_{+}, \ I=J \ \mbox{on} \ \D_{-}
$$
is integrable and orthogonal w.r.t. metric $g$ of $Z$.

In this section we show that $I$ admits at least two types of compatible Riemannian metrics with distinguished geometric properties. Some of these metrics are conformal with the K\"ahler metric $g$. The problem of finding conformal metrics with different or equal special holonomies w.r.t. different complex structures was already discussed in the literature, e.g.  \cite{BLMS}, \cite{MMP}. 
Recall now:

\begin{defn} \label{bal}
A Hermitian structure with fundamental form $\Omega$ defined on a manifold $Z^{2d}, d \geq 3$, is called { balanced} if its Lee form vanishes, equivalently $\di \Omega^{d-1}=0$.
\end{defn}

The complex manifold $(Z,I)$ carries natural families of such metrics as 
showed below. Let $[a,b]$ be the range of the momentum map $z:Z \to \bbR$; whenever $\psi:[a,b] \to \bbR$ is smooth we indicate with $\psi(z)$ the composition $\psi \circ z:Z \to \bbR$. 
\begin{propn} \label{I}
Let $\varphi:[a,b] \to (0,\infty)$ be decreasing, $\varphi^{\prime}<0$, and such 
that $\varphi^{\prime \prime}$ is not identically zero. Let 
\begin{equation*}
\Omega_{\varphi}=\varphi^{\prime}(z)\om_{+}+z^{-1}\varphi(z)\om_{-}.
\end{equation*}
The Hermitian structure $\left((-\varphi^{\prime}(z))^{-\frac{n-1}{m+n-1}}\Omega_{\varphi},I\right)$ is balanced non-K\"ahler if $n \geq 2$ and K\"ahler if $n=1$.
\end{propn}
\begin{proof}
The proof follows in fact by direct computation, from the structure equations of the foliation $\D_{+}$. However, to illustrate how these structures were found 
consider the Hermitian structure $(\alpha,I)$ given by 
\begin{equation} \label{da-0}
\alpha=-A(z)\om_{+}+B(z) \om_{-}
\end{equation}
for some functions $A,B:[a,b] \to (0,\infty)$. From \eqref{curv-d} and \eqref{do-} we get 
\begin{equation} \label{da-1}
\di \alpha=C(z)\di z \wedge \om_{+}+D(z)\di z \wedge \om_{-}
\end{equation}
where $C(z)=-A^{\prime}(z)$ and $D(z)=B^{\prime}(z)+z^{-1}(A+B)$. Recall that $n=\dim_{\bbC}\D_+$ and $m=\dim_{\bbC}\D_-$. Assume that $n \geq 2$. Then  
\begin{equation*}
\begin{split}
\di \alpha \wedge \om_{+}^{n-2} \wedge \om_{-}^{m}&=C(z)\di z \wedge \om_{+}^{n-1} 
\wedge \om_{-}^m,\\
\di \alpha \wedge \om_{+}^{n-1} \wedge \om_{-}^{m-1}&=D(z)\di z \wedge \om_{+}^{n-1} 
\wedge \om_{-}^m, \quad \di \alpha \wedge \om_{+}^n=0.
\end{split}
\end{equation*}
At the same time, taking into account the vanishing of $(\omega_{+})^{i}, i \geq n+1$, respectively $(\omega_{-})^j, j \geq m+1$, shows that
\begin{equation*}
\begin{split}
 \alpha^{n+m-2}&=\binom{m+n-2}{n-2}(-A\omega_{+})^{n-2}\wedge(B\omega_{-})^m\\
 &+\binom{m+n-2}{n-1}(-A\omega_{+})^{n-1}\wedge (B\omega_{-})^{m-1}\\
 &+\binom{m+n-2}{n}(-A\omega_{+})^{n}\wedge (B\omega_{-})^{m-2}.
 \end{split}
 \end{equation*}

The balanced equation $\di \alpha \wedge
\alpha^{n+m-2}=0$ reads thus $(n-1)BC=mAD$ or equivalently 
\begin{equation*}
mAB^{\prime}+(n-1)BA^{\prime}+mz^{-1}A(A+B)=0.
\end{equation*}
Elementary considerations show that $B=z^{-1}A^{-\frac{n-1}{m}}\varphi(z)$ and 
$A^{\frac{m+n-1}{m}}=-\varphi^{\prime}(z)$ where $\varphi>0$ and $\varphi^{\prime}<0$. Clearly $\di \alpha=0$ if and only if $C=D=0$, which corresponds to $\varphi^{\prime \prime}=0$.\\
When $n=1$ we have $\di z \wedge \om_{+}=0$ hence the balanced equation reduces to 
$D=0$. Equivalently $d\alpha=0$ and it straightforward to check that $A=-\varphi^{\prime}, B=z^{-1}\varphi$.
\end{proof}
As far as the conformal class of the metric $g$ is concerned this yields: 
\begin{cor} \label{corI}
The Hermitian structure $(z^{-\frac{2m}{m+n-1}}g,I)$ is balanced non-K\"ahler if $n \geq 2$ and K\"ahler if $n=1$.
\end{cor}
\begin{proof}
Take $\varphi(z)=z^{-1}$ in Proposition \ref{I}.
\end{proof}

\begin{rem}
	Balanced metrics (also known as semi-K\"ahler in Gray-Hervella's classification, belonging to the class $\mathcal{W}_3$, \cite{gh}) appear naturally on some compact complex manifolds, e.g.  on twistor spaces of  conformally flat Riemannian manifolds, cf. \cite{mich},  \cite[Proposition 11]{gau_weyl}). Balanced metrics on twistor spaces are related to totally geodesic, holomorphic foliations of complex dimension $m(m-1)=\dim( \mathrm{SO}(2m)/\mathrm{U}(m))$. But these foliations are not homothetic. Our examples are thus of a different nature.
	
	 On the other hand, there exist compact  complex manifolds which do not admit any compatible balanced metric. Such is the case of  $S^5\times S^1$ endowed with a natural complex structure which makes the projection onto $\mathbb{P}^2$  holomorphic (see  \cite{fu}). 
\end{rem}

\section{Local and global characterisations of the Weinstein constructions} \label{abs-c}

Let $(Z^{2(m+n)}, g, J), m+n\geq 2$, be a compact connected K\"ahler manifold endowed with a 
complex foliation $\mathcal{F}$ with tangent leaf distribution $\D_{+}$, $\dim_\bbC \D_+=n$. We denote by $\D_{-}$ the orthogonal complement of $\D_{+}$ and assume that 
\begin{equation} \label{tgh}
\D_{+} \ \mbox{is holomorphic}
\end{equation}
and 
\begin{equation} \label{hom}
\D_{+} \ \mbox{is totally geodesic and homothetic with Lee form}  \ \theta.
\end{equation}
In this situation we show how to recover most of the data involved in the Weinstein construction.
\begin{rem} \label{t=0}
If $\theta=0$ the distributions $\D_\pm$ are parallel w.r.t. $\nabla^g$ hence $(Z,g,J)$ is locally a product of K\"ahler manifolds. Therefore, we shall assume in the sequel that $\theta$ does not vanish identically on $Z$.
\end{rem}
Because $Z$ is compact Proposition \ref{detect-k} provides a holomorphic Killing vector field $K$ with momentum map $z>0$, that is $K \cntrct  \omega=\di z$. 
Moreover $\theta=\di\ln z$. If $\zeta:=\theta^\sharp$, then  
\begin{equation} \label{kk}
J\zeta=-z^{-1}K.
\end{equation}

\subsection{The local splitting}\label{tls}
We will now use the existence of $K$ to  obtain the full classification of complex foliations satisfying \eqref{tgh} and \eqref{hom}. We first need to set up some preliminaries.  

The form $\frac{1}{z}\omega_{-}$ is closed by \eqref{do-} and thus, 
it defines a ray of cohomology classes  $\bbR^+c$, where 
\begin{equation} \label{co-class}
c:=[z^{-1}\omega_{-}] \in H^{1,1}(Z,\bbR).
\end{equation}
\begin{rem} \label{cn0}
Assume that $Z$ is compact. Then
$$(z^{-1}\omega_{-})^{m} \wedge \omega^n=\left(z^{m}\left( \begin{array}{cc}
m+n \\ n \end{array} \right)\right)^{-1}\omega^{m+n}.$$ By using Stokes' theorem it follows that the cohomology class $c$ does not vanish.
\end{rem}

\begin{defn} The ray defined above is  called the twist class of the foliation $\F$ and is denoted $\tw(\F)$.
\end{defn}

In the rest of this section we work under the assumption that the twist class
\begin{equation}
\tw(\F) \in \bbR^+ H^2(Z,\mathbb{Z}).
\end{equation}
If this happens, one can choose the momentum map $z$ in such a way that $c\in H^2(Z,\mathbb{Z})$. This assumption allows constructing, by Chern-Weil theory, a {\it{twist bundle}} for $Z$; this is a principal circle bundle $\mathbb{S}^1 \hookrightarrow Q \xrightarrow{\pi_Q}Z$ with $c_1(Q)=c$ and principal circle action denoted by $(R^Q_\lambda)_{\lambda \in \bbS}$. 

\begin{rem}\label{idea}
	Were $Z$ obtained by the Weinstein construction from $M$ and $N$, the above $Q$ coincides with the twist bundle of $M\times N$, see Definition \ref{def-tw}. Hence, as the twist construction is symmetric, see Remark \ref{symmetry}, the twist of $Z$ by $Q$ is precisely $(M\times N, J_M\times J_N,  z_Ng_M+g_N)$.  Our strategy will thus be to twist $Z$ using the bundle $Q$ and an appropriate $c$-Hamiltonian $\bbS$-action. The natural candidate for this action is the action of $\bbR$ on $Z$ induced by the flow of $K$. However, two difficulties arise: (1) this flow is a priori not periodic, and (2) even in the periodic case, one has to show that the twist is smooth, i.e. the lifted circle action is free. To overcome them, we first construct a Riemannian metric on $Q$ which would project on the twist in case the latter existed. This metric will be a local product, and the integrality of the twist class will force periodicity for the  $\bbR$-action lifted in a Hamiltonian way to the universal cover of $Q$. Then treating the remaining issue amounts to understanding the geometry of the various group actions involved.
\end{rem}
Choose a principal connection form $\Theta$ in $Q$ with curvature form
\begin{equation}\label{curva}
-\di \Theta=\pi_Q^{\star}(z^{-1}\om_{-}).
\end{equation}
 Indicating with $T_Q $ the vector field 
tangent to the principal circle action on $Q$ we have a direct sum decomposition: 
\begin{equation} \label{split-sas}
\T Q=\span \{\ T_Q  \} \oplus \H,
\end{equation}
where $\H:=\ker(\Theta)$. Lifting the splitting $\T Z=\D_{+} \oplus \D_{-}$ to $\H$ allows decomposing 
\begin{equation} \label{spl1}
\H=\H_{+} \oplus \H_{-}.
\end{equation}
Consider the Riemannian metric $g_Q  $ on $Q$ given by 
\begin{equation*}
\begin{split}
g_Q  = (1+\pi_Q^{\star}x)\Theta^2+\pi_Q^{\star}g_{+}-\Theta \otimes \pi_Q^{\star}(K \cntrct g_{+})-\pi_Q^{\star}(K \cntrct g_{+}) \otimes \Theta
+\pi_Q^{\star}(z^{-1}g_{-})
\end{split}
\end{equation*}
where $x:=g(K,K)$, and $g_{\pm}=g|_{\D_\pm}$. 

\begin{rem}
	The Riemannian metric $g_Q$ is geometrically obtained as follows: 
	\begin{itemize}
		\item[(i)] Consider the $K$-invariant metric on $Z$ defined by $\overline{g}:=g_++z^{-1}g_-$;
		\item[(ii)] Twist $\overline{g}$ w.r.t. the principal connection $\Theta$ and Hamiltonian vector field $K$ to obtain $g_Q$.
	\end{itemize}
Observe that at this stage, the construction does not require periodic orbits of the flow.
\end{rem}

Whenever $U \in \T Z$ we indicate with $U^{\H}$ its horizontal lift to $\H$ w.r.t. the splitting  \eqref{split-sas}. Let 
$$\xi_Q:=T_Q+K^{\H}.$$
Purely algebraic considerations show that 
\begin{equation} \label{alg-Q}
g_Q(\xi_Q,\xi_Q)=1, \ g_Q(\xi_Q,\H)=0, \ g_Q(\H_{+},\H_{-})=0.
\end{equation}

In this section we mainly study the properties of the $g_Q$-orthogonal splitting 
\begin{equation} \label{splt3}
\T Q=\left(\span\{\xi_Q\} \oplus \H_{-}\right) \oplus \H_{+}.
\end{equation}
Eventually this will turn out to be parallel w.r.t. the Levi-Civita connection $\nabla^{g_Q}$ of the metric $g_Q$ (see Theorem \ref{loc-desc}).  

\begin{lemma} \label{symm}
We have 
\begin{itemize}
\item[(i)] $\L_{T_Q}g_Q=0$
\item[(ii)] $\L_{\xi_Q}g_Q=0$
\item[(iii)] $g_Q(\xi_Q,\cdot)=\Theta$.
\end{itemize}
\end{lemma}
\begin{proof}
(i) is due to $\L_{T_Q}\Theta=0$.\\
(ii) By (i) it suffices to check that $\L_{K^{\H}}g_Q=0$. Since $K\in \D_{+}$ is  Killing  it must preserve the $g$-orthogonal distributions 
$\D_{\pm}$, hence $\L_Kg_{+}=\L_Kg_{-}=0$. Clearly $\L_K z=\L_Kx=0$. In addition, from $\Theta(K^{\H})=0, K \cntrct \om_{-}=0$ and \eqref{curva} we get $\L_{K^{\H}}\Theta=0$ by means of Cartan's formula. All tensors involved in the definition of $g_Q$ are thus $K^{\H}$-invariant and the claim is proved.\\
(iii) follows algebraically from the definition of $g_Q$.
\end{proof}
\begin{propn} \label{spl2}
The distribution $\H_{+}$ is totally geodesic w.r.t. $g_Q$.
\end{propn} 
\begin{proof}
By \eqref{curva} we get 
\begin{equation} \label{br-1}
[U_1^{\H},U_2^{\H}]=[U_1,U_2]^{\H}+\pi_Q^{\star}(z^{-1}\om_-)(U_1,U_2)T_Q
\end{equation}
whenever $U_1,U_2$ are vector fields on $Z$. Pick $V, W \in \D_{+}$ and $X \in \H$. Using successively equations \eqref{alg-Q}, \eqref{br-1} and the integrability of $\D_+$ direct computations based on Koszul's formula give
$$g_Q(\nabla^{g_Q}_{V^{\H}}W^{\H},X^{\H})=\pi_Q^{\star}(g(\nabla^g_VW,X)).$$
Since $\D_+$ is totally geodesic, it follows that $\nabla^{g_Q}_{V^{\H}} W^{\H}\in \span\{\xi_Q\}\oplus\H_+$. Using successively that $g_Q(\xi_Q,\H_{+})=0$ together with the fact that 
$\xi_Q$ is a Killing vector field with $g_Q(\xi_Q,\cdot)=\Theta$ we get 
\begin{equation*}
2g_Q(\nabla^{g_Q}_{V^{\H}} W^{\H},\xi_Q)=-2g_Q(\nabla^{g_Q}_{V^{\H}}\xi_Q,W^{\H})=-\di\Theta(V^{\H},W^{\H})=0
\end{equation*}
where the last equality is granted  by \eqref{curva},  
and the proof is complete.
\end{proof}
\begin{lemma} \label{int-2} 
Let $\bar\Omega:=\pi_Q^{\star}\om_{+}+\pi_Q^{\star}(\di z) \wedge \Theta \in \Lambda^2Q$. The following hold:
\begin{itemize}
\item[(i)] $\ker (\bar\Omega :\T Q \to \T^{\star}Q)=\span \{\xi_Q\} \oplus \H_{-}$
\item[(ii)] $\di \bar\Omega=0$
\item[(iii)]
the distribution $\span \{\xi_Q \} \oplus \H_{-}$ is integrable.
\end{itemize}
\end{lemma}
\begin{proof}
(i) is proved by direct algebraic computation.\\
(ii) We have $\di\bar\Omega=\pi_Q^{\star}(\di\omega_{+})-\pi_Q^{\star}(\di z) \wedge \di \Theta=-\pi_Q^{\star}(\theta \wedge \om_{-})+
\pi_Q^{\star}(\di z \wedge (z^{-1}\omega_{-}))$ by using successively the structure equation 
\eqref{do+} and \eqref{curva}. The claim follows from $\theta=z^{-1}\di z$.\\
(iii) Let $U_1,U_2$ be sections of $\span \{\xi_Q \} \oplus \H_{-}$. Since $\bar\Omega$ is closed we have $\di \bar\Omega(U_1,U_2,\cdot)=0$. 
After expansion taking into account that $U_1 \cntrct \bar\Omega=U_2 \cntrct \bar\Omega=0$ (see (i)) this yields $\bar\Omega([U_1,U_2], \cdot)=0$ and the claim follows by using again (i). 
\end{proof}
\begin{propn} \label{tg-2}
The distribution $\span\{\xi_Q\} \oplus \H_{-}$ is totally geodesic w.r.t. $g_Q$.
\end{propn}
\begin{proof}
Pick sections $X_1,X_2 \in \D_{-}$ respectively $V \in \D_{+}$. Since $\span\{\xi_Q\} \oplus \H_{-}$ is integrable and $g_Q$-orthogonal to 
$\H_{+}$ using Koszul's formula leads to 
\begin{equation*}
\begin{split}
2g_Q(\nabla^{g_Q}_{X_1^{\H}}X_2^{\H},V^{\H})=&-V^{\H}g_Q(X_1^{\H},X_2^{\H})-g_Q([X_1^{\H},V^{\H}],X_2^{\H})-g_Q([X_2^{\H},V^{\H}],X_1^{\H}).
\end{split}
\end{equation*}
Since $[V^{\H},X_i^{\H}]=[V,X_i]^{\H},i=1,2$ by \eqref{br-1} the definition of $g_Q$ ensures that 
$$g_Q([X_1^{\H},V^{\H}],X_2^{\H})=\pi_Q^{\star}(z^{-1}g([X_1,V],X_2)) \ \mbox{and} \ g_Q([X_2^{\H},V^{\H}],X_1^{\H})=\pi_Q^{\star}(z^{-1}g([X_2,V],X_1).$$
Thus differentiating in $g_Q(X_1^{\H},X_2^{\H})=\pi_Q^{\star}(z^{-1}g(X_1,X_2))$ yields
\begin{equation*}
\begin{split}
2g_Q(\nabla^{g_Q}_{X_1^{\H}}X_2^{\H},V^{\H})&= \pi_Q^{\star}(z^{-2}\di z(V)g(X_1,X_2))\\
&-\pi_Q^{\star}(z^{-1}(Vg(X_1,X_2)+g([X_1,V],X_2)+g([X_2,V],X_1))).
\end{split}
\end{equation*}
But 
$Vg(X_1,X_2)+g([X_1,V],X_2)+g([X_2,V],X_1)=(\L_Vg)(X_1,X_2)=\theta(V)g(X_1,X_2)$
since the foliation is homothetic. Then $\theta=z^{-1}\di z$ yields $g_Q(\nabla^{g_Q}_{X_1^{\H}}X_2^{\H},V^{\H})=0$, in other words 
$\nabla^{g_Q}_{X_1^{\H}}X_2^{\H} \in \span \{\xi_Q\} \oplus \H_{+}$.

The integrability of $\span \{\xi_Q\}\oplus \H_{-}$ entails $g_Q(\nabla^{g_Q}_{\xi_Q}X^{\H},V^{\H})=g_Q(\nabla^{g_Q}_{X^{\H}}\xi_Q,V^{\H})$. Since 
$\xi_Q$ is a Killing vector field with $g_Q(\xi_Q, \cdot)=\Theta$ we find, after using \eqref{curva}, that $g_Q(\nabla^{g_Q}_{X^{\H}}\xi_Q,V^{\H})=\frac{1}{2}\di \Theta(X^{\H},V^{\H})=0$. Thus the vector fields $\nabla^{g_Q}_{\xi_Q}X^{\H}$ and $\nabla^{g_Q}_{X^{\H}}\xi_Q$ are both $g_Q$-orthogonal to $\H_{+}$.

Finally, since $\xi_Q$ is a unit Killing vector field with respect to $g_Q$ we find  $\nabla^{g_Q}_{\xi_Q}{\xi_Q}=0$ and the claim is proved by gathering the facts above.
\end{proof} 
Consider the universal cover $\widetilde{Q} \xrightarrow{\pi_{\tQ}} Q$ 
equipped with the lifted Riemannian metric $g_{\tQ}:=\pi_{\tQ}^{\star}g_Q$. 
\begin{thm}\label{loc-desc}  Let  $Z$ be compact. The Riemannian manifold $(\widetilde{Q},g_{\widetilde{Q}})$ splits as
	\begin{equation} \label{ma-split}
	(\widetilde{Q},g_{\widetilde{Q}})=(\widetilde{P},g_{\tilde P}) \times (\widetilde{N},g_{\tilde N})
	\end{equation}
where $(\widetilde{P},g_{\tilde P})$ and  $(\widetilde{N},g_{\tilde N})$ are simply connected and complete. In addition: 
\begin{enumerate}
	\item[(i)] $(\widetilde{P},g_{\widetilde P})$ is Sasakian, with Reeb vector field $\xi_{\widetilde{P}}$ given by the lift of $\xi_Q$ to $\widetilde Q$, contact form $\eta_{\widetilde{P}}={\pi}_{\tQ}^{\star}(\Theta)$ and contact distribution induced by $\H_-$ 
	\item[(ii)] $(\widetilde{N},g_{\widetilde N})$ is K\"ahler, with K\"ahler form $\omega_{\widetilde N}=\pi_{\tQ}^{\star}\bar\Omega$. The lift 
	$K_{\widetilde{N}}$ of $K^\H$ to $\widetilde Q$ is a holomorphic and Hamiltonian vector field on $\widetilde N$, with  momentum map $z_{\widetilde N}$ given by the lift of $z\circ \pi_Q$ to $\widetilde Q$.
\end{enumerate}
\end{thm}
\begin{proof}
By Propositions \ref{spl2} and \ref{tg-2}, the distributions appearing in the splitting \eqref{splt3}
are totally geodesic and $g_Q$-orthogonal, hence parallel w.r.t. $\nabla^{g_Q}$. Since $(\widetilde Q, g_{\widetilde Q})$ is simply connected and complete 
the claim in \eqref{ma-split} follows from the de Rham splitting theorem. For further use record that $\pi_Q^\star \left(\span\{\xi_{Q}\} \oplus \H_{-}\right) $
is tangent to $\widetilde P$ whilst  $\pi_Q^\star \H_{+}$ is tangent to $\widetilde N$. \\
(i) The form $\Theta$ vanishes on $\H_{+}$; whenever $U$ is a section of $\H_{+}$ we have $\L_U\Theta=0$ by using Cartan's formula, \eqref{curva} and 
the vanishing of $\om_{-}$ on $\D_{+}$. Therefore $\eta_{\widetilde{P}}:={\pi}_{\tQ}^{\star}\Theta$ is a $1$-form on $\widetilde{P}$, which does not depend on $\widetilde{N}$. Since $g_{\widetilde P}(\xi_{\widetilde P},\cdot)=\eta_{\widetilde{P}}$ it follows that $\xi_{\widetilde P}$ is a vector field 
on $\widetilde P$ which does not depend on $\widetilde N$. Then Lemma \ref{symm}, (ii) ensures that $\xi_{\widetilde P}$ is a unit length Killing vector field. The distribution $\mathcal{D}_{\tP}:=\ker(\eta_{\tP})$ is the lift to $\T\tP$ of $\H_{-}$. From \eqref{curva} we have $\di \eta_{\tP}=g_{\tP}(\phi \cdot, \cdot)$ where 
$\phi(\xi_{\tP})=0, \phi_{\vert \mathcal{D}_{\tP}}=J_{\tP}$ and $J_{\tP}$ is the lift to $\mathcal{D}_{\tP}$ of $J_{-}^{\H}:\H_{-} \to \H_{-}$. Here 
$J_{-}^{\H}$ is the horizontal lift of $J_{-}:\D_{-} \to \D_{-}$. Checking transverse integrability for $J_{\tP}$ amounts to 
$$[X^{\H},Y^{\H}]-[J_{-}^{\H}X^{\H},J_{-}^{\H}Y^{\H}]+J_{-}^{\H}([J_{-}^{\H}X^{\H},Y^{\H}]+[X^{\H},J_{-}^{\H}Y^{\H}]) =0$$
for all $X,Y \in \D_{-}$. This follows from \eqref{br-1}, $N^J(X,Y)=0$ on $Z$ and having $\om_{-}$ of type $(1,1)$ w.r.t. $J$.
\\
(ii) The form $\bar\Omega$ is closed and vanishes on $\span \{\xi_Q\} \oplus \H_{-}$ by Lemma \ref{int-2} (i) $\&$ (ii). Since $\L_{U}\bar\Omega=0$ whenever 
$U$ is a section of $\span \{\xi_{Q}\} \oplus \H_{-}$ it follows that $\omega_{\widetilde N}:={\pi}_{\tQ}^{\star}\bar\Omega$ induces a symplectic form 
on $\widetilde{N}$, which does not depend on $\widetilde{P}$. The full restriction of $\bar\Omega$ to $\H_{+}$ reads $\bar\Omega_{\vert \H_{+}}=(g_Q)_{\vert \H_{+}}(J_{+}^{\H}\cdot,\cdot)$ where $J_{+}^{\H}:\H_{+} \to \H_{+}$ is the horizontal lift of $J_{+}:\D_{+} \to \D_{+}$. Therefore the lift of $J_{+}^{\H}$ to $\widetilde{Q}$ induces an almost complex structure $J_{\widetilde{N}}$ on $\widetilde{N}$ which does not depend on $\widetilde{P}$ and satisfies $g_{\widetilde N}(J_{\widetilde N} \cdot, \cdot)=\om_{\widetilde N}$. Checking integrability for $J_{\widetilde N}$
amounts to 
$$[V^{\H},W^{\H}]-[J_{+}^{\H}V^{\H},J_{+}^{\H}W^{\H}]+J_{+}^{\H}([J_{+}^{\H}V^{\H},W^{\H}]+[V^{\H},J_{+}^{\H}W^{\H}]) =0$$
for all $V,W \in \D_{+}$. This follows from \eqref{br-1} and $N^J(V,W)=0$. We have showed that $(g_{\widetilde{N}},J_{\widetilde{N}})$ is K\"ahler.

Restricting $\L_{K^{\H}}g_{Q}=0$ to $\H_{+}$ shows that $K_{\widetilde{N}}$ is a Killing vector field w.r.t. $g_{\widetilde{N}}$. That $K_{\widetilde{N}}$ is Hamiltonian w.r.t. $\om_{\widetilde{N}}$ (and hence holomorphic w.r.t. $g_{\widetilde{N}}$) follows from 
$K^{\H} \cntrct \bar\Omega=\pi_Q^{\star}(K \cntrct \om_{+})+K^{\H} \cntrct (\pi_Q^{\star}(\di z)  \wedge \Theta)=\pi_Q^{\star}\di z=\di (\pi_Q^{\star}z)$. Here we have taken into account the definition of $\bar\Omega$ (see Lemma \ref{int-2}) as well as $\di z(K)=\Theta(K^{\H})=0$.
\end{proof}

\begin{lemma} \label{iso-s}
	$\Iso(\widetilde Q, g_{\widetilde Q})\simeq \Iso(\widetilde P,g_{\widetilde P})\times \Iso (\widetilde N, g_{\widetilde N})$. 
\end{lemma} 

\begin{proof}
Since $\widetilde P$ is Sasakian,  its metric  $g_{\widetilde P}$ is irreducible. Indeed, the curvature tensor of a Sasakian metric satisfies $R(X,Y)\xi=-\xi\cntrct (X^\flat\wedge Y^\flat)$, where $\xi$ denotes the Reeb field. If, by absurd, the tangent bundle were a direct sum $\T_1\oplus \T_2$ of parallel subbundles, then $0=R(X_1,X_2)\xi=-\xi\cntrct(X_1^\flat\wedge X_2^\flat)$, and hence $X_1^\flat(\xi)=0=X_2^\flat(\xi)$, thus $\xi$ is orthogonal to both $\T_1$ and $\T_2$, contradiction.
	
Now let $f\in \Iso(\widetilde Q, g_{\widetilde Q})$ and restrict its differential $f_\star$ to the subbundle $\T\tilde P$. The distribution $f_\star (\T\widetilde P)$ is $\nabla^{g_{\widetilde Q}}$ - parallel. Since $g_{\widetilde P}$ is irreducible, the only $\nabla ^{g_{\tQ}}$ - parallel distributions on $\widetilde Q$ are of the form $\T\widetilde P\oplus E$, $F$, with $E,F\subseteq \T\widetilde N$, parallel w.r.t. $\nabla^{ g_{\widetilde N}}$ (here we allow $E=\{0\}$). We exclude both cases as follows. Let $\widetilde N=\bbR^{2k}\times \widetilde N_1\times\cdots \times \widetilde N_l$ be the de Rham decomposition of $\widetilde N$; here $\widetilde N_i$ are the irreducible K\"ahler factors.
	
	If $f_\star (\T\widetilde P)=F$, then, as $g_{\widetilde P}$ is irreducible, $F$ should be equal with one of the $\T\widetilde N_i$, which are of even dimension, contradiction.
	
	If $f_\star (\T\widetilde P)=\T\widetilde P\oplus E$, a dimension argument shows that $E=\{0\}$. 
	
	Therefore, $f_\star (\T\widetilde P)=\T\widetilde P$. Since $\T\widetilde P\perp \T\widetilde N$ in $\T\widetilde Q$, we conclude that also $f_\star (\T\widetilde N)=\T\widetilde N$ and the proof is complete.
\end{proof}
Note that only the completeness of $g_{\widetilde Q}$ has been used above, as $\widetilde Q$ is not necessarily compact.

\subsection{Geometry of the group actions} \label{actiunii}
  Denote $\Gamma:=\pi_1(Q)$ and its action on $\widetilde{Q}$ by 
$(\tilde q,\gamma) \mapsto \tilde q\gamma$. The action of $\bbS$ lifts to an $\bbR$-action $({R}^{\tQ}_{t})_{t \in \bbR}$ on $\widetilde Q$ (see \cite{Bre}, Theorem 9.1) commuting with the action of $\Gamma$. Explicitly
\begin{equation} \label{proj-aa}
{\pi}_{\tQ} \circ {R}^{\tQ}_{t}=R^Q_{\exp(2\pi i t)} \circ {\pi}_{\tQ}, \quad 
{R}^{\tQ}_t(\tilde q\gamma)=( R^{\tQ}_t\tilde q)\gamma.
\end{equation}
  Indicate with $A:\pi_1(\bbS)=\mathbb{Z} \to \Gamma$ the homomorphism induced in homotopy by any fibre inclusion $\bbS\hookrightarrow Q$, as well as 
$$K_A:=\ker(A), \ G_A:=\bbR/K_A.$$
\begin{lemma}\label{actiuni} The following hold:
	\begin{enumerate}
	\item[(i)] ${R}^{\tQ}_1(\tilde q)=\tilde qa$ with $a$ in the center $Z(\Gamma)$ of $\Gamma$.
	\item[(ii)] $A(n)=a^n$.
	\item[(iii)] $
	K_A=\begin{cases}
	\{0 \}  \ \text{if a has infinite order}, \\[.1in]
	p\mathbb{Z}  \ \ \text{if} \ a^p=1,\  p\neq 0.
	\end{cases}
	$
	\item[(iv)] The induced $G_A$-action on $\widetilde{Q}$ is free and proper. 
	\end{enumerate}
\end{lemma}

\begin{proof}
	(i) Let $t=1$ in \eqref{proj-aa}, then  ${R}^{\tQ}_1(\tilde q)=\tilde qa$ with $a\in\Gamma$, since ${R}^{\tQ}_1$ is a gauge transformation of $\Gamma\rightarrow\tilde Q\longrightarrow Q$. To see $a$ is central, use the second equation in \eqref{proj-aa} and the fact that $\Gamma$ acts freely on $\tilde Q$.\\
(ii) Fix $q\in Q$, and $\tilde q \in \tilde Q$ in the fibre above $q$. Let $c(t)=R^Q_{\exp(2\pi i t)}q$, $t\in[0,1]$, be a loop at $q$. Let $\tilde c$ be its lift through $\tilde q$. By the very definition of the action of $\Gamma=\pi_1(Q)$ on the fibre over $q$, we have $\tilde c(1)={R}^{\tQ}_1(\tilde q)$, which by (i) equals $\tilde q a$, thus  $A(1)=a$ and the claim is proved.\\
(iii) follows directly from (ii).\\
(iv) follows from \cite[Proposition 1.7]{mo}. Note that if $\mathrm{ord}(a)=p$, then $G_A$ is topologically a circle and hence the action is automatically proper.
\end{proof}

Since by construction  $ R^{\tQ}_t\in \Iso(\tilde Q,g_{\tilde Q})$ we can split the $\bbR$-action on $\tilde Q$ according to the de Rham splitting of the latter.
\begin{propn} \label{split-Ra}
The following hold
\begin{itemize}
\item[(i)]The lifted $\bbR$-action on $\widetilde{Q}$ is a product 
${R}^{\tQ}_t={R}^{\tP}_t \times {R}^{\tN}_t$ 
where ${R}^{\tP}_t$, respectively ${R}^{\tN}_t$, are isometric $\bbR$-actions on $(\widetilde{P}, g_{\tP})$ 
respectively $(\widetilde{N}, g_{\tN})$. Moreover, ${R}^{\tN}_t$ is Hamiltonian 
\item[(ii)] $({R}^{\tP}_t)_{t \in \bbR}$ induces a free and proper $G_A$-action on $\widetilde{P}$ tangent to the Reeb field $\xi_{\tP}$ 
\item[(iii)] 	There exists a principal bundle $G_A\hookrightarrow \widetilde P \xrightarrow{\pi_{\tP}} \widetilde M$, with $\widetilde M$ simply connected. This bundle is (differentiably) trivial when $\mathrm{ord}(a)=\infty$
\item[(iv)] $\widetilde{M}$ has a K\"ahler structure $(g_{\tM},J_{\tM})$ with K\"ahler form $\om_{\tM}$ such that $-\di\eta_{\widetilde P}=\pi_{\tP}^{\star} 
\om_{\tM}$.
\end{itemize} 
\end{propn}
\begin{proof}
(i) follows from Lemma \ref{iso-s} since ${R}^{\tQ}_t\in \Iso(\tilde Q,g_{\tilde Q})$. \\
(ii) Let $(t,u) \in \bbR \times \widetilde{P}$ be such that $R^{\tP}_t(u)=u$. The vector field $K$ has zeroes since it is Hamiltonian on the compact manifold $Z$. 
Thus both $K^{\H}$ and its lift to $\widetilde{Q}$ will  have zeroes. It follows that $R^{\tN}_t$ has fixed points; if $\tilde n_0$ is such a point then $R^{\tN}_t(u,\tilde n_0)=(u,\tilde n_0)$. By \eqref{proj-aa} and using that $R^Q_{\lambda}$ is free on 
$Q$ it follows that $\exp(2\pi i t)=1$, that is $t \in \mathbb{Z}$. From (i) in Lemma \ref{actiuni}, 
$R^{\tQ}_t\tilde q=\tilde qa^t, \tilde q \in \widetilde{Q}$, it follows that $a^t=1$ hence $t \in K_A$. Thus $(R^{\tP}_t)_{t \in \bbR}$ induces a free $\bbR \slash K_A$-action on $\widetilde{P}$. To show the properness of the latter action  we only need consider the instance when $K_A=\{0\}$ since otherwise 
$G_A$ is compact. Let $K \subseteq \widetilde{P}$ be compact. Then 
\begin{equation*}
\{t \in \bbR: R^{\tP}_t(K) \cap K \neq \emptyset \}=\{t \in \bbR: R^{\tQ}_t(K \times \{\tilde n_0\}) \cap (K \times \{\tilde n_0\}) \neq \emptyset \}
\end{equation*}
has compact closure since $(R^{\tQ}_t)_{t \in \bbR}$ is proper on $\widetilde{Q}$ (here $\tilde n_0$ is a fixed point of the action $R^{\tN}_t$). By construction, the lift of $T_Q$ to $T_{\tQ}$ equals $\xi_{\tP}-K_{\tN}$ showing that the action on $\tP$ is tangent to the Reeb field.\\
(iii) The existence of the $G_A$-bundle follows from (ii). If $\mathrm{ord}(a)=\infty$, then $G_A=\bbR$, and the bundle is differentiably trivial (see e.g. \cite[Theorem 5.7, p. 58]{kn}). The simple connectedness of $\widetilde M$ follows from the long exact sequence in homotopy, since $\widetilde P$ is simply connected.\\
(iv) follows from the fact that $(\tP,g_{\tP}, \eta_{\tP})$ is Sasakian and (iii).
\end{proof}

By (iv) above  we obtain a smooth manifold 
\begin{equation*}
\tZ:=\widetilde{Q} \slash G_A
\end{equation*}
which is simply connected since $G_A$ is connected. We now show that $\tilde Z$ is precisely the universal cover of $Z$ (thus motivating the notation).

Recall that the action of $G_A$ on $\widetilde{Q}$ is the lift of the principal $\bbS$-action on $Q$, thus using 
\eqref{proj-aa} shows that the map $\pi_Q \circ {\pi}_{\tQ}$ induces a smooth map $\pi_{\tZ}:\tZ \to Z$. Now, from the commutation of the actions of $G_A$ and 
$\Gamma$, we obtain an action $\tZ \times \Gamma \to \tZ$. By Lemma \ref{actiuni}, (i) the central element $a \in \Gamma$ acts on $\widetilde Q$ as $R^{\tQ}_1$, thus trivially on $\tZ$. We obtain 
an action 
\begin{equation} \label{cov-ac}
\tZ \times \Gamma_a \to \tZ
\end{equation}
where $\Gamma_a=\Gamma \slash \langle a \rangle$. Note that $\pi_1(Z)=\Gamma_a$, as entailed by the construction of $a$. Record that the following diagram is commutative
\begin{equation}\label{com-cov}
\begin{tikzcd}
&\widetilde{Q} \arrow[d,swap,"p_{\widetilde{Q}}"] \arrow[r,"\pi_{\widetilde{Q}}"] &Q \arrow[d,"\pi_Q"]\\
&\widetilde{Z} \arrow[r,"\pi_{\widetilde{Z}}"]&Z\\
\end{tikzcd}
\end{equation}
where $p_{\tQ}: \widetilde{Q} \to \tZ$ is the canonical projection. By \cite[Proposition 1.9]{mo}, 
$$\Gamma_a \hookrightarrow \tZ \xrightarrow{\ \pi_{\tZ}\ } Z$$  
is a covering space, therefore the universal cover of $Z$. 

This relates to the de Rham splitting of $(\widetilde{Q},g_{\tilde{Q}})$ as follows. From Theorem \ref{loc-desc} 
\begin{equation*}
\tZ=\tP \times_{G_A}\tN
\end{equation*}
where the free $G_A$-action on $\tP \times \tN$ is induced by $t \in \bbR \mapsto R^{\tP}_t \times R^{\tN}_{t}$; moreover having $\tP$ polarising the 
K\"ahler manifold $(\tM,g_{\tM},J_{\tM})$ (see Proposition \ref{split-Ra}, (iv)), $(\widetilde N,g_{\widetilde N},J_{\widetilde N})$-K\"ahler and $G_A \subseteq \Iso(\widetilde N,g_{\widetilde N})$ Hamiltonian (see Proposition \ref{split-Ra} (i)) make it possible to endow $\tZ$ with the K\"ahler structure $(\tilde{g},\tilde{J})$ coming from the Weinstein construction.
\begin{lemma} \label{pZ} We have: 
\begin{itemize}
\item[(i)] $\pi_{\tZ}$ is isometric and holomorphic, $\pi_{\tZ}^{\star}g=\tilde{g}, (\di \pi_{\tZ})\tilde{J}=J(\di \pi_{\tZ})$.
\item[(ii)] $\Gamma_a \subseteq \Aut(\tZ,\tilde{J})$.
\end{itemize}
\end{lemma}
\begin{proof}
(i) Let $\tilde\H$ be the lift of $\H$ to $\tQ$. Since $\pi_{\tilde Q}$ is a local isometry, 
on $\tilde\H$ we have $\pi_{\tilde Q}^\star {g_Q}_{\vert \H}={g_{\tilde Q}}_{\vert\tilde \H}$. But  ${g_Q}_{\vert \H}=g_++z^{-1}g_-$, and ${g_{\tilde Q}}_{\vert\tilde \H}=\tilde g_++\tilde z^{-1}\tilde g_-$, where $\tilde z:=z\circ\pi_{Q}$. As $\di \pi_{\tilde Q}(\tilde \H_{\pm})=\H_{\pm}$, we find $\pi_{\tilde Q}^\star g_{\pm}=\tilde g_{\pm}$ and hence $\pi_{\tZ}^{\star}g=\tilde{g}$.

As for the holomorphy of $\pi_{\tZ}$, recall from Theorem \ref{loc-desc} that $\tilde \D_\pm=\di p_{\tQ}(\tilde\H_{\pm})$ and $\D_{\pm}=\di \pi_Q(\H_\pm)$. Since $\tilde J$ comes from the Weinstein construction, the projections $\pi_{\tQ}$ and $\pi_Q$ are transversally holomorphic, i.e. they are holomorphic on $\tilde\H$ and $\H$. The commutativity of diagram \eqref{com-cov} then proves the claim.

(ii) Follows form the holomorphy of $\pi_{\tZ}:(\tZ,\tilde{J}) \to (Z,J)$ by using that $\pi_{\tZ}$ is a local diffeomorphism, invariant under the action of $\Gamma_a$. 
\end{proof}

Let 
\begin{equation*}
\Aut(\widetilde P,g_{\widetilde P},\eta_{\tP}):=\{f \in \Iso(\widetilde P,g_{\widetilde P}):f^{\star} \eta_{\widetilde P}=\eta_{\widetilde P}\}
\end{equation*}
be the group of contact automorphisms of the Sasakian manifold $(\widetilde P,g_{\widetilde P}, \eta_{\tP})$. The proof of the following 
Lemma is straightforward.
\begin{lemma} \label{a-inf1}
Assume that $(\tM,g_{\tM},J_{\tM})$ is an exact K\"ahler manifold, $-\di \alpha_{\tM}=\om_{\tM}$ which moreover is simply connected. Denote 
$G_{\tM}:=\Aut(\tM,g_{\tM},J_{\tM})$.
\begin{itemize}
\item[(i)] we have a well defined map $a:G_{\tM} \to C^{\infty}(\tM,\bbR), f \mapsto a_f$, uniquely determined from 
$$ \di a_f=\alpha_{\tM}-f^{\star}\alpha_{\tM}\ \mathrm{and} \ a_{1}=0
$$
\item[(ii)] consider the Sasakian manifold $(\tP=\tM \times \bbR,\  \eta_{\tP}=\di t+\alpha_{\tM},\ g_{\tP}=\eta_{\tP} \otimes \eta_{\tP}+g_{\tM})$. The map $\varepsilon : G_{\tM} \to \Aut(\widetilde P,g_{\widetilde P},\eta_{\tP})$ is a well defined 
group isomorphism.
\begin{equation*}
\varepsilon(f)(m,t)=(f(m),a_f(m)+t).
\end{equation*}
\end{itemize}
\end{lemma}
In \cite[p. 368]{arn} the contact manifold  $(\tM \times \bbR, \di t+\alpha_{\tM})$ is called contactification. Lemma \ref{a-inf1} makes explicit the action of $G_{\tM}$ on $\tP$, by automorphisms of the Sasakian structure. We prove now that the action of $\Gamma=\pi_1(Q)$ on $\widetilde Q$ splits according to \eqref{ma-split}. 
\begin{propn}\label{prop_split}
	W.r.t. the decomposition \eqref{ma-split} the action of $\, \Gamma$ splits as a product action 
	
	\begin{equation}\label{menahem1}
	\Gamma \subseteq \Aut(\widetilde P,g_{\widetilde P}, \eta_{\tP}) \times C_{\Iso (\widetilde N, g_{\widetilde N})}(G_A),
	\end{equation}
	where $C_H(S)$ denotes the centraliser of a set $S$ in the group $H$. Therefore, $\Gamma$ induces a product action on $\widetilde M\times \widetilde N$  
	\begin{equation}\label{menahem2}
	\Gamma\subseteq \Aut(\widetilde M, g_{\widetilde M}, J_{\widetilde M})\times C_{\Iso (\tilde N, g_{\widetilde N})}(G_A).
	\end{equation}	
\end{propn}

\begin{proof}
	The action of $\Gamma$ and $ R^{\tQ}_t$ on $\tilde Q$ commute. Moreover, both actions split with respect to  \eqref{ma-split}, and hence the actions of $\Gamma$ and $R^{\tQ}_t$ on each factor must commute. Therefore on the factor $\tilde N$, $\Gamma$ acts as a subgroup of the centraliser $C_{\Iso (\tilde N, g_{\tilde N})}(G_A)$. Similarly,  $\Gamma$ acts on $\tilde P$ as a subgroup of $\Iso(\tilde P, g_{\tilde P})$ commuting with the Reeb flow $(R^{\tP}_t)$. Since an isometry which preserves the Reeb field of a Sasakian structure preserves the contact form and the transverse complex structure too, \eqref{menahem1} is proven. 
	
	The Reeb flow $(R^{\tP}_t)$ is regular by Proposition \ref{split-Ra}  (iii), thus elements in $\Aut(\widetilde P, g_{\widetilde P}, \eta_{\tP})$ project onto elements in $\Aut(\widetilde M, g_{\widetilde M}, J_{\widetilde M})$, hence \eqref{menahem2}.
\end{proof}

Based on the above, we can prove the main global result of this section.

\begin{thm}\label{global-carac}
Let $Z$ be a compact K\"ahler manifold endowed with a totally geodesic, holomorphic and homothetic foliation. Assume the cohomology class $c$ in  \eqref{co-class} is integral. Then the universal cover $\tZ$ with the pulled-back K\"ahler structure is either
\begin{itemize}
\item[(i)] obtained from the Weinstein construction

\noindent or satisfies
 
\item[(ii)] $\pi_{\tZ}^\star (c)=0$ and $\tZ=\tM\times \tN$ where $\tM, \tN$ are simply connected complete K\"ahler manifolds with $\om_{\tM}$ exact; the K\"ahler structure on $\tZ$ is given by the local Weinstein construction in Section \ref{loc_wein_def}. The fundamental group $\pi_1(Z)$ acts on $\tZ$ as in \eqref{menahem2}.
\end{itemize}
In particular, if $Z$ is simply connected, then it is obtained by the Weinstein construction.
\end{thm}

\begin{proof}
(i) corresponds to $\mathrm{ord}(a)=p$, in which case  $G_A$ is topologically a circle, see Lemma \ref{actiuni} (iii), and the result is proved by the considerations above.\\
 (ii) If $\pi_{\tZ}^\star(c)=0\in H_{dR}(\tZ)$, then 	$\mathrm{ord}(a)=\infty$.  In this case $\tP=\tM\times\bbR$ as principal $\bbR$-bundles, see Proposition \ref{split-Ra} (iii).
The $\bbR$-action on $\tP\times \tN$ then reads
$$(\tilde m, t,\tilde n)s=(\tilde m, t+s, R_s^{\tN}\tilde n),$$
and hence the map $F: (\tM\times\bbR)\times\tN\longrightarrow \tM\times \tN$ given by $F(\tilde m, t,\tilde n)=(\tilde m, R_{-t}^{\tN}\tilde n)$ is invariant and 
 provides an identification $f:\tZ\longrightarrow \tM\times\tN$.

 The Sasakian structure on $\tP$ is given by the contact form $\di t+\alpha_{\tM}$, where $\alpha_{\tM}\in\Lambda^1\tM$ does not depend on $\bbR$ and satisfies $-\di \alpha_{\tM}=\om_{\tM}$. The contact distribution is then $$\tilde\H=\{X-\alpha_{\tM}(X)\partial_t\ ;\ X\in\T \tM\}.$$
  Therefore, on $\tM\times\tN$ we have: 
\begin{equation*}
\begin{split}
\tilde\D_-&=\di F(\tilde\H)=\{X+\alpha_{\tM}(X)K_{\tN} ; X\in\T \tM\},\\
\tilde\D_+&=\di F(\T \tN)=\T \tN.
\end{split}
\end{equation*} 
The K\"ahler form and the complex structure on $\tM\times\tN$ are then given by \eqref{local_wein}. 
The action of $\pi_1(Z)$ on $\tM \times \tN$ is described in Proposition \ref{prop_split}. Note that the metric on $\tilde M$ is complete because the metric on $\tilde P$ is complete and $\tilde P\to\tilde M$ is a Riemannian submersion.
\end{proof}

\begin{rem}
	However, not even in case (ii) of Theorem \ref{global-carac} $Z$ itself is necessarily constructible from the Weinstein Ansatz. The obstruction is that from \eqref{menahem2} one cannot deduce that $\Gamma$ splits as a product $\Gamma_{\tM}\times\Gamma_{\tN}$ with $\Gamma_{\tM}$ respectively $\Gamma_{\tN}$ acting trivially on $\tN$ respectively $\tM$. We provide below an example for this situation.
\end{rem}

\begin{exo}
	Examples of group actions as in Proposition \ref{prop_split} can be obtained  from stable bundles over Riemann surfaces as follows. Let $\Sigma$ be  a Riemann surface of genus at least 2 and fundamental group $\Gamma$. Let $E\rightarrow \Sigma$ be a stable vector bundle over $\Sigma$ of rank $r+1$. Then the  projectivisation $\mathbb{P}(E)$ can be identified (see \cite[Theorem 2.7]{kob}) with $(\tilde \Sigma\times \mathbb{P}^r)/\Gamma$, modulo a representation $\rho$ of $\Gamma$ in $\mathrm{PU}(r+1)$ (here $\tilde \Sigma$ is the universal cover of $\Sigma$). Explicitly, this action reads: $ (\tilde s, [z])\gamma=(\tilde s\gamma, \rho(\gamma)^{-1}[z])$. 
\end{exo}

\begin{rem} \label{i-vs-h}
\begin{itemize}
\item[(i)] In case $C_{\Iso (\tN, g_{\tN})}(G_A) \subseteq \Aut(\tN,J_{\widetilde{N}})$ the quotient $(\widetilde M\times\widetilde N)/\Gamma$ is a locally product K\"ahler orbifold which, in general, is not globally a product of two K\"ahler orbifolds.
\item[(ii)] In case the action of $C_{\Iso (\widetilde N, g_{\widetilde N})}(G_A)$ on $\tN$ has fixed points an argument similar to that in Proposition \ref{split-Ra} (ii) shows that $\Gamma$ acts freely on $\widetilde{P}$. 
\item[(iii)] Since $\tM$ is simply connected it is well known that the projection map 
$$\Aut(\tP,g_{\tP},\eta_{\tP}) \to \Aut(\tM,g_{\tM},J_{\tM})$$ is 
surjective.
\end{itemize}
\end{rem}

\hfill

\noindent{\bf Acknowledgments.} It is a pleasure to thank Florin Belgun, Andrei Moroianu and Christina W. Tønnesen-Friedman for useful 
exchanges during the preparation of this paper.

\hfill

\end{document}